\documentclass[letterpaper]{amsart}
\usepackage{amsfonts,amssymb,amscd,amsmath,enumerate,url,verbatim}
\usepackage[colorlinks]{hyperref}

\input xypic
\input xy
\xyoption{all}
\usepackage[right=1.5in,left=1.5in,top=1.2in,bottom=1.2in]{geometry}

\newcommand{\agr}[2][{}]{{{#2}^{\sf{g}}_{#1}}}
\newcommand{\wt}[1]{\tilde {#1}}
\def\codim{\operatorname{codim}}

\def\ch{\operatorname{char}}
\def\ov{\overline}
\def\HH{\operatorname{H}}
\def\Ext{\operatorname{Ext}}
\def\ann{\operatorname{ann}}
\def\reg{\operatorname{reg}}
\def\init{\operatorname{in}}
\def\rank{\operatorname{rank}}

\def\m{\mathfrak m}
\def\a{\mathfrak a}
\def\b{\mathfrak b}
\def\n{\mathfrak p}

\def\q{\mathfrak q}

\def\ann{\operatorname{ann}}
\def\socle{\operatorname{socle}}

\def\kk{\pmb{k}}
\def\Tor{\operatorname{To r}}
\def\Ker{\operatorname{Ker}}
\def\Po{\operatorname{P}}

\def\HH{\operatorname{H}}

\newtheorem{theorem}{Theorem}[section]
\newtheorem{lemma}[theorem]{Lemma}
\newtheorem{corollary}[theorem]{Corollary}
\newtheorem{proposition}[theorem]{Proposition}
\newtheorem{proposition-no-advance}[equation]{Proposition}

\newtheorem{claim-no-advance}[equation]{Claim}

\newtheorem{subtheorem}{Theorem}[theorem]
\newtheorem*{Theorem}{Theorem}
\newtheorem*{Main Theorem}{Main Theorem}

\newtheorem*{Lemma}{Lemma}

\newtheorem{quick consequences}[theorem]{Quick Consequences}

\theoremstyle{definition}
\newtheorem*{acknowledgement}{Acknowledgement}
\newtheorem{facts and definitions}[theorem]{Facts and Definitions}

\newtheorem{remark}[theorem]{Remark}

\newtheorem{remark-no-advance}[equation]{Remark}

\newtheorem{careful calculation}[theorem]{Careful Calculation}

\newtheorem{present summary}[theorem]{Present Summary}

\newtheorem{further reductions}[theorem]{Further Reductions}

\newtheorem{chunk}[theorem]{}
\newtheorem{subchunk}[subtheorem]{}

\newtheorem{circle the wagons}[theorem]{Circle the wagons}

\theoremstyle{remark}
\newenvironment{bfchunk}{\begin{chunk}\textit}{\end{chunk}}

\numberwithin{equation}{theorem} \numberwithin{table}{theorem}

\newtheorem*{Case3.1}{Case 3.1}
\newtheorem*{Case3.2}{Case 3.2}
\newtheorem*{Case3.2.1}{Case 3.2.1}
\newtheorem*{Case3.2.2}{Case 3.2.2}
\newtheorem*{Case3.2.2(a)}{Case 3.2.2(a)}
\newtheorem*{Case2.2.2(a)}{Case 2.2.2(a)}
\newtheorem*{Case3.2.2(b)}{Case 3.2.2(b)}
\newtheorem*{Case1.1}{Case 1.1}
\newtheorem*{Case1.2}{Case 1.2}
\newtheorem*{Case1.1.1}{Case 1.1.1}
\newtheorem*{Case1.1.2}{Case 1.1.2}
\newtheorem*{Case1.1.2(a)}{Case 1.1.2(a)}
\newtheorem*{Case1.1.2(b)}{Case 1.1.2(b)}
\newtheorem*{Case2.1}{Case 2.1}
\newtheorem*{Case2.1.1}{Case 2.1.1}
\newtheorem*{Case2.1.2}{Case 2.1.2}
\newtheorem*{Case2.2.1}{Case 2.2.1}
\newtheorem*{Case2.2.2}{Case 2.2.2}
\newtheorem*{Case2.2.3}{Case 2.2.3}
\newtheorem*{Case2.2.4}{Case 2.2.4}
\newtheorem*{Case2.2.2(b)}{Case 2.2.2(b)}
\newtheorem*{Case2.2.1(a)}{Case 2.2.1(a)}
\newtheorem*{Case2.2.1(b)}{Case 2.2.1(b)}
\newtheorem*{Case2.1.3}{Case 2.1.3}
\newtheorem*{Case2.1.3(a)}{Case 2.1.3(a)}
\newtheorem*{Case2.1.3(b)}{Case 2.1.3(b)}
\newtheorem*{Case2.2}{Case 2.2}
\newtheorem*{Case1}{Case 1}
\newtheorem*{Case2}{Case 2}

\newtheorem*{Case(1)}{Case {\rm (1)}}
\newtheorem*{Case(2)}{Case {\rm (2)}}
\newtheorem*{Case(3)}{Case {\rm (3)}}
\newtheorem*{Case(4)}{Case {\rm (4)}}
\newtheorem*{Case(5)}{Case {\rm (5)}}
\newtheorem*{Case(6)}{Case {\rm (6)}}
\newtheorem*{Case(7)}{Case {\rm (7)}}
\newtheorem*{Case(8)}{Case {\rm (8)}}


\begin{document}
\title[]{The  absolutely Koszul and Backelin-Roos properties\\ for spaces of quadrics of small codimension}

\author[R.~A.~Maleki]{Rasoul Ahangari Maleki}
\address{Rasoul Ahangari Maleki\\School of Mathematics\\Institute for Research in Fundamental Sciences (IPM)\\ P. O. Box: 19395--5746\\Tehran\\Iran}
     \email{rahangari@ipm.ir}

\author[L.~M.~\c{S}ega]{Liana M.~\c{S}ega}
\address{Liana M.~\c{S}ega\\ Department of Mathematics and Statistics\\
   University of Missouri\\ \linebreak Kansas City\\ MO 64110\\ U.S.A.}
     \email{segal@umkc.edu}

\subjclass[2000]{13D02} \keywords{}
\thanks{This  work was supported in part by a grant from the Simons Foundation (\# 354594, Liana \c Sega) and a grant from IPM  (No.
96130111, Rasoul Ahangari Maleki)}

\begin{abstract}
Let $\kk$ be a field,  $R$ a standard graded quadratic
$\kk$-algebra with $\dim_{\kk}R_2\le 3$,  and let $\ov\kk$ denote an algebraic closure of $\kk$.  We construct a  graded
surjective Golod homomorphism $\varphi \colon  P\to R\otimes_{\kk}\ov{\kk}$ such that $P$
is a complete intersection of codimension at most $3$. Furthermore,
we show that $R$ is absolutely Koszul (that is, every finitely
generated $R$-module has finite linearity defect)  if and only if
$R$ is Koszul if and only if $R$ is not a trivial fiber extension of
a  standard graded $\kk$-algebra with Hilbert series $(1+2t-2t^3)(1-t)^{-1}$.  In particular, we recover earlier results  on the
Koszul property  of  Backelin \cite{Backelin}, Conca \cite{C2} and
D'Al\`i \cite{Alessio}.
\end{abstract}

\maketitle \tableofcontents

\section*{Introduction}
Let $\kk$ be a field and  let $Q$ denote a polynomial ring in
finitely many variables of degree $1$ over $\kk$.  Let $R=\bigoplus
_{i\geqslant 0}R_i$ denote a standard graded $\kk$-algebra such that
$R=Q/I$, where $I$ is a homogeneous ideal contained in $Q_{\geqslant
2}$.  The  algebra $R$ is said to be {\it quadratic} if $I$ is
generated by homogeneous polynomials of degree $2$. In this paper we
study several homological properties of quadratic
algebras satisfying $\dim_{\kk}R_2\le 3$.

A finitely generated graded $R$-module is said to be {\it linear} if
it admits a graded free resolution with maps described by matrices
whose non-zero entries are linear forms. The  algebra $R$ is said to be {\it Koszul}  if
$\kk$ is a linear $R$-module. It is said to be {\it absolutely
Koszul} if every finitely generated $R$-module has finite linearity
defect (equivalently, it has a syzygy whose associated graded module
with respect to the maximal homogeneous ideal of $R$ is linear); see
Herzog and Iyengar \cite{HI} and R\"omer and Iyengar \cite{IR} for more information on these concepts.
Finally, we say that $R$ has the {\it Backelin-Roos} property if
there exists a surjective Golod homomorphism $\varphi\colon P\to R$
of standard graded $\kk$-algebras, with  $P$ a complete
intersection.

The concepts recalled above are related. If $R$ is Koszul, then it is
quadratic. If $R$ is absolutely Koszul, then $R$ is Koszul and has
the property that the Poincar\'e series of all finitely generated
graded $R$-modules are rational, sharing a common denominator, see \cite[Proposition 1.13, Proposition 1.15]{HI}.  Examples of
\,Roos \cite{Roos-good} show that not  all  Koszul algebras have the
latter property. Absolute Koszulness is thus stronger
than Koszulness and it gives information on minimal graded free resolutions of modules, rather than just the resolution of $\kk$.  A method for showing that an algebra is absolutely Koszul is provided in  \cite[Theorem 5.9]{HI} through the following implication: If $R$ has the Backelin-Roos property and $R$
is Koszul, then $R$ is absolutely Koszul. This method
was used recently by  Conca et al.\,\cite{veronese}  to show that
some Veronese subrings of polynomial rings are absolutely Koszul.

We say that $R$ is a {\it trivial fiber extension} of a standard
graded $\kk$-algebra  $R'$ if there exist linear forms $u_1, \dots,
u_n$ in $R$ such that $u_iR_1=0$ for all $i$ and $R'=R/(u_1, \dots,
u_n)$.

\begin{comment} We say that $R$ is {\it exceptional}  if it satisfies the
following properties: (a) $R$ is quadratic; (b)
$\dim_{\kk}R_2=3=\dim_{\kk}R_1$; (c) $R$ is not artinian; (d) $R$ is
not Koszul.
\end{comment}

\begin{Main Theorem}
Let $R$ be a standard graded quadratic $\kk$-algebra with
$\dim_{\kk}R_2\le 3$. Let $\ov \kk$ denote an algebraic closure of $\kk$.

There  exists then a surjective Golod homomorphim of graded $\ov\kk$-algebras $P\to R\otimes_{\kk}\ov {\kk}$ 
such that $P$ is a complete intersection of codimension at most $3$. Furthermore,  the following statements are equivalent:
\begin{enumerate}[\quad\rm(1)]
\item $R$ is absolutely Koszul;
\item $R$ is Koszul;
\item $R$ is not trivial fiber extension of a standard graded $\kk$-algebra with Hilbert series equal to $(1+2t-2t^3)(1-t)^{-1}$.
\end{enumerate}
\end{Main Theorem}

To prove the theorem, we may assume that $\kk$ is an algebraically closed field. Assume $R$ satisfies the hypothesis of the theorem. We say that $R$ is {\it exceptional} if $R$ has no socle elements in degree $1$ and is not Koszul. With this terminology, the  implication (2)$\iff$(3) says that, if $R$ has no socle elements in degree $1$, then $R$ is exceptional if and only if it has Hilbert series $(1+2t-2t^3)(1-t)^{-1}$.

The implication (2)$\implies$(3) is based on an easy  numerical argument, see Lemma \ref{non-K}.
When $\ch\kk\ne 2$, the implication  (3)$\implies$(2) recovers known results. It can be deduced by putting together work of Backelin
\cite{Backelin}, Conca \cite{C2} and D'Al\`i \cite{Alessio}. More
precisely, the work in \cite{Backelin} covers the case when
$\dim_{\kk}R_2\le 2$ (this case does not require the characteristic
assumption),  the work in \cite{C2} covers the case when
$\dim_{\kk}R_2= 3$ and $R$ is artinian, and the work in
\cite{Alessio} covers the case when $\dim_{\kk}R_2=3$ and $R$ is not
artinian.

When $\ch\kk\notin\{2,3\}$ and $\kk$ is algebraically closed there are
exactly $2$ isomorphism classes of exceptional algebras. When
$\ch\kk=3$ there are three such classes. This  classification is proved by
D'Al\`i \cite[Theorem 3.1]{Alessio}, who also notes that the
classification needs to be enlarged when $\ch\kk=2$. We are able to describe the structure of such rings when $\ch\kk=2$ as well,  see Remark \ref{classify},  but we do not provide a classification in terms of isomorphism classes of $\kk$-algebras.

We now describe the ingredients of the proof  of the Main Theorem. A result of Avramov et al.\,\cite{AKM} gives the Backelin-Roos
property when $\dim_{\kk}R_1=3$, and in particular for all exceptional rings. Since the homological properties of interest
behave well under trivial fiber extensions, see
\ref{reduction-prelim}, we assume for most of the proof that $R$ does not have any socle
elements in degree $1$ and $\dim_{\kk}R_1>3$ whenever $\dim_{\kk}R_2=3$.  With these
assumptions, we proceed to investigate structural properties of the
ring, in terms of understanding the ideal $R_{\geqslant 2}$.  We then identify an appropriate complete
intersection $P$ such that the  induced homomorphism $P\to R$ is
Golod.

Our methods for understanding the structure of $R$ are inspired by
and use some of the approaches and results in \cite{C1}, \cite{C2}
and \cite{Alessio}. We develop new considerations in order to allow
for the treatment of the case $\ch\kk=2$ as well. Akin to
\cite{Alessio} and \cite{C2}, we need to consider a large number of
cases. However, we do not develop structural results that are as
detailed as the ones in {\it loc.\,cit.}, and we only establish
structure as needed to be able to construct the desired Golod
homomorphism.

The  Golod property is established as follows. When $P$ is a
quadratic complete intersection and $\varphi\colon P\to R$ is a
surjective homomorphism of standard graded $\kk$-algebras with
$\Ker(\varphi)\subseteq P_{\geqslant 2}$,  we show in Proposition \ref{main-criterion} that $R$ is Koszul
and $\varphi$ is Golod if and only if the map
$$
\nu^P\colon \Tor^P(\m^2,\kk)\to \Tor^P(\m,\kk)
$$
induced by the inclusion $\m^2\subseteq \m$ is zero, where $\m$
denotes the maximal homogeneous ideal of $R$. The  minimal free
resolution of $\kk$ over the complete intersection $P$ is well
understood due to a construction of Tate and it has a structure of DG $\Gamma$-algebra. We denote
by $\mathcal D$ the DG-algebra obtained by tensoring this minimal
free resolution with $R$ and we identify $\nu^P$ with the map
$\nu(\m\mathcal D)\colon \HH(\m^2\mathcal D)\to \HH(\m\mathcal D)$
induced by the inclusion $\m^2\mathcal D\subseteq \m\mathcal D$. We
use the DG $\Gamma$-algebra structure of $\mathcal D$ to develop
several results that allow us to prove that $\nu(\m\mathcal D)=0$, under various structural conditions on the ring $R$. The work of Avramov et al.\,\cite{AIS} provided some of the inspiration in this direction.  A particular aspect of our approach is that the
complete intersection $P$ is only identified at the end of each case
considered. We  identify a DG-algebra $\mathcal D$ first, we prove
$\nu(\m\mathcal D)=0$, and then we define $P$ to correspond to this
algebra.  For this reason, many of the preliminary results are only
stated in terms of  DG-algebras, rather than Golod homomorphisms.

We also give a local version of the Main Theorem. To describe it, assume now that $R$ is a
commutative local Noetherian ring with maximal ideal $\m$. We say
that $R$ is {\it Koszul} if its associated graded algebra $\agr R$
with respect to $\m$ is Koszul. The  absolutely Koszul property can
be defined for local rings in the same way we defined it earlier for
graded rings. In Lemma \ref{Golod when Koszul} we show that if $R$ is Koszul
and $\agr R$  has the Backelin-Roos property, then the completion
$\widehat R$ with respect to $\m$ has the Backelin-Roos property
as well. This  lemma, together with the Main Theorem, settles
the case when $R$ is Koszul in the theorem below. The non-Koszul cases are handled using the information on the Hilbert series of exceptional rings given by the implication (2)$\implies$(3) of the Main Theorem.  We show:

\begin{Theorem}{\bf (Local Version)}
Let $R$ be a commutative local Noetherian ring with maximal ideal
$\m$.  Assume $\dim_{R/\m}\m^2/\m^3\le 3$ and $\agr R$ is
quadratic.

There exists then a flat homomorphism of local rings $R\to R'$, where $R'$ is a local ring with  maximal ideal $R'\m$ and has the property that there  exists a surjective Golod homomorphim $P\to R'$ of local rings, with $P$
a complete intersection of codimension at most $3$.  Furthermore, if $R$ is Koszul then it is absolutely Koszul.
\end{Theorem}

In view of a result of Levin, see for example \cite[Proposition 5.18]{AKM}, this result adds the rings satisfying its hypothesis to a growing list of local rings for which it is known that  all finitely generated modules have rational Poincar\'e series, sharing a common denominator.

The  structure of the paper is as follows. Section
\ref{Golod-Koszul-Tate} explains the terminology and presents the
strategy for  establishing the Golod property for various
homomorphisms. Section \ref{key}  develops key homological tools involving DG-algebras that enable us to use the strategy. The  components of the
proof of the Main Theorem are spread across Sections \ref{artinian}, \ref{dim2}, \ref{dim3-nonartinian} and \ref{exceptional-sec}, as
indicated by the title of each section. The  proof of the Main
Theorem is given at the end of Section \ref{exceptional-sec}, where we also describe the structure of exceptional rings. The  local version of the Main Theorem  is
proved in Section \ref{local rings}.
\begin{acknowledgement}
We are grateful to the referee, who read our manuscript with remarkable patience and attention, and helped us straigthen out arguments and typos. We also wish to thank Aldo Conca and Alessio D'Al\`i for talking to us about the part of their work that is used here.
\end{acknowledgement}

\section{Preliminaries: Golod homomorphisms, Koszul rings and Tate resolutions}
\label{Golod-Koszul-Tate} In this section we present
concepts and preliminary results and we introduce notation and
conventions. To  begin  with, we indicate that the terminology of
{\it local ring}  includes the convention that the ring is
commutative and Noetherian. If $\kk$ is a field, a {\it standard
graded $\kk$-algebra}  is a graded algebra $R=\oplus_{i\geqslant
0}R_i$ with $R_0=\kk$ and $\dim_{\kk}R_1<\infty$, and such that $R$
is generated  by $R_1$ as a $\kk$-algebra. If $R$ is a standard
graded $\kk$-algebra, then $R$ can be written as $R=Q/I$ with $Q$ a
polynomial ring over $\kk$ in finitely many variables in degree $1$
and $I$ a homogeneous ideal generated by forms of degree at least
$2$. We say that $R$ is {\it quadratic} if $I$ is
generated by forms of degree $2$ (quadrics).

\begin{bfchunk}{The  ring $R$.}
\label{R}Let $\kk$ be a field. Let $(R,\m,\kk)$ denote a local  ring
with maximal ideal $\m$ and residue field $\kk$ or a standard graded
$\kk$-algebra with maximal homogenous ideal $\m=R_{\geqslant 1}$. In
the latter case, we employ the assumption that all objects involved
in our statements (such as modules, free resolutions, homology,
homomorphisms, etc.) are also graded.  When we want only the graded
hypothesis to be in effect,  we simply say ``$R$ is graded''. Some
of the results and definitions that we use and recall below are
stated in the original references only for local rings. It is
however a standard observation that such results can be translated
in the graded case, with assumptions and notation as described
above. We will expand the statements as to cover the graded case as
well, mostly without commenting on this issue.

Let $M$ be a finitely generated $R$-module. We denote by
$$\sum_{i\geqslant 0}\rank_{\kk}(\Tor_i^R(M,\kk))z^i\in \mathbb Z[[z]]$$
the  {\it Poincar\'e
series} $\Po_M^R(z)$ of $M$.  When $R$ is graded we also consider the {\it bigraded Poincar\'e
series}, which is denoted as follows:
$$\Po^R_M(z,t)=\sum_{i\geqslant
0, j\in \mathbb Z}\rank_{\kk}(\Tor_{i}^R(M,\kk)_j)z^it^j\in \mathbb Z[t,t^{-1}][[z]]\,,$$ with $j$ denoting
internal degree. Also, we denote by
$$
H_M(t)=\sum_{j\in \mathbb Z}\dim_{\kk}(M_j)t^j\in \mathbb Z[[t]][t^{-1}]
$$
the {\it Hilbert series} of $M$.

We let $\agr R$, respectively $\agr M$, denote the associated graded algebra, respectively asscociated graded module, with respect to
$\m$; $\agr R$ is a standard graded $\kk$-algebra with maximal homogeneous
ideal $({\agr R})_{\geqslant 1}$ and $\agr M$ is a graded $\agr R$-module.
\end{bfchunk}

\begin{bfchunk}{Koszul and absolutely Koszul rings.}
\label{Koszul-defs} We say that a finitely generated $R$-module $M$
is {\it Koszul} if $\agr M$ has a
linear graded free resolution over $\agr R$, meaning that the
non-zero entries of the matrices describing the differentials are
forms of degree $1$. We say that the ring $R$ is {\it Koszul} if
$\kk$ is a Koszul module. This  definition agrees with the classical
definition of a Koszul algebra when $R$ is graded.  Koszul algebras have been studied extensively. We refer to Conca {\it et.\,al} \cite{CDR} for some recent developments, as well as the classical properties recorded below.
\begin{subchunk}
If $R$ is Koszul, then $\agr R$ is quadratic.
\end{subchunk}
\begin{subchunk}
\label{PH} A standard graded $\kk$-algebra $R$ is Koszul if and only
if $\Po_{\kk}^R(z)H_R(-z)=1$.
\end{subchunk}
\begin{subchunk}
\label{Grobner}
If $R$ is graded and the defining ideal $I$ of $R$ has a Gr\"obner basis of quadrics with repect to some term order (or in other words, $R$ is G-quadratic), then $R$ is Koszul. In particular, $R$ is Koszul when $I$ is generated by quadratic monomials.
\end{subchunk}

 We say that $R$
is {\it absolutely Koszul} if every finitely generated $R$-module
has a Koszul syzygy or, in other words, it has  finite {\it
linearity defect}; see \cite{HI} for the definition of linearity
defect.  A consequence of the definition is as follows, cf.\,\cite[Proposition 1.15]{HI}:
\begin{subchunk}
If $R$ is an absolutely Koszul standard graded $\kk$-algebra, then $\Po_M^R(z)H_R(-z)\in \mathbb Z[z]$ for every finitely generated graded $R$-module $M$.
\end{subchunk}
\begin{subchunk}
\label{flat-absKoszul}
Let $R\to R'$ be a flat ring homomorphism (of local or graded rings), where $(R',\m',\kk')$ is such that $\m'=R'\m$. It follows from \cite[Remark 1.8]{AIS} that the $R$-module $M$ and the $R'$-module $M\otimes_RR'$ have the same linearity defect. In particular, $R$ is Koszul if and only if $R'$ is Koszul and, if $R'$ is absolutely Koszul, then $R$ is absolutely Koszul. 
\end{subchunk}
\end{bfchunk}

\begin{bfchunk}{The  maps $\nu$.}
\label{nu-def} Let $M$ be a finitely generated $R$-module. For each
$i\ge 0$ we let $\nu^R_i(M)$ denote the map
\begin{equation}
\nu^R_i(M)\colon \Tor_i^R(\m M,\kk)\to \Tor_i^R(M,\kk)
\end{equation}
induced by the inclusion $\m M\subseteq M$. We write $\nu^R(M)=0$ to
indicate that the maps  $\nu^R_i(M)$ are identically zero for all
$i\ge 0$. The  vanishing of these maps can be used to detect the
Koszul property, as described next.
\begin{subchunk}
\label{Koszul-nu} By \c{S}ega \cite[Proposition 7.5]{powers}, $R$ is
Koszul if and only if $\nu^R(\m^n)=0$ for all $n\ge 0$.
\end{subchunk}
\begin{subchunk}
\label{graded-Koszul-nu}
 It is known that  $\nu^R(\m)=0$ if and only if the Yoneda algebra $\Ext_R(\kk,\kk)$ is generated in degree $1$, cf.\! Roos \cite[Corollary 1]{Ro}. In particular,  if $R$ is graded, then  $\nu^R(\m)=0$ if and only if $R$ is Koszul.
\end{subchunk}
\begin{subchunk}
\label{powers} If $R$ is Koszul and $M$ is a finitely generated
$R$-module, it follows from \cite[Theorem 3.2]{powers} that the
Castelnuovo-Mumford regularity of $\agr M$ over $\agr R$ is given by
the formula:
\begin{equation}
\label{CM-reg} \reg_{\agr R}(\agr M)=\inf\{s\ge 0\mid
\text{$\nu^R(\m^nM)=0$ for all $n\ge s$}\}\,.
\end{equation}
\end{subchunk}
\end{bfchunk}

\begin{bfchunk}{The homomorphism $\varphi$; Golod rings and homomorphisms.}
\label{phi}
Let $\varphi\colon (P,\n,\kk)\to (R,\m,\kk)$ denote a surjective
homomorphism of local rings or of standard graded $\kk$-algebras.   We let $\agr\varphi\colon
\agr P \to \agr R$ denote the induced map between the associated
graded algebras; note that $\varphi=\agr \varphi$ in the graded
case.
\begin{subchunk} By a result of Serre, cf.\,Avramov \cite[Prop.
3.3.2]{Avr98} for example, the following coefficientwise inequality
holds:
\begin{equation}
\label{Golod-def}
\Po_{\kk}^R(z)\preccurlyeq\frac{\Po_{\kk}^P(z)}{1-z(\Po_R^P(z)-1)}\,.
\end{equation}
In the graded case, the  bigraded version of this inequality is as
follows:
\begin{equation}
\label{Golod-def-graded}
\Po_{\kk}^R(z,t)\preccurlyeq\frac{\Po_{\kk}^P(z,t)}{1-z(\Po_R^P(z,t)-1)}\,.
\end{equation}
 When $R$ is local, we say that the homomorphism $\varphi$ is {\it Golod} when equality holds in \eqref{Golod-def}. In the graded case, we say that $\varphi$  is  {\it Golod}  if equality holds in \eqref{Golod-def-graded}. When  $P$ is a regular ring, we have that  $R$ is a Golod ring if and only if $\varphi$ is Golod, see \cite[Section 5]{Avr98}; we will use this equivalence for a working definition of a Golod ring.

In general, it is difficult to prove directly the equalities in
\eqref{Golod-def} and \eqref{Golod-def-graded},  and  methods such
as constructing trivial Massey operations need to be employed. We
record below two results that we use in establishing the Golod
property for $\varphi$.
\end{subchunk}

\begin{subchunk}
\label{2-linear} Assume that $R$ is graded, and hence $\varphi$ is a
homomorphism of standard graded $\kk$-algebras.  By \cite[5.8]{HI}
the following are equivalent:
\begin{enumerate}
\item $\varphi$ is Golod and $R$ is Koszul;
\item The  $P$-module $\Ker(\varphi)$ has a $2$-linear resolution (hence $\reg_P(R)=1$) and $P$ is Koszul.
\end{enumerate}
\end{subchunk}

\begin{subchunk}
\label{Golod-lemma} (Rossi, \c Sega \cite[Lemma 1.2]{RS}) Suppose
there exists a positive integer $a$ such that
\begin{enumerate}[\quad\rm(1)]
\item The   map $\Tor_i^P(R,k)\to \Tor_i^P(R/\m^{a},k)$ induced by the projection $R\to R/\m^a$ is zero for all $i>0$.
\item The  map $\Tor_i^P(\m^{2a},k)\to \Tor_i^P(\m^{a},k)$  induced by the inclusion $\m^{2a}\hookrightarrow \m^a$ is zero for all $i\ge 0$.
\end{enumerate}
 Then  $\varphi$ is a Golod homomorphism.
\end{subchunk}
\end{bfchunk}

\begin{lemma}
\label{golod-criterion} Assume $\text{Ker}(\varphi)\subseteq \n^2$. If $\nu^P(\n)=0=\nu^P(\m)$, then $\varphi$
is Golod.
\end{lemma}
\begin{proof}
We apply \ref{Golod-lemma} with $a=1$.  Condition (2) is satisfied,
given that $\nu^P(\m)=0$.  Let $i>0$. To  verify (1), we see that the map
$\Tor_i^P(R,\kk)\to \Tor_i^P(R/\m,\kk)$  factors through the
map $\Tor_i^P(R/\m^2,\kk)\to \Tor_i^P(R/\m,k)$ induced by the projection $R/\m^2\to R/\m$. Since  $\text{Ker}(\varphi)\subseteq \n^2$, this map can be
identified with $\nu_{i-1}^P(\n)$,  and it is thus zero for all
$i>0$.
\end{proof}

\begin{bfchunk}{Complete intersections.}
We say $P$ is a {\it complete intersection} of {\it codimension} $d$
if $P=Q/(f_1, \dots, f_d)$, where $(Q,\q,\kk)$ is  a regular local
ring (when $R$ is local) or $(Q,\q,\kk)$ is a polynomial ring over $\kk$ (when $R$ is
graded) and  $f_1, \dots, f_d$ is a regular sequence with $f_i\in\n^2$ for all $i$. The elements $f_i$ are assumed homogeneous in the graded case.  Note that we do not consider completions in the
definition in the local case.
When $R$ is graded and $\deg(f_i)=2$ for all $i$, we say that $P$ is
a {\it complete intersection of quadrics} or a {\it quadratic
complete intersection}.
\begin{subchunk}
\label{Golod-ci} If $R$ is Koszul, $\varphi$ is Golod and $P$ is a
complete intersection, then $R$ is absolutely Koszul, by
\cite[Theorem 5.9]{HI}.
\end{subchunk}
\begin{subchunk}
\label{nu-ci}  If $R$ is graded, $P$ is a complete intersection
of quadrics and $\Ker(\varphi)\subseteq \n^2$, then  $\nu^P(\m)=0$ implies $\nu^R(\m)=0$. The
explanation for this statement is postponed to \ref{proof-of-nu-ci}.
\end{subchunk}
\end{bfchunk}

The  next result is our main criterion for checking that a graded
ring is absolutely Koszul.

\begin{proposition}
\label{main-criterion} Assume that $R$ is graded,  $\Ker(\varphi)\subseteq \n^2$  and $P$ is  a
complete intersection of quadrics. The  following are equivalent:
\begin{enumerate}[\quad\rm(1)]
\item $\varphi$ is Golod and $R$ is Koszul;
\item $\nu^P(\m)=0$.
\end{enumerate}
If these assumptions hold, then $R$ is absolutely Koszul.
\end{proposition}
\begin{proof}
If (1) holds, then $\reg_P(R)=1$ by \ref{2-linear} and then
$\nu^P(\m^n)=0$ for all $n\ge 1$ by \ref{powers}. Conversely, since
$P$ is Koszul we have $\nu^P(\n)=0$ by \ref{graded-Koszul-nu}. Thus,
if (2) holds then $\varphi$ is Golod by Lemma \ref{golod-criterion}.
To  see that $R$ is Koszul, note that  $\nu^P(\m)=0$ implies
$\nu^R(\m)=0$ by \ref{nu-ci}, hence $R$ is Koszul by
\ref{Koszul-nu}. Apply \ref{Golod-ci} for the last statement.
\end{proof}

The  next result contributes towards the proof of the  local
version of the main theorem.

\begin{proposition}
\label{graded-help} If $\agr \varphi$ is Golod and $R$ is Koszul,
then $\varphi$ is also Golod.
\end{proposition}

\begin{proof} Assume $\agr\varphi$ is Golod. By \ref{2-linear} we conclude that $\agr P$, thus $P$, is Koszul and $\reg_{\agr P}(\agr R)=1$. In particular, we have  $\Ker(\varphi)\subseteq \n^2$. By \eqref{CM-reg} we conclude  $\nu^P(\m^n)=0$ for all $n$. Since $P$ is Koszul, we also have $\nu^P(\n)=0$ by \ref{Koszul-nu}.  Then  $\varphi$ is Golod by Lemma \ref{golod-criterion}.
\end{proof}

For simplicity, we state and prove the result below for local rings,
and we note that this statement (and its proof) can be translated in
a standard manner to the graded situation.

\begin{proposition}
\label{general-reduction} Let $\varphi\colon (P,\n,\kk)\to
(R,\m,\kk)$ be a surjective homomorphism of local rings with
$\Ker(\varphi)\subseteq \n^2$ and let $u\in \n\smallsetminus \n^2$.
We set $v=\varphi(u)$, $\ov R=R/(v)$ and $\ov P=P/(u)$ and let
$\ov\varphi\colon \ov P\to \ov R$ denote the induced homomorphism.

If $v\m=0$, $u$ is $P$-regular and  $\ov \varphi$ is Golod, then
$\varphi$ is Golod.
\end{proposition}

\begin{proof}
Note that $v\in \m\smallsetminus \m^2$. The  hypotheses on $u$ and
$v$ give that the canonical maps $R\to \ov R$ and $P\to \ov P$ are
large, in the sense of Levin, see \cite[Theorem 2.2]{Levin-large}.
In particular, \cite[Theorem 1.1]{Levin-large} gives the
equations below. (Note that we write $\Po^R_M$
instead of $\Po^R_M(z)$, to simplify notation.)
\begin{gather}
\label{a}\Po^R_{\kk}=\Po^{\ov R}_{\kk}\cdot \Po^R_{\ov R}=\Po^{\ov R}_{\kk}\cdot (1+z\Po^R_{\kk})\\
\label{b}\Po^P_{\kk}=\Po^{\ov P}_{\kk}\cdot \Po^P_{\ov P}=\Po^{\ov P}_{\kk}\cdot (1+z)\\
\label{c}\Po^P_{\ov R}=\Po^{\ov P}_{\ov R}\cdot \Po^P_{\ov
P}=\Po^{\ov P}_{\ov R}\cdot (1+z)
\end{gather}
The equality  $\Po^R_{\ov R}=1+z\Po^R_{\kk}$ used in the first
equation comes from the fact that $x\m=0$, which yields a short
exact sequence
\begin{equation}
\label{seq} 0\to \kk\to R \to\ov R\to 0\,.
\end{equation}
In addition, the fact that $\ov \varphi$ is Golod gives:
\begin{equation}
\label{d}\Po^{\ov R}_{\kk}=\frac{\Po^{\ov P}_{\kk}}{1-z(\Po^{\ov
P}_{\ov R}-1)}\,.
\end{equation}
Using the formulas above, we obtain:
\begin{equation}
\label{xx} \Po^R_{\kk}=\frac{\Po^{\ov R}_{\kk}}{1-z\Po^{\ov
R}_{\kk}}=\frac{\Po^{\ov P}_{\kk}}{1-z(\Po^{\ov P}_{\ov
R}-1)-z\Po^{\ov
P}_{\kk}}=\frac{\Po^{P}_{\kk}}{1-z\left((\Po^{P}_{\ov
R}+\Po^{P}_{\kk}-1-z)-1\right)}\,.
\end{equation}
More precisely, we used \eqref{a} for the first equality, \eqref{d}
for the second equality and \eqref{c} and \eqref{b} for the last
equality.

The  short exact sequence \eqref{seq} induces a long exact sequence
\begin{align*}
\dots \to&\Tor_i^P(\kk, \kk)\to \Tor_i^P(R,\kk)\to \Tor_i^P(\ov R,
\kk)\to \dots \to \Tor_2^P(\ov R,\kk)\to\Tor_1^P(\kk,\kk)\to\\ \to
&\Tor_1^P(R,\kk)\to \Tor_1^P(\ov R,\kk)\to \Tor_0^P(\kk,\kk)
\xrightarrow{0}\Tor_0^P(R,\kk)\xrightarrow{\cong}\Tor_0^P(\ov
R,\kk)\to 0\,.
\end{align*}
For each $i$ and each finite $P$-module $M$ we denote  by
$\beta_i(M)$ the number $\rank_{\kk}\Tor_i^P(M,\kk)$. From the
sequence above we can read the following (in)equalities:
\begin{gather*}
\beta_0( R)=\beta_0(\ov R)=\beta_0(\ov R)+\beta_0(\kk)-1  \\
\beta_1(R)\le \beta_1(\ov R)+\beta_1(\kk)-\beta_0(\kk)=\beta_1(\ov R)+\beta_1(\kk)-1\\
\beta_i(R)\le \beta_i(\ov R)+\beta_i(\kk)\quad\text{for all
$i>1$}\,.
\end{gather*}
Putting these inequalities together, we have the coefficientwise
inequality
\begin{equation}
\label{mm} \Po^P_R\preccurlyeq \Po^P_{\ov R}+\Po^P_{\kk}-1-z\,.
\end{equation}
Note that if  $A$ and $B$ are power series with zero constant term, non-negative
coefficients and $A\preccurlyeq B$, then $\frac{1}{1-A} \preccurlyeq
\frac{1}{1-B}$, hence \eqref{mm} yields the coefficientwise
inequality below. The  equality comes from \eqref{xx}.
\begin{equation}
\label{Golode} \frac{\Po^{P}_{\kk}}{1-z(\Po^P_R-1)}\preccurlyeq
\frac{\Po^{P}_{\kk}}{1-z\left((\Po^{P}_{\ov
R}+\Po^{P}_{\kk}-1-z)-1\right)}=\Po_{\kk}^R\,.
\end{equation}
We see that equality must hold in Serre's inequality
\eqref{Golod-def}, hence $\varphi$ is Golod.
\end{proof}

The  result above allows for an important reduction, as described next.

\begin{bfchunk}{ Reduction.}
\label{reduction-prelim}
We say that $R$ is a {\it trivial fiber extension} of a ring $\ov R$ if
there exist elements $v_1, \dots v_n$ with $v_i\in \m\smallsetminus
\m^2$ and $v_i\m=0$ for all $i$,  and  such that $\ov R=R/(v_1, \dots, v_n)$. We say that $R$ is {\it not} a trivial fiber extension if there is no element $v\in \m\smallsetminus \m^2$ with $v\m=0$.

Let $R=Q/I$ with $(Q,\q,\kk)$ a
regular local ring or, in the graded case, a polynomial ring.
Assume that $R$ is a trivial fiber extension of $\ov R=R/(v_1, \dots, v_n)$, where $v_i$ are as above. We may assume that the elements $v_1, \dots, v_n$ are part of a minimal generating set of $\m$.
Let $\boldsymbol w=w_1, \dots, w_n$ denote preimages of $v_1, \dots, v_n$ in $Q$; they are part of a minimal generating set of $\q$ and in particular they form a $Q$-regular sequence. Set
 $$\ov Q=Q/(\boldsymbol w),\quad \ov\q=\q/(\boldsymbol w),\quad\text{and}\quad \ov I=\left(I+(\boldsymbol w)\right)/(\boldsymbol w)\,.$$
 Then  $(\ov Q,\ov\q,\kk)$ is also regular, $\ov I\subseteq \ov{\q}^2$ and $\ov R=\ov Q/\ov I$. Let $\ov f_1,  \dots, \ov f_d$  be a $\ov Q$-regular sequence
contained in $\ov I$, and let $f_1, \dots, f_d$ denote preimages of
these elements in $I$. Set $P=Q/(f_1, \dots f_d)$ and $\ov P=\ov
Q/(\ov f_1, \dots, \ov f_d)$ and for each $i$  let $u_i$ denote the image of $w_i$ in
$P$. Note that $w_1,\dots, w_n, f_1, \dots,f _d$ is a $Q$-regular sequence and in
particular it follows that $u_1, \dots, u_n$ is a $P$-regular sequence. Let
$\ov\varphi\colon \ov P\to \ov R$ and $\varphi\colon P\to R$ denote
the canonical projections. With this notation, we have the following
properties:
\begin{subchunk}
If $\ov \varphi$ is Golod, then $\varphi$ is Golod; this follows
by applying inductively Proposition \ref{general-reduction}. In particular, if $\ov R$ is Golod, then $R$ is Golod.
\end{subchunk}
\begin{subchunk}
The  complete intersections $P$ and $\ov P$ have the same
codimension. In the graded case, $\ov P$ is quadratic if and only
$P$ is quadratic.
\end{subchunk}
\begin{subchunk}
\label{reduction-prelim-Koszul}
When $R$ is graded, $\ov R$ is Koszul if and only if  $R$ is Koszul,
see \cite[Theorem 4]{BF2}.
\end{subchunk}
\end{bfchunk}

Our main technique used in the proof of the Main Theorem is the use
of DG-algebra structures. We now proceed to describe it.  Should the
reader need it, a standard reference for the definition of the
notion of DG-algebra, DG module, divided power variables and related
background is the survey \cite{Avr98}.

\begin{bfchunk}{Semi-free $\Gamma$-extensions and a result of Levin.}
\label{Levin} Let $\mathcal D$ be a DG-algebra. A {\it semi-free
$\Gamma$-extension} of $\mathcal D$ is a DG-algebra $\mathcal
A=\mathcal D\langle X\rangle$ obtained by iterated adjunction of
variables, as described in \,\cite[6.1]{Avr98}. Note that $\mathcal
A$ is free as a graded $\mathcal D$-module (i.e.\! forgetting
differentials). If $\mathcal A$ is a complex of $R$-modules, we say that $\mathcal A$ is {\it minimal} if
$\partial(\mathcal A)\subseteq \m\mathcal A$. It is known that there
exists  a semi-free $\Gamma$-extension $\mathcal T=R\langle
X\rangle$ of $R$  such that $\mathcal T$ is a minimal free
resolution of $\kk$; such a resolution is called a {\it Tate
resolution} of $\kk$.

 We recall below a result of Levin.

\begin{Lemma}
{\rm (\cite[Lemma 2]{Lev})} Let $\mathcal D$ be a DG  $R$-algebra
and $\mathcal A$ a DG $\mathcal D$-module which is free as a graded
$\mathcal D$-module and such that $\mathcal A$ is minimal. Let $M$,
$N$ be $R$-modules such that $\m M\subseteq N\subseteq M$ and the canonical map $ g\colon N\otimes_R\mathcal A\to
M\otimes_R\mathcal A$ is injective.

If the induced homomorphism $\HH(N\otimes_R\mathcal D)\to
\HH(M\otimes_R\mathcal D)$ is zero, so is the induced homomorphism
$\HH(N\otimes_R\mathcal A)\to \HH(M\otimes_R\mathcal A)$.
\end{Lemma}

The context in which this lemma will be applied is as follows: Assume $\mathcal D$ is a semi-free $\Gamma$-extension of $R$ and
$\mathcal A$ be a minimal semi-free $\Gamma$-extension of $\mathcal
D$. Since $\mathcal A$ is a free $R$-module, the map $g$ in the
statement of the lemma is injective.  Then  the hypotheses of the
lemma are satisfied for this choice of $\mathcal D$ and $\mathcal
A$, with $M$ and $N$ such that $\m M\subseteq N\subseteq M$.
\end {bfchunk}

\begin{bfchunk}{Short Tate complexes.}
\label{truncated} Assume that there exists a local or
graded regular ring $(Q,\q,k)$ and an ideal $I$ such that that
$R=Q/I$ with $I\subseteq \q^2$. (This  assumption is satisfied
whenever $R$ is graded or $R$ is complete.)

Let $\mathcal K$ denote the Koszul $R$-complex on a
minimal generating set of $\m$. Let $\mathcal K^Q$ denote a  Koszul complex on a minimal generating set of $\q$ such that $\mathcal K=\mathcal K^Q\otimes_QR$. If $z$ is a cycle in $\mathcal K_1$, we associate to it an element $f\in I$ as follows: Lift $z$ to an element $\tilde z\in \mathcal K_1^Q$ so that $z=\tilde z\otimes 1$, and set $f=\partial(\tilde z)$. In turn,  any element $f\in I$ can be used to construct a cycle $z\in \mathcal K_1$.  While $f$ and $z$ do not determine each other uniquely, this construction yields an isomorphism $H_1(\mathcal K)\cong I/\q I$.

Let $z_1,\dots, z_d$ be a set of cycles in $\mathcal K_1$. For each $i$ we construct $f_i\in I$ corresponding to $z_i$, as described above. Let $\mathcal D$ denote the semi-free $\Gamma$-extension of $\mathcal K$
defined by
\begin{equation}
\label{define-D} \mathcal D=\mathcal K\langle Z_1, \dots, Z_d\mid
\partial(Z_i)=z_i\rangle\,.
\end{equation}
If $[z_1], \dots, [z_d]$  are linearly independent (equivalently, $f_1, \dots, f_d$ is part of a minimal generating set of $I$) and $f_1, \dots, f_d$ is a regular sequence in
$Q$, we say that
$\mathcal D$ is a {\it short Tate complex corresponding to the complete intersection} $P=Q/(f_1, \dots, f_d)$. Conversely, any complete intersection ring  $P=Q/(f_1, \dots, f_d)$ such that the regular sequence $f_1, \dots, f_d$ is part of a minimal generating set of $I$ can be used to
construct a short Tate complex $\mathcal D$ corresponding to $P$.

Assume now that $\mathcal D$ is a short Tate complex corresponding to a complete intersection $P=Q/(f_1,
\dots, f_d)$. One can construct then a Tate
resolution $\mathcal A$ of $\kk$ over $R$ to be a semi-free $\Gamma$-extension of $\mathcal D$. In particular, $\mathcal A$ is a free graded module over $\mathcal D$.

The  construction of a Tate resolution of $\kk$ over the complete
intersection ring $P$  is well understood: One starts with a Koszul
complex $\mathcal K^P$ over $P$, that can be chosen such that
$\mathcal K=\mathcal K^P\otimes_PR$, and adjoins to $\mathcal K^P$
variables in degree $2$ to kill the cycles $z_1, \dots, z_d$ corresponding to the
regular sequence $f_1, \dots, f_d$ . In particular, a minimal Tate
resolution $\mathcal F$ of $\kk$ over $P$ can be chosen such that
$\mathcal D=R\otimes_P\mathcal F$. We see then that the maps
$\nu^P_i(\m)\colon\Tor_i^P({\m}^{2},\kk)\to \Tor_i^P({\m},\kk)$ can
be identified with the canonical maps
\begin{equation}
\label{identify-D} \HH_i(\m^2\otimes_R\mathcal D)\to
\HH_i(\m\otimes_R\mathcal D)
\end{equation}
induced by the inclusion $\m^2\subseteq \m$. Also, since $\mathcal
A$ is a minimal free resolution of $\kk$ over $R$,  the maps
$\nu^R_i(\m)\colon\Tor_i^R({\m}^{2},\kk)\to \Tor_i^R({\m},\kk)$ can
be identified with the maps
\begin{equation}
\label{identify-A} \HH_i(\m^2\otimes_R\mathcal A)\to
\HH_i(\m\otimes_R\mathcal A)
\end{equation}
induced by the  same inclusion.
\begin{subchunk}
\label{levin-help} As noted above, $\mathcal A$ is a free as a
graded $\mathcal D$-module, and thus Levin's result \ref{Levin}
yields that  $\nu^P(\m)=0$ implies $\nu^R(\m)=0$.
\end{subchunk}
\begin{subchunk}
\label{Q} Assume that $R$ is graded.  In this case, note that if
$f_1, \dots, f_d$ is a regular sequence of quadrics, then $f_1,
\dots, f_d$ is also part of a minimal generating set of $I$, for
degree reasons.
\end{subchunk}
\end{bfchunk}

\begin{bfchunk}{Proof of {\rm \ref{nu-ci}}.}
\label{proof-of-nu-ci} Assume $R$ and $P$ are as in \ref{nu-ci},
that is, $R$ is graded, and $P$ is a complete intersection of
quadrics.  Then  $R=Q/I$ with $Q$ and $I$ as in \ref{truncated}.  We
may assume that $P=Q/J$ with $J$ an ideal generated by a regular
sequence of quadrics such that $J\subseteq I$. As noted above, this
regular sequence must be part of a minimal generating set of $I$. We
consider then a short Tate complex $\mathcal D=\mathcal K\langle
Z_1, \dots, Z_d\mid \partial(Z_i)=z_i\rangle$ corresponding to $P$. Using \ref{levin-help}, we conclude that
$\nu^P(\m)=0$ implies $\nu^R(\m)=0$. \hfill \qed
\end{bfchunk}

\begin{bfchunk}{Notation.}
\label{more-notation} If $\mathcal F$ is a complex of $R$-modules we
denote by $\nu(\mathcal F)$ the homomorphism of complexes
\begin{equation}
\label{nu-complex} \nu(\mathcal F)\colon \HH(\m \mathcal F)\to
\HH(\mathcal F)
\end{equation}
induced by the inclusion $\m\mathcal F\subseteq \mathcal F$. When
$\mathcal F$ is a free resolution of $\kk$, we identify
$\nu(\m\mathcal F)$ with the map
$$
\HH(\m^2\otimes_R\mathcal F)\to \HH(\m\otimes_R\mathcal F)
$$
induced by the inclusion $\m^2\subseteq \m$.  In particular, if
$\mathcal D$ and $\mathcal A$ are as in \ref{truncated}, with
$\mathcal D$ a short Tate complex, and $P$ is the corresponding
complete intersection ring, then our conventions yield for each $i$
the identifications
\begin{equation}\nu_i^P(\m)=\nu_i(\m \mathcal D)\quad \text{and}\quad  \nu_i^R(\m)=\nu_i(\m \mathcal A)\,.
\end{equation}
\end{bfchunk}

In view of the discussion in \ref{truncated}, \ref{Q} and the
notation in \ref{more-notation}, we can reformulate Proposition
\ref{main-criterion} in terms of short Tate complexes as follows.

\begin{corollary}
\label{tool} Assume that $R$ is graded and quadratic and write
$R=Q/I$ with $Q$ and $I$ as in {\rm \ref{truncated}}. Let
$\mathcal D$ be a short Tate complex corresponding to a quadratic complete intersection $P=Q/(f_1, \dots, f_d)$.

The  following are  equivalent:
\begin{enumerate}[\quad\rm(1)]
\item The  induced homomorphism $\varphi\colon P\to R$ is Golod and $R$ is Koszul;
\item $\nu(\m\mathcal D)=0$.
\end{enumerate}
If these assumptions hold, then $R$ is absolutely Koszul. \qed
\end{corollary}

\begin{remark} \label{Golod-ring} Taking $P=Q$ in the corollary above, we see that the following are equivalent for a standard graded quadratic $\kk$-algebra $R$:
\begin{enumerate}
\item $R$ is a Koszul Golod ring;
\item $\nu(\m\mathcal K)=0$.
\end{enumerate}
\end{remark}

\section{Key homological techniques}
\label{key}
The  results in this section identify some of our main techniques in
proving $\nu(\mathcal \m \mathcal D)=0$, under adequate conditions on
a short Tate complex $\mathcal D$. We record only the information on the
complex $\mathcal D$ and we leave for later the discussion of the
conclusions that can be formulated in terms of Golod homomorphisms,
in view of Corollary \ref{tool}.

Throughout this section, the ring $R$ is assumed to be as in \ref{R}
(either local or graded), and $x_1, \dots, x_e$ denotes a minimal
generating set of $\m$. We fix  integers $s, t,u\in\{1,\dots, e\}$, not necessarily distinct.
In addition, we use the notation introduced below for Koszul
complexes and for reduction modulo an element.

\begin{bfchunk}{Notation for Koszul complexes.}
\label{notation-Koszul} Let $\mathcal K=R\langle X_1, \dots,
X_e\rangle$ with $\partial(X_i)=x_i$ for all $i$ denote the Koszul
complex on $x_1, \dots, x_e$. In general, if $a\in\m$, we use the
corresponding capital leter  $A$ to identify an element $A\in
\mathcal K_1$ such that $\partial(A)=a$.

 If $X$ is a subset of  $\{X_1, \dots, X_e\}$ we denote by $\mathcal K^X$ the DG-algebra $$\mathcal K^X= R\langle X_i\mid i\in \{1,\dots,e\}, X_i\notin X\}\,.$$
We regard $\mathcal K^X$ as a subalgebra of $\mathcal K$. If
$X=\{X_i\mid i\in I\}$ then we may also write $\mathcal K^I$ instead
of $\mathcal K^X$. If $I=\{i\}$ or $I=\{i,j\}$, we also write
$\mathcal K^I=\mathcal K^i$, respectively $\mathcal K^I={\mathcal
K}^{i,j}$.

As explained in Section 1, we adjoin DG $\Gamma$-variables of degree $2$ to $\mathcal K$ to create a short Tate complex. In order to easily  distinguish beteween variables of degree $1$ and variables of degree $2$, we denote the variables of degree $2$ using capital greek letters such as $\Lambda$ and $\Upsilon$.
\end{bfchunk}

\begin{bfchunk}{Reduction modulo $x_i$.}
\label{reduction-x} Let $x\in \{x_1, \dots, x_e\}$; the choice of
the element $x$ will be made clear in each instance when the
notation introduced here is used. We set $\ov R=R/(x)$. If $\mathcal
A$ is a complex of $R$-modules, we set $\overline{\mathcal
A}=\mathcal A/x\mathcal A$. If $a\in R$ or $A\in \mathcal A$, we
denote by $\ov a$, respectively $\ov A$ the corresponding class
modulo $(x)$, respectively $x\mathcal A$. If $\a$ is an ideal of
$R$,  then $\ov\a$ denotes the ideal $\a+(x)/(x)$ of $\ov R$.   In
particular, if $x=x_i$, note that $\ov{\mathcal K^i}$ is the Kozul
complex on the minimal generating set of $\ov\m$ consisting of all
$\ov x_j$ with $j\ne i$.

Let $\a\subseteq \m$ be an ideal. If $\mathcal A$ is  a minimal
semi-free $\Gamma$-extension of $R$, then the following equivalence
holds:
\begin{equation}
\label{implication}\nu(\ov\a\ov{\mathcal A})=0\,\iff\,Z(\m\a\mathcal
A)\subseteq B(\a\mathcal A)+x\m\mathcal A\,.
\end{equation}
Indeed, the equality $\nu(\ov\a\ov{\mathcal A})=0$ is equivalent to
the inclusion $Z(\m\a\mathcal A)\subseteq B(\a\mathcal A)+x\mathcal
A$. Since $Z(\m\a\mathcal A)$ and  $B(\a\mathcal A)$ are both contained
in $\m^2\mathcal A$, this inclusion is equivalent to  the inclusion $Z(\m\a\mathcal A)\subseteq B(\a\mathcal A)+ x\mathcal
A\cap \m^2\mathcal A$. If $A$ is a homogeneous
element in $\mathcal A_j$ for some $j$ and $xA\in \m^2\mathcal A_j$,
note that we must have $A\in \m\mathcal A_j$, since $\mathcal A_j$
is a free $R$-module and $x\in \m\smallsetminus \m^2$. We have then
$x\mathcal A\cap \m^2\mathcal A=x\m\mathcal A$, establishing thus
\eqref{implication}.
\end{bfchunk}

\begin{proposition}
\label{reduction-proposition} Assume $x_s\m^2=0$. Let  $\mathcal
E=\mathcal K\langle \Upsilon\rangle$  be a minimal semi-free
$\Gamma$-extension of $\mathcal K$ and  $\mathcal A=\mathcal
K^s\langle {\Lambda}\rangle$  be a minimal DG $\Gamma$-extension of
$\mathcal K^s$, where $\Upsilon$ and $\Lambda$ are sets of DG $\Gamma$-variables
of homological degree $2$. Set $\mathcal D=\mathcal K\langle \Upsilon\cup
\Lambda\rangle$.

Adopt the notation in {\rm \ref{reduction-x}} with $x=x_s$. If
$\nu(x_s\mathcal E)=0$ and $\nu(\overline{\m}\overline{\mathcal
A})=0$, then $\nu(\m \mathcal D)=0$.
\end{proposition}

Note that, since $\Upsilon$ and $\Lambda$ in this statement are sets of variables
in the same degree, the order in which we adjoin these variables does
not matter.

\begin{proof}
Set $\Upsilon'=\Upsilon\smallsetminus \Lambda$ and $\Lambda'=\Lambda\smallsetminus \Upsilon$ .  The complex  $\mathcal
D$  can be identified with ${\mathcal E}\langle {\Lambda'}\rangle$ and
also with $\mathcal A\langle X_s, \Upsilon'\rangle$. Since  $\mathcal D$ is
free when considered as a graded module over $\mathcal E$ and
$\mathcal D_i$ is a free $R$-module for each $i$, we  use  Levin's
result \ref{Levin} and the hypothesis that $\nu(x_s\mathcal E)=0$ to
see that $\nu(x_s\mathcal D)=0$.

Note that $\overline{\mathcal D}$ is also free as a graded module over
$\overline{\mathcal A}$.  Using again \ref{Levin} and the hypothesis
$\nu(\overline{\m}\overline{\mathcal A})=0$, we obtain
$\nu(\overline{\m}\overline{\mathcal D})=0$. We use then
\eqref{implication} and $x_s\m^2=0$ to deduce:
$$
Z(\m^2\mathcal D)\subseteq B(\m\mathcal D)+x_s\m\mathcal
D=B(\m\mathcal D)+Z(x_s\m\mathcal D)\,.
$$
Since $\nu(x_s\mathcal D)=0$, we see then
$Z(\m^2\mathcal D)\subseteq B(\m\mathcal D)$, hence $\nu(\m\mathcal D)=0$.
\end{proof}

\begin{comment}
\begin{remark}
\label{special-case-prop} A special case of the Proposition is when
$\mathcal D=\mathcal E$, which occurs when $\Upsilon=\Lambda$.  In this case, an
analysis of the proof of the proposition shows that the assumption
$\nu(x_s\mathcal E)=0$ can be replaced with the assumption that the
map $\HH(x_s\m\mathcal D)\to \HH(\m\mathcal D)$ induced by the
inclusion $x_s\m\subseteq \m$ is zero. Since $x_s\m^2=0$, this
assumption is equivalent to the fact that every element in
$x_s\m\mathcal D$ is a boundary of $\m\mathcal D$.
\end{remark}
\end{comment}

\begin{proposition}
\label{1} Let $\b\subseteq \m$  be an ideal  such that
$\m\b=x_t\b=x_s\b$ and $x_tx_s\in \ann(\b)\m$. Set $\mathcal
E={\mathcal K}\langle \Upsilon\mid \partial (\Upsilon)=x_tX_s-K\rangle$, where
$K\in \ann(\b)\mathcal K_1$ is such that $\partial(x_tX_s-K)=0$.

 Then    $\m^2\b=0$ and $\nu(\b \mathcal E)=0$.
\end{proposition}

\begin{proof} First, note that there exists $K\in \ann(\b)\mathcal K_1$ such that $\partial(x_tX_s-K)=0$. Indeed, we write $x_tx_s=\sum_{i=1}^e a_ix_i$ with $a_i\in \ann(\b)$ and we can take $K=\sum_{i=1}^e a_iX_i$.

 Let $K\in \ann(\b)\mathcal K_1$ be such that $\partial(x_tX_s-K)=0$.  We have $$\m^2\b=\m(\m\b)=\m(x_s\b)=x_s(\m\b)=x_s(x_t\b)=(x_sx_t)\b=0\,.$$

Let $b\in \b$. Since $b K=0$ and $X_s^2=0$, we have then for every
$m\ge 0$:
\begin{gather*}
\partial(b\Upsilon^{(m+1)})=b(x_tX_s-K)\Upsilon^{(m)}=bx_tX_s\Upsilon^{(m)}\,;\\
\partial(bX_s\Upsilon^{(m)})=-bX_s(x_tX_s-K)\Upsilon^{(m-1)}+bx_s\Upsilon^{(m)}=bx_s\Upsilon^{(m)}\,.
\end{gather*}
 Let $i\in \{0,1\}$ and $m\ge 0$.  Since $\m\b=x_s\b=x_t\b$, we see from the above that there exist then $i'\in \{0,1\}$, and $m'\in \{m, m+1\}$  such that
\begin{equation}
\label{basis-induction} \m\b X_s^i\Upsilon^{(m)}\subseteq\partial(\b
X_s^{i'}\Upsilon^{(m')})\,.
\end{equation}

For every $i\in \{0,1\}$ and $m\ge 0$, $n\ge 0$ we  have then
\begin{align*}
\m\b X_s^i\Upsilon^{(m)}\mathcal E_{\leqslant n}& \subseteq B(\b \mathcal E)+\b  X_s^{i'}\Upsilon^{(m')}\partial( \mathcal E_{\leqslant n})\\
&\subseteq  B(\b \mathcal E)+ \m \b X_s^{i'}\Upsilon^{(m')} \mathcal
E_{\leqslant n-1}
\end{align*}
where in the first line we used the Leibnitz rule and
\eqref{basis-induction} and  in the second line we used the fact
that $\mathcal E$ is minimal. Arguing by induction on $n$ and noting
that the base case $n=-1$ holds trivially, we see
$$
\m\b X_s^i\Upsilon^{(m)}\mathcal E_{ \leqslant  n}\subseteq B(\b \mathcal E)
$$
for all $n$, $i$ and $m$ as above. We conclude $\m\b\mathcal
E\subseteq B(\b \mathcal E)$, hence $\nu(\b\mathcal E)=0$.
\end{proof}

\begin{corollary}
\label{use-Golod} Assume $R$ is a standard graded quadratic algebra
and  the following hold: $x_u\m=(x_ux_s)=(x_ux_t)$, $x_sx_t\in
\ann(x_u)\m$. Set $\mathcal D=\mathcal K\langle \Upsilon\mid
\partial(\Upsilon)=x_tX_s-K\rangle$, where $K\in \ann(x_u)\mathcal K_1$ is
such that $\partial(x_tX_s-K)=0$.

If the ring $R/(x_u)$ is Golod and Koszul, then $\nu(\m\mathcal
D)=0$.
\end{corollary}

\begin{proof}
We apply Proposition \ref{1} with $\b=(x_u)$ to get $\nu(x_u\mathcal
D)=0$. We adopt the notation in \ref{reduction-x} with $x=x_u$.
 Since $\ov{\mathcal K^u}$ is the corresponding Koszul complex of the Golod Koszul ring $\ov R=R/(x_u)$,  Remark \ref{Golod-ring} gives $\nu(\overline{\m}\overline{\mathcal{K}^u})=0$.  We apply Proposition \ref{reduction-proposition} with $s=u$, $\mathcal E=\mathcal D$ and $\mathcal A=\mathcal K^u$ to conclude that $\nu(\m \mathcal{D})=0$.
\end{proof}

\begin{proposition}
\label{mix} Assume $x_sx_t=0$. Let $\mathcal A$ be a minimal
semi-free $\Gamma$-extension of $R$. Set
$$\mathcal E=R\langle X_s,\Lambda\mid \partial(X_s)=x_s, \partial(\Lambda)=x_tX_s\rangle\quad\text{and}\quad \mathcal D=\mathcal E\otimes_R\mathcal A\,.$$
Let $\a\subseteq \m$ be an ideal such that for any $x\in
\{x_t,x_s\}$ the following hold:
\begin{enumerate}[\quad\rm(1)]
\item $\m^2=x\m+\a\m$;
\item $x\a=0$;
\item $\nu(\ov\a\ov{\mathcal A})=0$,
where  overbars denote the corresponding objects  modulo $(x)$ as in
{\rm \ref{reduction-x}}.
\end{enumerate}
 Then  $\nu(\m\mathcal D)=0$.
\end{proposition}

\begin{proof} Using (1),  (2) and the assumption  $x_sx_t=0$,  we see $x\m^2=0$ for $x\in\{x_t,x_s\}$. Indeed, without loss of generality, take $x=x_s$. We have then $x_s\m^2=x_s(x_t\m+\a\m)=0$.

With $x\in \{x_s, x_t\}$, hypothesis (1) further gives:
\begin{equation}
\label{reformulate1} Z(\m^2\mathcal D)=x\m\mathcal D+Z(\a\m\mathcal
D)\,.
\end{equation}

Let  $w=\sum_i\varepsilon_i \otimes c_i\in \a\m\mathcal D$ with
$c_i\in\a\m\mathcal A$, where the elements $\varepsilon_i$ form a
basis of $\mathcal E$ over $R$. (The  sum has only finitely many
non-zero terms, for degree reasons.) Note that
$\partial(\varepsilon_i)\in (x_s,x_t)\mathcal E$. We have
$(x_s,x_t)c_i=0$ for all $i$ by the hypothesis (2), hence
$\partial(w)=\sum_i\pm \varepsilon_i \otimes \partial(c_i)$. Since
$\mathcal A$ is free as an $R$-algebra,  we have $\partial(w)=0$ if
and only if $\partial(c_i)=0$ for all $i$. We have thus
\begin{equation}
\label{ij-first}
 Z(\a\m\mathcal D)=\mathcal E\otimes_RZ(\a\m\mathcal A)\,.
\end{equation}
In particular, one can see from \eqref{reformulate1} and \eqref{ij-first}
that we have for all $n$:
\begin{equation}
\label{reformulate2} Z_n(\m^2\mathcal D)=\sum_{i+j=n}(\m^2\mathcal
E_i\otimes_R \mathcal A_j)\cap Z(\m^2\mathcal
D)=\sum_{i+j=n}\left(x\m\mathcal E_i\otimes_R\mathcal A_j+\mathcal
E_i\otimes_RZ_j(\a\m\mathcal A)\right)
\end{equation}
where we are using the following notation convention: If $M$ is a submodule of $\mathcal E_i$ and $N$ is a submodule of $\mathcal A_j$, then we identify $M\otimes_R\mathcal A_j$ and  $\mathcal E_i\otimes_RN$ with their homomorphic  images in $\mathcal E_i\otimes_R\mathcal A_j$. (Recall that $\mathcal E_i$ and $\mathcal A_j$ are free $R$-modules.) In particular, if $\b$ is an ideal of $R$, we have $\mathcal E_i\otimes_R\b\mathcal A_j=\b(\mathcal E_i\otimes_R\mathcal A_j)=\b \mathcal  E_i\otimes_R\mathcal A_j$.

Note that  $x_tX_s\Lambda^{(k)}=\partial(\Lambda^{(k+1)})$ for all integers
$k$, and hence $x_t\mathcal{E}_{2k+1}= B_{2k+1}(\mathcal{E})$. Also,
since $X_s^2=0$ we see that  $x_s\Lambda^{(k)}=\partial(X_s\Lambda^{(k)})$, thus
$ x_s\mathcal{E}_{2k}= B_{2k}(\mathcal{E})$. We set $r(i)=s$ if $i$
is an even integer and $r(i)=t$ if $i$ is odd.  Therefore we have
\begin{equation}
\label{r} x_{r(i)}\mathcal{E}_i=B_i(\mathcal{E})\quad\text{for all
$i$.}
\end{equation}
Let $i\ge 0$, $j\ge -1$. We use \eqref{r}, the Leibnitz rule and the
minimality of  $\mathcal A$ to get:
\begin{equation}
\label{comp1}
x_{r(i)}\m\mathcal E_i\otimes \mathcal A_j\subseteq B_i(\m\mathcal E)\otimes_R\mathcal A_j\subseteq B_{i+j}(\m\mathcal D)+\m^2\mathcal E_{i+1}\otimes_R \mathcal A_{j-1}\,.
\end{equation}
 Using \eqref{implication} and the hypothesis (3), we have  $Z(\a\m\mathcal A)\subseteq B(\a\mathcal A)+x\m\mathcal A$.
This  justifies the first line below:
\begin{align}
\label{comp2}
\begin{split}
\mathcal E_i\otimes_RZ_{j}(\a\m\mathcal A)&\subseteq \mathcal E_i\otimes_RB_{j}\left(\a\mathcal A\right)+\mathcal E_i\otimes_Rx_{r(i)}\m\mathcal A_{j}\\
&\subseteq B(\mathcal E\otimes_R \a\mathcal A)+\a B_{i-1}(\mathcal E_i)\otimes_R\mathcal A_{j+1}+x_{r(i)}\m\mathcal E_i\otimes_R \mathcal A_{j}\\
&\subseteq B(\m\mathcal D)+\a x_{r(i-1)}\mathcal E_{i-1}\otimes_R\mathcal A_{j+1}+x_{r(i)}\m\mathcal E_i\otimes_R \mathcal A_{j}\\
&=B(\m\mathcal D)+\m^2\mathcal E_{i+1}\otimes_R \mathcal
A_{j-1}.
\end{split}
\end{align}
We also used the Leibnitz rule in the second line, \eqref{r} in the
third, and the  hypothesis (2) and \eqref{comp1} in
the last line.

Putting together  \eqref{reformulate2} with $x=x_{r(i)}$,
\eqref{comp1} and \eqref{comp2} we have:
\begin{align*}
(\m^2\mathcal E_i\otimes_R\mathcal A_{j})\cap Z(\m^2\mathcal D)&=x_{r(i)}\m\mathcal E_i\otimes_R\mathcal A_{j}+\mathcal E_i\otimes Z_{j}(\a\m\mathcal A)\\
&\subseteq  B(\m\mathcal D)+(\m^2 \mathcal E_{i+1}\otimes_R\mathcal
A_{j-1})\cap Z(\m^2\mathcal D)\,.
\end{align*}
We induct on $j$. The  inclusion  $(\m^2\mathcal
E_i\otimes_R\mathcal A_{j})\cap Z(\m^2\mathcal D)\subseteq
B(\m\mathcal D)$ obviously holds for $j=-1$ and all $i\ge 0$. Arguing now by
induction on $j$ we see that $(\m^2\mathcal E_i\otimes_R\mathcal
A_{j})\cap Z(\m^2\mathcal D)\subseteq B(\m\mathcal D)$  holds for
all $j\ge -1$ and all $i\ge 0$. We obtain $Z( \m^2\mathcal
D)\subseteq B(\m\mathcal D)$.
\end{proof}

\begin{corollary}
\label{apply-st} Assume $x_sx_t=0$ and $\m^2=x_s\m=x_t\m$.

If $\mathcal D=\mathcal K\langle \Lambda\mid \partial(\Lambda)=x_tX_s \rangle$,
then $\nu(\m\mathcal D)=0$.
\end{corollary}
\begin{proof}
Apply Proposition \ref{mix} with $\a=0$ and $\mathcal A=\mathcal
K^s$.
\end{proof}

In \cite{AIS},  an element $x\in \m$ such that $x\m=\m^2$ and $x^2=0$ is
called a {\it Conca generator}. If we take $s=t$ in the statement of Corollary \ref{apply-st}, then the hypothesis is that $x_s$ is a Conca generator. Together with Corollary \ref{tool}, Corollary \ref{apply-st} can be used to recover the statement of \cite[Theorem 1.4]{AIS}.

\begin{comment}
\begin{corollary}
\label{apply-ss} Assume $x_s^2=0$, $\m^2=x_s\m$.

If $\mathcal D=\mathcal K\langle \Lambda\mid \partial(\Lambda)=x_sX_s \rangle$,
then $\nu(\m\mathcal D)=0$.
 \qed
\end{corollary}

Together with the discussion in Section 1, this corollary can be
used to recover the statement of \cite[Theorem 1.4]{AIS}. In {\it
loc.\,cit.} an element $x\in \m$ such that $x\m=\m^2$ and $x^2=0$ is
called a {\it Conca generator}.
\end{comment}

\begin{lemma}
\label{ann-connect} Assume $\a$ and $\b$ are ideals contained in
$\m$  such that:
\begin{enumerate}[\quad\rm(1)]
\label{ann-connect-hypothesis1}
\item $\b\m^2=0$ and $\m^2=\b\m+x_s\a$;
\item $\ann(x_s)\cap \m^2\subseteq \b\m$.
\end{enumerate}
If $\mathcal D$ is a minimal semi-free $\Gamma$-extension of
$R\langle X_s\mid \partial(X_s)=x_s\rangle$, then the following
hold:
\begin{enumerate}[\quad\rm(a)]
\item $Z(\m^2{\mathcal D})\subseteq B(\a{\mathcal D})+\b\m\mathcal D$;
\item If $\b=(x_t)$, then $\nu(\ov \a\ov{\mathcal D})=0$, where overbars denote classes modulo $(x_t)$ as in {\rm \ref{reduction-x}};
\item If  $\nu(\b\mathcal D)=0$ then $\nu(\m \mathcal D)=0$.
\end{enumerate}
\end{lemma}

\begin{proof}
We first show that the hypothesis (1)  implies
\begin{equation}
\label{ann-connect-formula}Z_i(\m^2{\mathcal D})\subseteq
B_i(\a{\mathcal D})+\left(\ann(x_s)\cap \m^{2}\right){\mathcal
D}_i+\b\m\mathcal D_i\quad\text{for all $i\ge 0$.}
\end{equation}
Let  $k\in \m^2{\mathcal D}_i$ be a cycle. We  write
$$k=k'+x_sk''\quad\text{with}\quad k'\in \b\m {\mathcal D}_i\quad\text{and}\quad k''\in \a{\mathcal D}_i\,.
$$
 Since $\b\m^2=0$ and ${\mathcal D}$ is minimal, we see that $k'$ is a cycle, hence $x_sk''$ must be a cycle as well. We have then $x_s\partial(k'')=0$. Since ${\mathcal D}_{i-1}$ is a free $R$-module and $\partial(k'')\in \m^{2}{\mathcal D}_{i-1}$, we conclude  $\partial(k'')\in \left (\ann(x_s)\cap \m^2\right ){\mathcal D}_{i-1}$.  We have then
$$
x_sk''=\partial(X_s)k''=\partial(X_sk'')+X_s\partial(k'')\in
\partial(\a{\mathcal D}_{i+1})+\left(\ann(x_s)\cap
\m^{2}\right){\mathcal D}_i\,,$$ where the second equality comes
from the Leibnitz formula. This  establishes
\eqref{ann-connect-formula}. If we further employ the hypothesis (2), we obtain (a). Part (b) follows
directly from (a).

To  explain (c), note that the hypothesis $\nu(\b\mathcal D)=0$ and the hypothesis $\b\m^2=0$
imply $$\b\m\mathcal D=Z(\b\m{\mathcal D})\subseteq B(\b{\mathcal
D})\,.$$ In  conjunction with (a), this shows $\nu(\m\mathcal
D)=0$.
\end{proof}

\begin{proposition}
\label{useful} Assume $\b$ is an ideal contained in $\m$
satifying the following conditions: $\m^2=\b\m+x_s\a$,   $\b\m=x_t\b$, $\ann(x_s)\cap\m^2\subseteq \b\m$ and $x_t^2\in
\ann(\b)\m$.

Let $\mathcal D=\mathcal K\langle \Upsilon \mid
\partial(\Upsilon)=x_tX_t-K\rangle$, where $K\in\ann(\b)\mathcal K_1$ is
such that $\partial(x_tX_t-K)=0$.

 Then  $\nu(\m\mathcal D)=0$.
\end{proposition}

\begin{proof}
Use Proposition \ref{1} to see $\nu(\b\mathcal D)=0$ and
$\b\m^2=0$, then apply Lemma \ref{ann-connect}(c).
\end{proof}

\begin{proposition}
\label{nonA-use} Assume $\m^2=x_t\m+(x_s^2)$ and $x_t^2=0=x_sx_t$.

 If $R$ is not artinian and  $\mathcal D=\mathcal K\langle \Lambda\mid \partial(\Lambda)=x_tX_t\rangle$, then $\nu(\m\mathcal D)=0$.
\end{proposition}

\begin{proof}
Observe first   $x_t\m^2=x_t^2\m+x_t(x_s^2)=0$. We further have
$$
\m^3=x_t\m^2+x_s\m^2=x_s\m^2=x_s\left(x_t\m+(x_s^2)\right)=(x_s^3)\,.$$
An induction argument yields
\begin{equation}
\label{x2} \m^i=(x_s^i)\quad\text{for all}\quad i>2\,.
\end{equation}
 Since $R$ is not artinian, we have $x_s^{i}\ne 0$ for all $i$.
It follows that  $\ann(x_s)\cap \m^3=0$.

Let $a\in \ann(x_s)\cap\m^2$. We have  $a=x_ty+\beta x_s^2$ with
$y\in \m$ and $\beta \in R$.  Then  $ax_s=\beta x_s^3$, hence
$\beta x_s^3=0$. Since $x_s^3\ne 0$, it follows that $\beta\in \m$. Since $\ann(x_s)\cap \m^3=0$, we conclude $\beta x_s^2=0$. We have thus $a=x_ty\in
x_t\m$. We proved thus the inclusion  $\ann(x_s)\cap \m^2\subseteq
(x_t)\m$.

Consider the DG-$\Gamma$ algebras
$$
\mathcal E=R\langle X_t, \Lambda \rangle\quad\text{and}\quad \mathcal
A=\mathcal K^t
$$
where $\partial(\Lambda)=x_tX_t$. We apply Lemma \ref{ann-connect}(b) with
$\b=(x_t)$ and $\a=(x_s)$  to conclude that $\nu(\ov x_s\ov{\mathcal
A})=0$, where overbars denote the corresponding classes modulo
$(x_t)$. Note that $\mathcal D$ can be identified with $\mathcal
E\otimes_R\mathcal A$. We have $\m^2=x_t\m+(x_s)\m$ and
$x_t^2=0=x_t(x_s)$. We apply Proposition \ref{mix} with $\a=(x_s)$
to conclude  $\nu(\m\mathcal D)=0$.
\end{proof}

\begin{proposition}
\label{newp}
Assume that the following conditions  hold:
\begin{enumerate}[\quad\rm(1)]
\item $\m^2=x_s\m$ and $x_sx_t=0$;
\item $ \ann(x_s)\cap \m^2\subseteq x_t\m$;
\item $x_i\m\subseteq x_t\m$ for all $i\neq s$.
\end{enumerate}
Then $\nu(\m \mathcal D)=0$, where $\mathcal D=\mathcal K\langle \Lambda\mid \partial(\Lambda)=x_tX_s\rangle$.
\end{proposition}

\begin{proof}
Using the hypotheses (1) and (3)  we get $x_t\m^2=0$ and  $\m^2=x_t\m+(x_s^2)$.
In view of (2) as well, the hypotheses of Lemma \ref{ann-connect}  hold with $\a=(x_s)$ and $\b=(x_t)$ and we have
$Z(\m^2{\mathcal D})\subseteq B(x_s{\mathcal D})+x_t\m\mathcal D$.
It suffices thus to show $x_t\m\mathcal D\subseteq B(\m{\mathcal D})$.

Let $n\ge -1$ and $\m\ge 0$. Let $k'\in (\mathcal K^s)_{n-1}$ and $a\in \m$. The Leibnitz rule gives
$$x_tak'X_s\Lambda^{(m)}=\partial(ak'\Lambda^{(m+1)})\pm a\partial(k')\Lambda^{(m+1)}\,.$$
 Since $\partial(k')\in \sum_{i\ne s} x_i\mathcal K^s$, hypothesis (3) gives $a\partial(k')\in x_t\m(\mathcal K^s)_{n-2}$ and thus
\begin{equation}
\label{new1}
x_t\m X_s(\mathcal K^s)_{n-1}\Lambda^{(m)}\subseteq B(\m\mathcal D)+x_t\m (\mathcal K^s)_{n-2}\Lambda^{(m+1)}\,.
\end{equation}
In view of (1), there exists $b\in \m$ such that $x_ta=x_sb$. If $k''\in (\mathcal K^s)_n$, we have
 $$x_tak''\Lambda^{(m)}=x_sbk''\Lambda^{(m)}=\partial(bX_sk''\Lambda^{(m)})\pm bX_s\partial(k'')\Lambda^{(m)}\,.$$
As above, hypothesis (3) gives  $b\partial(k'')\in x_t\m(\mathcal K^s)_{n-1}$, hence
\begin{equation}
\label{new2}
x_t\m(\mathcal K^s)_{n}\Lambda^{(m)}\subseteq B(\m\mathcal D)+x_t\m X_s(\mathcal K^s)_{n-1}\Lambda^{(m)}\,.
\end{equation}
Since $x_t\m\mathcal K_{n}=x_t\m(\mathcal K^s)_{n}+x_t\m X_s(\mathcal K^s)_{n-1}$, equations \eqref{new1} and \eqref{new2} give
$$
x_t\m \mathcal K_{n}\Lambda^{(m)}\subseteq B(\m\mathcal D)+x_t\m \mathcal K_{n-2}\Lambda^{(m+1)}\,.
$$
Induction on $n$ yields then the inclusion $x_t\m\mathcal K_n \Lambda^{(m)}\subseteq  B(\m\mathcal D)$ for all $n$ and $m$. (For the base case, note that the inclusion holds trivially for all $m$ when $n\le -1$.) We have thus $x_t\m\mathcal D\subseteq B(\m\mathcal D)$ and this finishes the proof.
\end{proof}

\begin{corollary}
\label{special-case}  Assume that $R$ is graded, and either $R$ is not artinian or $R$ is quadratic. Assume further that
$$x_1^2=x_1x_2=x_2^2-x_1x_3=x_2x_3=x_i\m=0\quad\text{for all} \quad i\ge 4\,.
$$

 Then  $\nu(\m\mathcal D)=0$, where $\mathcal D=\mathcal K\langle
\Lambda\mid \partial(\Lambda)=x_2X_3\rangle$.
%and the induced map $Q/(\wt x_2\wt
%x_3)\to R$ is a Golod homomorphism.
\end{corollary}

\begin{proof}
Set $s=3$ and $t=2$. We check that the hypotheses of  Proposition \ref{newp} are satisfied.

We have $\m^2=(x_2^2,x_3^2)=(x_1x_3,x_3^2)=x_3\m$, and this implies
\begin{equation}
\label{mi}
\m^i=x_3^{i-2}\m^2=x_3^{i-2}(x_2^2,x_3^2)=(x_3^i)\quad \text{for all}\quad i\ge 3\,.
\end{equation}

Since $\m^2=x_3\m$ and $x_2x_3=0$, hypothesis (1) of Proposition \ref{newp} is satisfied.

Since $x_1\m=(x_1x_3)$, $x_2\m=(x_2^2)=(x_1x_3)$, and $x_i\m=0$ for all $i\ge 4$, we see that hypothesis (3) is also satisfied.

It remains to verify hypothesis (2), which reads
$$
\ann(x_3)\cap \m^2\subseteq x_2\m\,.
$$

First, assume that $R$ is not artinian. In particular, equation \eqref{mi} shows that $x_3^i\ne 0$ for all $i$. Take a homogeneous element $u\in \ann(x_3)\cap \m^2$. If $\deg u=i\ge 3$, then $u=\alpha x_3^i$, with $\alpha \in \kk$. As $0=ux_3=\alpha x_3^{i+1}$, we get $\alpha=0$, and hence $u=0$. If $\deg u=2$, we write $u=\alpha x_3^2+\beta x_2^2$ with $\alpha,\beta \in \kk$. We have then $x_3u=\alpha x_3^3$, so $\alpha=0$ and hence $u=\beta x_2^2\in x_2\m$, as desired.

Assume now that $R$ is quadratic. In view of the previous case, we may assume $R$ is artinian. Set $\ov R=R/(x_1, x_2, x_4, \dots, x_e)$.  The ring $\ov R$ can be presented as $\ov R=\kk[\wt x_3]/J$, where $\wt x_3$ is a polynomial variable that corresponds to the image of $x_3$ in $\ov R$, and $J$ is a quadratic ideal. Since $R$ is artinian, we must have $J=(\wt x_3^2)$, and hence
$x_3^2\in (x_1, x_2, x_4, \dots, x_e)\m$. Since $(x_1, x_2, x_4, \dots, x_e)\m=(x_1,x_2)\m=x_2\m$, we conclude $x_3^2\in x_2\m$, and then $\m^2=x_2\m$. This proves hypothesis (2), as well.
\end{proof}

\begin{corollary}
\label{last-one} Assume that $R$ is graded, and either $R$ is not artinian or $R$ is quadratic. Assume further that $e=4$ and
$$x_1^2=x_1x_4=x_1x_2=x_2x_4=x_2x_3=x_2^2-x_3x_4=x_4^2-x_1x_3=0\,.$$
 Then  $\nu(\m\mathcal D)=0$, where $\mathcal D=\mathcal K\langle
\Upsilon\mid \partial(\Upsilon)=x_2X_2-x_4X_3\rangle$.
%and $Q/(\wt x_2^2-\wt x_3\wt
%x_4)\to R$ is a Golod homomorphism.
\end{corollary}
\begin{proof}   First, note that  $x_2^3=x_2x_3x_4=0$, hence, if $R/(x_2)$ is artinian, then $R$ is artinian. In view of the hypothesis, we conclude that $R/(x_2)$ is not artinian or $R/(x_2)$ is quadratic. 

Note that $\m(x_2)=(x_2^2)=x_2(x_2)$, $x_2^2=x_3x_4$ and
$x_4\in\ann(x_2)$.  We use Proposition \ref{1} with $\b=(x_2)$,
$x_s=x_t=x_2$ and $K=x_4X_3$ to see  $\nu(x_2\mathcal D)=0$ and
$x_2\m^2=0$.

Consider the ring $\ov R=R/(x_2)$ and let overbars denote the
corresponding objects modulo $(x_2)$. We have $\ov \m^2=(\ov x_1\ov
x_3, \ov x_3^2)$ and the relations
${\ov x_1}^2=\ov x_1\ov x_4=\ov x_3\ov x_4={\ov x_4}^2-\ov x_1\ov x_3=0$.
Corollary \ref{special-case} gives $\nu(\overline {\m}\mathcal
A)=0$, where $\mathcal A=\ov{\mathcal K^2}\langle \Upsilon'\mid
\partial(\Upsilon')=\ov x_4\ov X_3\rangle$.

Note $\ov{\mathcal D}=\mathcal A\langle X \rangle$ with $X=\ov X_2$ and $\partial(X)=0$. Since $\nu(\overline
{\m}\mathcal A)=0$ we see $\nu(\ov \m\ov{\mathcal D})=0$ by
\ref{Levin}. We also use \eqref{implication} and the facts
$x_2\m^2=0$ and $\nu(x_2\mathcal D)=0$ to conclude
$$
Z(\m^2\mathcal D)\subseteq B(\m\mathcal D)+x_2\m\mathcal
D=B(\m\mathcal D)+Z(x_2\m\mathcal D)\subseteq B(\m\mathcal D)\,,
$$
hence $\nu(\m\mathcal D)=0$.
%The  last conclusion follows from Corollary \ref{tool}.
 \end{proof}

\section{Artinian quadratic $\kk$-algebras with  $\dim_{\kk}R_2\le 3$}
\label{artinian}

Let $\kk$ denote a field, and consider the polynomial ring
$Q=\kk[\wt x_1, \dots, \wt x_e]$ with variables $\wt x_i$ in degree
$1$. Throughout this section, $R$ denotes a standard graded
quadratic $\kk$-algebra with $R=Q/I$ for a homogenous ideal $I$
generated by quadrics.

For each $i$ we let $x_i$ denote the image of
$\wt x_i$ in $R$ and we let $\m$ denote the maximal homogeneous
ideal of $R$. In general, if $a\in R$, then $\wt a$ denotes a preimage of $a$ in $Q$; unless otherwise indicated (such as in the case of $\wt x_i$) the choice of this preimage is not restricted.  If $x\in R_1$, we denote by $\rank(x)$ the $\kk$-vector space
dimension of the image of the map $R_1\to R_2$ given by
multiplication by $x$.

In this section we prove:

\begin{theorem}
\label{artinian-case} Let $R$ be a standard graded artinian quadratic
$\kk$-algebra with $\dim_{\kk}R_2\le 3$. Assume that $\kk$ is algebraically closed.

The  following then hold:
\begin{enumerate}[\quad\rm(a)]
\item $R_4=0$ and $\dim_{\kk} R_3\le 1$. Furthermore, $\dim_{\kk}R_3=1$ if and only if $R$ is a trivial fiber extension of a quadratic complete intersection of embedding dimension $3$;
\item  $R$ is Koszul and there exists a surjective Golod homomorphism $P\to R$ of graded  $\kk$-algebras, with $P$ a quadratic complete intersection of codimension at most $3$. In particular, $R$ is absolutely Koszul.
\end{enumerate}
\end{theorem}

The  proof of this result is given at the end of the section.  We
establish first needed structural results. We follow some of the
techniques in \cite{C1} and \cite{C2} and we develop new
considerations in order to handle the case $\ch\kk=2$. The treatment of the case $\ch\kk=2$ constitutes roughly half of the proof.

\begin{chunk}
\label{conditions} We consider the following conditions on the
minimal generating set $x_1, \dots, x_e$ of $\m$:
\begin{enumerate}[\quad\rm(1)]
\item $\m^2=x_1\m$ and $x_1^2=0$;
\item $\m^2=x_1\m+x_2\m$, $x_1^2=0$, $x_2^2\in x_1\m$ and  $x_1\m=(x_1x_j)$, where $j=2$ if $x_1x_2\ne 0$ and $j=3$  if $x_1x_2=0$;
\item $\m^2=x_1\m+x_2(x_3)$, $x_1^2=0=x_1x_2$ and $x_2^2\in x_1\m$;
\item $\m^2=x_1\m+x_2(x_3)$, $x_1^2=0=x_1x_3$ and $x_2^2\in x_1\m$;
\item $\m^2=x_1\m+x_2(x_3)$, $x_1^2=0=x_1x_4$, $x_2^2\in x_1\m, x_2x_4\in x_1\m$ and $x_4^2-x_2x_3\in x_1\m$;
 \item $\m^2=x_1\m+x_2\m$, $x_1^2=0=x_1x_3$, $x_2^2\in x_1\m$,  $x_3\m=(x_3^2)$,  $x_1\m=(x_1x_j, x_3^2)$ where $j=4$ if  $x_1x_2\in x_3\m$ and $j=2$ if $x_1x_2\notin x_3\m$;
\item $\m^2=x_1\m+x_2\m+x_3\m$,  $x_1^2=0$, $x_2^2\in x_1\m$, $x_3^2\in (x_1, x_2)\m$, $x_1\m=(x_1x_i)$ for some $i\ne 1$ and $(x_1,x_2)\m= (x_1x_i, x_2x_j)$ with $x_2x_j\notin (x_1x_i)$ for some  $j\notin\{1,2\}$;
\item $\m^2=x_1\m=x_2\m$ and  $x_1x_2=0$.
\end{enumerate}
\end{chunk}

\begin{proposition}\label{codim1-2structure} Let $R$ be a standard graded quadratic $\kk$-algebra. Assume $R$ is artinian and $\kk$ is algebraically closed.

The  following hold:
\begin{enumerate}[\quad\rm(a)]
\item If $\dim_{\kk}R_2=1$, then {\rm \ref{conditions}(1)} holds, up to a change of variables.
\item If  $\dim_{\kk}R_2=2$ then {\rm\ref{conditions}(1)} or {\rm\ref{conditions}(2)} holds, up to a change of variables.
\item If $\dim_{\kk}R_2=3$ then either one of the conditions {\rm\ref{conditions}(1)--(8)} holds, up to a change of variables, or else $R$ is a trivial fiber extension of a complete intersection of embedding dimension $3$.
\end{enumerate}
\end{proposition}

\begin{remark}
Using the terminology discussed after Corollary \ref{apply-st}, the
fact that \ref{conditions}(1) holds up to a change of variables is
equivalent to saying that $R$ has a Conca generator.

Assume the hypothesis of the proposition holds. If $\ch\kk\ne 2$,
then a statement stronger than (b) holds, namely:
\begin{enumerate}
\item[(b$'$)] If  $\dim_{\kk}R_2=2$ then $R$ has a Conca generator.
\end{enumerate}
This   statement can be deduced from the  proof of \cite[Proposition
6]{C1}. However (b$'$) does not hold when  $\ch\kk=2$.  Indeed,
assume $\ch\kk=2$ and let $$R=\kk[\wt x_1,\wt x_2,\wt  x_3]/ ({\wt
x_1}^2,\wt x_1 \wt x_2, {\wt x_3}^2-\wt x_1\wt x_3,{\wt x_2}^2-\wt
x_2\wt x_3)\,.$$ In this ring we have $\m^2=(x_2^2,x_3^2)$, and
$x_2^2, x_3^2$ are linearly independent.  Assume there exists $x\in
R_1$ non-zero such that $x^2=0$ and $\m^2=x\m$. Let $\alpha, \beta,
\gamma\in\kk$ so that  $x=\alpha x_1+\beta x_2+\gamma x_3$. Since
$x^2=0$
 and $\ch\kk=2$, we have  $\beta^2 x_2^2+\gamma^2 x_3^2=0$ and hence $\beta=\gamma=0$,  thus $x=\alpha x_1$.  Since
 $x_2^2\notin x_1\m$, this contradicts the assumption that $\m^2=x\m$.
\end{remark}

We will use (repeatedly) the following result recorded in \cite[Lemma
3]{C1}:

\begin{bfchunk}{Existence of null-square linear forms.}
\label{conca-element} Assume $\kk$ is algebraically closed and let
$U\subseteq R_1$  be a subspace. If $\dim_{\kk} U>\dim_{\kk} U^2$,
then there exists  $x\in U$ with $x\ne 0$ and such that $x^2=0$.
\end{bfchunk}

\noindent{\it Proof of Proposition {\rm \ref{codim1-2structure}}.}

If $R$ is a trivial fiber extension of an algebra $R'$, note that
the conclusions of the proposition do not change when replacing $R$
with $R'$. We may assume thus that $R$ has no socle elements in
degree $1$, hence $xR_1\ne 0$ for all $x\in R_1$ with $x\ne 0$.

 (a) Assume $\dim_{\kk}R_2=1$. Note that $e>1$. Using \ref{conca-element}, we see that there exists $x\in R_1$ non-zero with $x^2=0$. Since $xR_1\ne 0$, we have then $xR_1=R_2$ and thus $\m^2=x\m$. We can take $x_1=x$ after a change of variables and then \ref{conditions}(1) holds.

(b) Assume $\dim_{\kk}R_2=2$. Note that  $e>2$. Using again
\ref{conca-element}, we see that there exists $x\in R_1$ non-zero
with $x^2=0$. If $\rank(x)=2$, then \ref{conditions}(1) holds with $x_1=x$.

Assume $\rank(x)=1$. The  ring $\ov R=R/(x)$ is an artinian quadratic
$\kk$-algebra with $\dim_{\kk}\ov R_2=1$. We let overbars denote the
corresponding objects modulo $(x)$. Applying (a), there exists
$y\in R_1$ with $\ov y\ne 0$ such that $\ov y^2=0$ and $\ov y\,\ov
\m=\ov \m^2$.  We have thus $\m^2=x\m+y\m$. We make a
change of variables so that $x=x_1$ and $y=x_2$. Recall that
$\rank(x_1)=1$, so that $x_1\m=(x_1z)\ne 0$ for some $z\in R_1$. If
$x_1x_2\ne 0$, then we see that \ref{conditions}(2) holds with
$j=2$. If $x_1x_2=0$, note that $x_1, x_2,  z$ are linearly
independent. We take then $z=x_3$ and see that  \ref{conditions}(2)
holds with $j=3$.

(c) Assume $\dim_{\kk}R_2=3$. Recall  $R=Q/I$ and $I$ is
generated by quadrics. If $e=3$, then the ideal $I$ is minimally
generated by $3$ quadrics and $\dim Q=3$. Since $\dim R=0$, the
ideal $I$ has height $3$, and it is thus generated by a regular
sequence. The  ring $R$ is a complete intersection in this case.

Assume now $e>3$. By \ref{conca-element}, there exists $x\in R_1$
non-zero such that $x^2=0$. After a change of variables, we assume
$x_1=x$. We distinguish between the possible values of $\rank(x)$. Most of the work needs to be done when $\rank(x)=2$, so we treat the easier cases first.

If $\rank(x)=3$ then \ref{conditions}(1) holds.

Assume $\rank(x)=1$ and  consider the ring $\ov R=R/(x)$, which has
$\dim_{\kk}\ov R_2=2$.  We apply part (b) to the ring $\ov R$
to see that $\ov R$ satisfies \ref{conditions}(1) or
\ref{conditions}(2), up to a change of variables.  If $\ov R$
satisfies \ref{conditions}(1) up to a change of variables, then
proceed as in the proof of (b) to conclude that $R$ satisfies
\ref{conditions}(2).  Assume now that $\ov R$ satisfies
\ref{conditions}(2), with $\ov x_2, \dots, \ov x_e$ playing the role
of the minimal generating set of $\ov\m$. We have thus
${\ov\m}^2=\ov x_2\,\ov \m+\ov x_3\,\ov\m$, with ${\ov x_2}^2=0$,
${\ov x_3}^2\in \ov x_2\,\ov\m $ and $\ov x_2\ov\m=(\ov x_2\ov x_j)$
for some $j\notin \{1,2\}$. Lifting back to $R$ we have
$$\m^2=x_1\m+x_2\m+x_3\m\,,\,\, x_1^2=0\,,\,\, x_2^2\in x_1\m\,,\,\, x_3^2\in
(x_1, x_2)\m\,,\,\, (x_1,x_2)\m=(x_2x_j)+x_1\m\,.$$
Since $\rank(x_1)=1$,
we have $x_1\m=(x_1z)$ for some $z\in R_1$. Since $x_1^2=0$ and $R$ has no socle elements in degree $1$, we may assume $z=x_i$ for some $i\ne 1$ and thus \ref{conditions}(7)
holds.

Assume now $\rank(x)=2$. We introduce the following notation:
$$
 W=\{r\in R_1\mid r\m\subseteq x\m\}\,,\quad\text{and}\quad V=\ann(x)\cap R_1\,.\quad
$$
Since $x\in W$, note that $(W)\m=x\m$. We will consider now two
cases: $V\not\subseteq W$ and $V\subseteq W$. The second case requires most of the work.

\begin{Case1} $V\not\subseteq W$.

 There  exists $y\in V$ with $y\m\not\subseteq x\m$.  We
have then $$\m^2=x\m+y\m\quad\text{and}\quad xy=0\,.$$

 If  $y^2\in x\m$, then, since $\dim_{\kk} R_2=3$ and $\rank(x)=2$, we see there exists an element $z\in R_1\smallsetminus(\kk x+\kk y)$ such that $m^2=x\m+(yz)$. Then \ref{conditions}(3)
holds after a change of variables so that  $x_2=y$ and $x_3=z$.

Assume now $y^2\notin x\m$.  Set $\ov R=R/(x)$. We denote by
overbars the corresponding objects  modulo $(x)$. The ring $\ov
R$ has $\dim_{\kk}(\ov R_2)=1$, hence $\ov\m^2=({\ov y}^2)$. By (a),
there is $z\in R_1$ such that ${\ov z}^2=0$ and $\ov{\m}^2=\ov
z\,\ov\m$, hence $y^2-zz'\in x\m$ for some $z'\in R_1$, and $z^2\in
x\m$.

If  $\ov y\,\ov z\ne 0$, we  have $yz\notin x\m$, hence
$\m^2=x\m+z(y)$. The  hypotheses that $y^2\notin x\m$, $z^2\in x\m$
and $yz\notin x\m$ imply that $x,y,z$ are linearly independent.  Then
\ref{conditions}(4) holds, after a change of variables with $x_2=z$
and $x_3=y$.

If  $\ov y\,\ov z=0$, then $yz\in x\m$ and we have $\m^2=x\m+z(z')$.
The  hypotheses that $y^2\notin x\m$, $zz'\notin x\m$, $z^2\in x\m$
and $yz\in x\m$ imply that $x,y,z,z'$ are linearly independent.  Then
\ref{conditions}(5) holds, after a change of variables with $x_2=z$,
$x_3=z'$ and $x_4=y$.
\end{Case1}

\begin{Case2} $V\subseteq W$.

Since $\rank(x)=2$, we have $\dim_{\kk}(V)=e-2$, hence
$\dim_{\kk}(W)\ge e-2$. We cannot have $\dim_{\kk}(W)=e$, since
$\m^2\ne x\m$. If $\dim_{\kk}(W)=e-1$, then $S=R/(W)$ is an artinian
quadratic algebra with $\dim_{\kk}S_1=\dim_{\kk}S_2=1$. This   is a
contradiction, since no such algebra exists, hence
$\dim_{\kk}(W)=e-2$ and thus $W=V$ and $(V)\m=(W)\m= x\m$.
One has
\begin{equation}
\label{condition} (\ann(x)\cap R_1)\m=x\m\,.
\end{equation}
The  algebra $S=R/(V)=R/(W)$ is an artinian quadratic algebra with
$\dim_{\kk}S_1=2$ and $\dim_{\kk}S_2=1$. We denote by overbars the
corresponding classes modulo $(V)$. By (a), we know that
\ref{conditions}(1) holds for $S$, hence there exists $u\in S_1$ with $\ov
u^2=0$ such that $\ov u\,\ov \m=\ov \m^2$  and there exists $z\in
R_1$ such that $\ov \m^2=(\ov u\,\ov z)\ne 0$. We have found thus
linearly independent elements $u,z\in R_1\setminus V$ such that
\begin{equation}
\label{uu'1} R_1=W\oplus \kk u\oplus \kk z\,, \quad u^2\in x\m,\quad
uz\notin x\m \quad\text{and}\quad z^2\in (uz)+x\m\,.
\end{equation}
We have  thus $\m^2=x\m +(uz)=(xu,xz,uz)$. If $\ch \kk\ne
2$, one may also  assume $z^2\in x\m$, after possibly renaming the
elements.

If $\ch\kk\ne 2$, the proof of \cite[Proposition 2.1; Case (3)]{C2}
shows that there exists a non-zero element $\ell\in R_1$ with
$\ell^2=0$ such that if $\rank(\ell)=2$ then $(\ann(\ell)\cap
R_1)^2\not\subseteq \ell\m$.  Replacing $x$ with $\ell$, the cases
when $\rank(\ell)=1$ or $\rank(\ell)=3$ have already been settled.
Assume thus $\rank(\ell)=2$. Since  $(\ann(\ell)\cap R_1)\m\ne
\ell\m$ in this case, we see that, when replacing $x$ with $\ell$,
we cannot be in Case 2, and hence we must be in Case 1 above,  which
is already settled. Thus, if $\ch\kk\ne 2$ then the proof ends here.

For the remainder of the proof we assume  $\ch\kk=2$. This  assumption will be heavily used in our
computations below. We further distinguish several subcases of Case
2. The  convention is that the hypotheses and assumptions made within
the discussion of a case get passed on to its subcases.

\begin{Case2.1} There  exists $y\in V\setminus \kk x$ such that
$y^2=0$.

If $y$ is of rank $1$ or $3$ we can use our earlier work, with $y$
playing the role of $x$.  We assume thus $\rank(y)=2$. Since $y\in
V=W$, we have $y\m\subseteq x\m$, hence $y\m=x\m$ since both $x$ and
$y$ have rank $2$.  Since Case 1 is already settled, we may also
assume that  we fall into Case 2 with $y$ playing the role of $x$,
and thus
$$\ann(y)\cap R_1=
\{r\in R_1\mid r\m\subseteq y\m\}= W=V\,.
$$
The  arguments above show that we may assume
\begin{equation}
\label{fe} \ann(y')\cap R_1=\ann(x)\cap R_1\quad\text{for all}\quad
y'\in V\setminus \kk x\quad\text{ such that}\quad  y'^2=0\,.
\end{equation}

 If $yu=\alpha xu$ for
some $\alpha\in \kk$, we have then $u(y-\alpha x)=0$. Taking
$y'=y-\alpha x$ we obtain a contradiction with \eqref{fe}, since
$ux\ne 0$,  $uy'=0$ and $y'^2=0$.

We  have thus $yu\notin(xu)$,  hence $x\m=y\m=(yu, xu)$ and
$\m^2=(yu,xu,uz)$. We conclude $\rank(u)=3$.   We recall that
$u^2\in x\m$. Let $\alpha, \beta\in \kk$ such that  $u^2=\alpha
yu+\beta xu$.  We have then $uu'=0$ where $u'=u-\alpha y-\beta x$.
Note that $u'x=ux$, $u'y=yu\notin(xu)$, $u'z=uz-\alpha yz-\beta x
z$ and $u'z\notin x\m$ because $uz\notin x\m$ from \eqref{uu'1} and
$yz\in y\m= x\m$. We have then $\rank(u')=3$. If $u$ and $u'$ are
linearly independent, then \ref{conditions}(8)  holds after a change
of variables so that $x_1=u$ and $x_2=u'$. Otherwise, we have $u^2=0$
and \ref{conditions}(1) holds with $x_1=u$.
\end{Case2.1}

\begin{Case2.2} $y^2\ne 0$ for all $y\in V\setminus \kk x$.

If $\rank(y)=1$ for some $y\in V\setminus \kk x$, then, since
$y^2\ne 0$ and $y^2\in x\m$, we have $\m^2=x\m+u\m$, $y\m=(y^2)$ and
$x\m=(xt, y^2)$ with $t=u$ or $t=z$. Recall $u^2\in x\m$.
Condition \ref{conditions}(6) is satisfied with $x_1=x$, $x_2=u$,
$x_3=y$ and $j=2$ if $y^2\notin (xu)$ or  $j=4$ and $x_4=z$
if $y^2\in (xu)$.

We assume now $\rank(y)>1$ for all $y\in V\setminus \kk x$. Note
that $\rank(y)\ne 3$ because $y\m\subseteq (V)\m=x\m$. We have thus
\begin{equation}
\label{ym} \rank(y)=2 \quad\text{and}\quad x\m=y\m\quad\text{for
all}\quad y\in V\setminus \kk x\,.
\end{equation}
Since $V^2\subseteq (V)\m=x\m$, we have $\dim_{\kk}V^2\le 2$. As
$\dim_{\kk}V=e-2$ and $e\ge 4$, note that \ref{conca-element}
implies $e= 4$.  Let  $y\in V\setminus \kk x$. Note that  $x,y,z,u$
are linearly independent,  hence $\m=(x,y,z,u)$.

We further distinguish four subcases of Case 2.2. Before we proceed,
we point out here one of our strategies.   Since $x,u,z,y$ are linearly independent, we can choose the
variables $\tilde x_1, \dots, \tilde x_4$ such that
$\{x,u,z,y\}=\{x_1,  \dots, x_4\}$, say $x_1=x$, $x_2=u$, $x_3=z$ and $x_4=y$. It is however cumbersome to work with subindices, so we will only spell out such choices at the end of each case. In what
follows we will perform suitable changes of variables, that we view as
renaming certain elements. For example, when we want to rename an
element $z'$ as $z$, we indicate the change of variables by writing
$z\leftarrow z'$; this notation indicates that $z'$ becomes the
new element $z(=x_3)$, and all other $x_i$ with $i\ne 3$ remain
the same.  Whenever such a change of variables is performed, we need
to be careful that the relevant hypotheses are preserved. For
example, in order for a replacement $z\leftarrow z'$ to
preserve the hypotheses introduced so far, one would need to
check $z'\in R_1\smallsetminus V$, the elements $u,z'$ are linearly
independent,  $uz'\notin x\m$ and $z'^2\in (uz')+x\m$.

Recall that $z^2\in (uz)+x\m$, hence $z^2-\theta uz\in x\m$ for some
$\theta \in\kk$. Since $uz\notin x\m$ by \eqref{uu'1}, we have  $\theta=0$ if and only if  $z^2\in x\m$.

\begin{Case2.2.1} $z^2\in x\m$.

Since  $u^2\in x\m$, $(W)\m=x\m$ and $R_1=W\oplus \kk u\oplus \kk z$, the
assumption that $\ch\kk=2$ implies that  $w^2\in x\m$ for all $w\in
R_1$. Since $\rank(x)=2$  we see that $y^2,u^2,z^2$ are not linearly
independent. Using again the fact that $\ch\kk=2$, we conclude
there exists $w=\alpha y+\beta u+\gamma z$ with $\alpha, \beta,
\gamma\in \kk$ such that $w\ne 0$  and  $w^2=0$. Since $w\notin \kk
x$, we must have $w\notin V$, in view of the hypothesis of Case 2.2,
hence $\beta\ne 0$ or $\gamma\ne 0$. Without loss of generality,
assume $\beta\ne 0$. After a replacement $u\leftarrow
\beta u$, we may assume $\beta=1$. Note that $wz=\alpha yz+uz+\gamma
z^2\notin x\m$ because $yz\in x\m$, $z^2\in x\m$ and $uz\notin x\m$.
We have thus $\m^2=(xw,xz,wz)$. Since $w^2=0$, in view of the work
above, the only case that needs to be considered is when
$\rank(w)=2$ and the hypotheses of  Case 2 hold, with $w$ playing
the role of $x$. We have thus $w\m=(xw,wz)$,  and $(\ann(w)\cap
R_1)\m=w\m$. Note that $wy\in w(V)\subseteq (V)\m=x\m$. We have
$$wy\in  (w\m\cap x\m)_2=(\kk xw\oplus \kk wz)\cap (\kk xw\oplus \kk xz)=\kk xw\,.$$

We  have thus $wy=\lambda wx$ for some $\lambda\in \kk$.  Set
$y'=y-\lambda x$.  We have then
$wy'=0$. Note that $y'\in V\setminus \kk x$, hence $\rank(y')=2$ and
$y'\m=x\m$ by \eqref{ym}.   We also have   $y'\m\subseteq
(\ann(w)\cap R_1)\m=w\m$, hence $y'\m=w\m$ because
$\rank(y')=2=\rank(w)$.  We conclude  $w\m=x\m$ and thus $wz\in
x\m$, a contradiction.
\end{Case2.2.1}

We now proceed with the cases when $\theta\ne 0$. After the replacement $u\leftarrow \theta u$ we may assume  $z^2-uz\in x\m$.

\begin{Case2.2.2} $z^2-uz\in x\m$,  $yu=0$ and   $y^2\notin (xu)$ for some $y\in V\smallsetminus \kk x$.

Recall that we also  have $u^2\in x\m$ and $uz\notin x\m$ from
\eqref{uu'1} and $x\m=y\m$ from \eqref{ym}.

We write  $y^2=\alpha xu+\varepsilon xz$ with $\alpha, \varepsilon
\in \kk$ and $\varepsilon \ne 0$. We may assume $\varepsilon=1$
after the change $x\leftarrow \varepsilon x$. We claim that
after the replacement  $z\leftarrow z+\alpha u$, all the
hypotheses of Case 2.2.2 still hold. Indeed, we only need to check
$(\alpha u+z)^2-u(\alpha u+z)\in x\m$ and $u(\alpha u+z)\notin x\m$.
We have  $u(\alpha u+z)=\alpha u^2+uz\notin x\m$ because $u^2\in
x\m$ and $uz\notin x\m$. We also have $ (\alpha u+z)^2-u(\alpha
u+z)=(\alpha^2-\alpha) u^2+z^2-uz\in x\m $.

Since $y^2=x(\alpha u+z)$, the change $z\leftarrow z+\alpha u$ gives
\begin{equation}
\label{y^2} y^2=xz\,.
\end{equation}

Recall from \eqref{ym} that $\rank(y)=2$ and $y\m=x\m$.

We  write $yz=\beta xu+\gamma xz$ with $\beta, \gamma\in \kk$. If $\beta=0$,
we see that $y\m=(xz)$ because  $yx=0=yu$ and $y^2=xz$. This   is a
contradiction, since $\rank(y)=2$.  We have thus $\beta\ne 0$.
Note that  $y(z-\gamma y)=\beta x u$ and
$x(z-\gamma y)=xz$. The  replacement $z\leftarrow z-\gamma y$ preserves
the hypotheses of Case 2.2.2 and \eqref{y^2}. In addition, we now  have
\begin{equation}
\label{yz} yz=\beta xu\quad\text{with $\beta\ne 0$}\,.
\end{equation}

Since $z^2-uz\in(xu,xz)$ we write $z^2-uz=\varepsilon xu+\lambda xz$
with $\varepsilon, \lambda\in\kk$. Set $z'=z+\varepsilon x$. We have
then  $z'^2=z^2$, $xz'=xz$, $yz'=yz=\beta xu$ and $uz'=uz+\varepsilon ux$.
In particular, we see that $\rank(z')=3$. Using $\ch \kk=2$, we have then
$z'^2+uz'+\lambda xz'=0$, hence $z'v=0$,  where $v=z'+u+\lambda x$.
We have $vx=xz+ux$, $vy=zy=\beta xu$ and $vu=uz'+u^2+\lambda
xu=uz+(\varepsilon+\lambda)ux+u^2\notin x\m$.  We conclude
$\rank(v)=3$. Condition \ref{conditions}(8) then holds with $x_1=z'$
and $x_2=v$.
\end{Case2.2.2}

\begin{Case2.2.3}  $z^2-uz\in x\m$,  $yu=0$ and   $y^2\in (xu)$ for some $y\in V\smallsetminus \kk x$.

Since $y^2\ne 0$, we may assume (after replacing $x$ with a scalar multiple of itself):
\begin{equation}
\label{y^2=xu}
y^2=xu\,.
\end{equation}
 Since $\rank(y)=2$ and
$yu=0=yx$, we have  $yz=\gamma xu+\delta xz$ for some $\gamma, \delta \in \kk$ with $\delta\ne 0$,
hence $y(z-\gamma y)=\delta xz=\delta x(z-\gamma y)$. With  $z\leftarrow
z-\gamma y$ we get
\begin{equation}
\label{yz=xz} yz=\delta xz\quad\text{with $\delta\ne 0$}\,.
\end{equation} This  replacement leaves intact the hypotheses of Case 2.2.3 and \eqref{y^2=xu}.

Furthermore, since $z^2-uz\in(xu,xz)$ we have $z^2-uz=\alpha
xu+\beta xz$ for some $\alpha,\beta\in \kk$, hence $(z-\alpha x)^2-u(z-\alpha x)=z^2-uz+\alpha
xu=\beta xz$.  With $z\leftarrow z-\alpha x$, we have
\begin{equation}
\label{z^2-uz} z^2-uz=\beta xz\,.
\end{equation}
This replacement leaves intact the hypotheses of Case 2.2.3, as well as
\eqref{y^2=xu} and \eqref{yz=xz}.

 Set  $w=z-u-(\beta+\delta)x$. We have then $(z+y)w=z^2-uz-\beta xz-\delta xz+yz=0$.  Furthermore, note that $x(z+y)=xz$, $y(z+y)=yz+y^2=\delta xz+xu$ and $z(z+y)=z^2+zy$. In particular, $z(z+y)\notin x\m$ because $zy\in x\m$ and $z^2\notin x\m$, hence $\rank(z+y)=3$.
We also have $xw=xz-xu$, $yw=zy=\delta xz$ and $uw=zu-u^2-(\beta+\delta)xu \notin
x\m$, hence $\rank(w)=3$. We conclude that \ref{conditions}(8)
holds, after a change of variables so that $x_1=w$ and $x_2=z+y$.
\end{Case2.2.3}

\begin{Case2.2.4}  $z^2-uz\in x\m$ and $y'u\ne 0$ for all $y'\in V\smallsetminus \kk x$.

Let $y\in V\smallsetminus \kk x$. If $yu=\lambda xu$ for some
$\lambda\in \kk$, then $(y+\lambda x)u=0$. Since $y+\lambda x\in
V\smallsetminus \kk x$, this contradicts the hypothesis of the case.
We may assume thus $yu\notin (xu)$. We write $yu=\gamma xu+ \varepsilon xz$ with $\gamma, \varepsilon
\in \kk$ and $\varepsilon \ne 0$. We may assume (after replacing $x$ with a scalar multiple of itself) $\varepsilon=1$, and
hence $(y-\gamma x)u=xz$. With $y\leftarrow y-\gamma x$ we see
that the hypotheses of Case 2.2.4 are preserved and in addition we
have
\begin{equation}
\label{yu=xz}yu=xz\,.
\end{equation}

Now we write $y^2=\alpha xu+\beta xz$ with $\alpha,\beta\in \kk$.
 Then  $y^2=\alpha xu+\beta yu=(\alpha x+\beta y)u$.

Set $y'=\alpha x+\beta y$. We have $y'y=\beta y^2=\beta y'u$. We have thus $y'(\beta u+y)=0$ and $y'\in V$. Assume $\beta \ne 0$, hence $y'\in
V\smallsetminus \kk x$.  With  $u\leftarrow u+\beta^{-1}y$, we see that $y'u=0$
and the hypotheses of Case 2.2.2 or Case 2.2.3 hold, with $y'$ playing the role of $y$.

We assume now $\beta= 0$, hence $y^2=\alpha xu$. Recall that $y^2\ne
0$ by the hypothesis of Case 2.2, hence $\alpha \ne 0$. We have thus
\begin{equation}
\label{y^2xu} y^2=\alpha xu \quad\text{with $\alpha\ne 0$}\,.
\end{equation}
Since $\m^2=(xz,xu, uz)$,  we see that $\rank(u)=3$, in view of
\eqref{yu=xz}. Recall $u^2\in x\m$, hence $u^2=\alpha'
xu+\beta' xz$ with $\alpha',\beta'\in \kk$. Set $u'=u-\alpha'
x-\beta' y$. We have then $uu'=0$.  Note that $u'x=xu$,
$u'y=uy-\beta' \alpha xu=xz-\beta'\alpha xu$ and $u'z=uz-\alpha' xz-\beta'yz$. Since $y\m=x\m$
and $uz\notin x\m$, we  see that $\rank( u')=3$.
If $\alpha'\ne 0$ or $\beta'\ne 0$, then condition
\ref{conditions}(8) is satisfied after a change of variables
with $x_1=u$ and $x_2=u-\alpha'x -\beta'y$. If $\alpha'=\beta'=0$,
then \ref{conditions}(1) is satisfied with $x_1=u=u'$.
\end{Case2.2.4}
\end{Case2.2}
\end{Case2}

We now proceed to prepare for the proof of Theorem
\ref{artinian-case}. We first establish some notation to be used in
the proof, in reference to the cases introduced in \ref{conditions}.
The  notation for Koszul complexes introduced in
\ref{notation-Koszul} will be used below.

\begin{bfchunk}{Notation for cases {\rm (2)}--{\rm(6)}}
\label{2-3} Assume that the following hypothesis holds:
\begin{equation}
\label{hyp} \m^2=x_1\m+x_2\a\quad\text{and}\quad x_1^2=0\,,x_2^2\in
x_1\m
\end{equation}
where $\a$ is homogeneous ideal such that $\a=\m$ or $\a\subseteq
(x_2, \dots, x_e)$. For each $j$ we write
\begin{equation}
\label{x_j^2} x_j^2=x_1a_j+x_2b_j
\end{equation}
with $a_j\in \sum_{i\ne 1}(x_i)$,  $b_j\in \sum_{i\ne 1}(x_i)\cap
\a$.

We denote by $A_j$, $B_j$  elements in the Koszul complex $\mathcal
K^1$ such that $\partial(A_j)=a_j$ and $\partial(B_j)=b_j$. Since
$x_1^2=0$, we make the convention that $a_1=b_1=0$ and $A_1=B_1=0$.
Also, since $x_2^2\in x_1\m$ we make the convention that $b_2=0$ and
$B_2=0$. For each $j$ we consider the following elements in $Q$:
$$
f_j=\wt x_j^2 -\wt x_1\wt a_j-\wt x_2\wt b_j\,,
$$
where $\wt a_j, \wt b_j\in \sum_{i\ne 1}(\wt x_i)$ are preimages in $Q$ of the elements $a_j$, $b_j$. The  elements $f_j$ give rise to cycles $x_jX_j-x_1A_j-x_2B_j$ in $\mathcal K_1$. In the proof of Theorem \ref{artinian-case} we will adjoin divided powers variables $\Lambda_j$ of degree $2$  such that
$$
\partial(\Lambda_j)=x_jX_j-x_1A_j-x_2B_j\,.
$$
Considering the reverse lexicographic  term order with $\wt x_1<\wt
x_2<...$ we see that $\init(f_j)=\wt x_j^2$ for all $j$, and hence
$f_i,f_j$ form a regular sequence for all $j\ne i$.

We need to set some special assumptions when case \ref{conditions}(6) holds. In this
case, $\a=\m$ and one also has $x_3\m\subseteq x_1\m$ and
$x_1x_3=0$. For this reason, we may assume $a_j,b_j\in
\sum_{i\notin\{1,3\}}(x_i)$, hence we may assume  $A_j, B_j\in
{\mathcal K}^{1,3}$ for all $j$, and also $b_3=0$, $B_3=0$.
\end{bfchunk}

\begin{bfchunk}{Notation for case {\rm (7).}}
\label{4-5} Assume that the following hypothesis holds:
\begin{equation}
\label{hyp2} \m^2=x_1\m+x_2\m+x_3\m\quad\text{and}\quad
x_1^2=0\,,x_2^2\in x_1\m\,,x_3^2\in (x_1, x_2)\m\,.
\end{equation}
For each $j$ we write
$$
x_j^2=x_1a_j+x_2b_j+x_3c_j
$$
with $a_j,b_j, c_j\in \sum_{i\ne 1}(x_i)$.  We denote by $A_j$,
$B_j$, $C_j$  elements in the Koszul complex $\mathcal K^1$ such
that $\partial(A_j)=a_j$ and so on. Since $x_1^2=0$, we make the
convention that $a_1=b_1=c_1=0$ and $A_1=B_1=C_1=0$. Also, since
$x_2^2\in x_1\m$ we make the convention that $b_2=c_2=0$ and
$B_2=C_2=0$ and, since $x_3^2\in (x_1, x_2)\m$ we make the
convention that $c_3=0$ and $C_3=0$. For each $j$ we consider the
following elements in $Q$:
$$
f_j=\wt x_j^2 -\wt x_1\wt a_j-\wt x_2\wt b_j-\wt x_3\wt c_j\,,
$$
where $\wt a_j, \wt b_j, \wt c_j\in \sum_{i\ne
1}(\wt x_i)$  are preimages in $Q$ of the elements $a_j, b_j, c_j$, and $f_1=\wt x_1^2$, $f_2=\wt x_2^2-\wt x_1\wt a_2$,
$f_3=\wt x_3^2-\wt x_1\wt a_3-\wt x_2\wt b_3$. As before, the
elements $f_j$ give rise to cycles $x_jX_j-x_1A_j-x_2B_j-x_3C_j$ in $\mathcal K_1$. In the proof of Theorem \ref{artinian-case}, we will adjoin variables $\Lambda_j$ such that
$$
\partial(\Lambda_j)=x_jX_j-x_1A_j-x_2B_j-x_3C_j\,.
$$
Considering the reverse  lexicographic  term order with $\wt x_1<\wt
x_2<\wt x_3<...$ we see that $\init(f_j)=\wt x_j^2$ for all $j$,
hence $f_i,f_j, f_k$ form a regular sequence for all distinct
integers $i, j,k$.
\end{bfchunk}

\noindent{\it Proof of Theorem {\rm\ref{artinian-case}}}. In view of
\ref{reduction-prelim}, we may assume that $R$ has no socle
elements in degree $1$. Using Proposition
\ref{codim1-2structure}, we see that it suffices to construct the
desired Golod homomorphism in each one of the cases (1)-(8) from
\ref{conditions}. When $R$ is a complete intersection, no further
work is needed. These cases will be recalled below, as we handle
them. We will also make the point that $R_3=0$ in each of the cases
(1)--(8). If $R$ is  a complete intersection of embedding
dimension $3$ (and hence codimension $3$, since $R$ is artinian),
then $R_3=1$ and $R_4=0$.  For the cases (1)--(8) we identify a
short Tate complex corresponding to a quadratic complete intersection $P$ of codimension at most $3$.  Corollary
\ref{tool} shows then that the induced
homomorphism $P\to R$ is Golod and $R$ is absolutely Koszul.

\begin{Case(1)} $\m^2=x_1\m$ and $x_1^2=0$.

We have $\m^3=x_1\m^2=x_1^2\m=0$. By Corollary \ref{apply-st}, the
complex $\mathcal D=\mathcal K\langle \Lambda\mid
\partial(\Lambda)=x_1X_1\rangle$ satisfies $\nu(\m\mathcal D)=0$. The
compex $\mathcal D$ is a short Tate complex corresponding to the quadratic complete intersection is $P=Q/(\wt x_1^2)$.
\end{Case(1)}

In what follows, we use the notation introduced in \ref{2-3} and \ref{4-5}.

\begin{Case(2)} $\m^2=x_1\m+x_2\m$, $x_1^2=0$, $x_2^2\in x_1\m$ and $x_1\m=(x_1x_j)$, where  $j=2$ if $x_1x_2\ne 0$ and $j=3$ if $x_1x_2=0$.

We use the notation in \ref{2-3} with $\a=\m$. Consider the
following DG $\Gamma$-algebras:
$$
\mathcal D=\mathcal K\langle \Lambda_n\mid n\in \{2,j\}\rangle\,,\quad
\mathcal E=\mathcal K\langle  \Lambda_j\rangle\quad\text{and}\quad\mathcal
A=\mathcal K^1\langle \Lambda_2 \rangle\,.
$$

We first show that $\nu(x_1\mathcal E)=0$ and $x_1\m^2=0$.  If
$x_1x_2\ne 0$, then $j=2$ and $\m(x_1)=x_2(x_1)$. We have
$x_2^2=x_1a_2$ and $x_1\in \ann(x_1)$. Recall that
$\partial(\Lambda_2)=x_2X_2-x_1A_2$. We apply Proposition \ref{1} with
$\b=(x_1)$, $x_t=x_s=x_2$ and $K=x_1A_2$ to conclude that
$\nu(x_1\mathcal E)=0$ and $x_1\m^2=0$. If $x_1x_2=0$, then $j=3$
and $\m(x_1)=x_3(x_1)$. We have then $x_3^2=x_1a_3+x_2b_3$ and $x_1,
x_2\in \ann(x_1)\m$. Recall that
$\partial(\Lambda_3)=x_3X_3-x_1A_3-x_2B_3$. We  apply Proposition \ref{1}
with $\b=(x_1)$, $x_s=x_t=x_3$  and $K=x_1A_3+x_2B_3$ to conclude
that $\nu(x_1\mathcal E)=0$ and $x_1\m^2=0$.

Let overbars denote the corresponding objects modulo
$x_1$ as in \ref{reduction-x}. In the ring $\ov R=R/(x_1)$ we have
$\ov \m^2=\ov x_2\ov\m$, $\ov x_2^2=0$ and $\ov{\mathcal
A}=\ov{\mathcal K^1}\langle \ov \Lambda_2\rangle$, with $\partial(\ov
\Lambda_2)=\ov x_2\ov X_2$.  Corollary \ref{apply-st} gives
$\nu(\ov\m\ov{\mathcal A})=0$. Since $\nu(x_1\mathcal E)=0$ and $\nu(\ov\m\ov{\mathcal A})=0$,
Proposition \ref{reduction-proposition} shows  $\nu(\m\mathcal
D)=0$. The ideal $J$ generated by $f_n$ with $n\in
\{2,j\}$ is a complete intersection ideal. Thus  $\mathcal D$ is a
short Tate complex  corresponding to the  complete intersection ring
 $P=Q/J$. Note that $P$ has codimension $1$ or $2$.

Finally, to show $\m^3=0$, we recall that $x_1\m^2=0$ from above and
$x_2^2\in x_1\m$. We have
$$
\m^3=x_1\m^2+x_2\m^2=x_2\m^2=x_2(x_1\m+x_2\m)\subseteq x_1\m^2=0\,.
$$
\end{Case(2)}

\begin{Case(3)} $\m^2=x_1\m+x_2(x_3)$, $x_1^2=0=x_1x_2$ and $x_2^2\in x_1\m$.

We write $x_3^2=x_1a_1+x_3b$ with $b=\beta x_2$, $\beta\in \kk$ and
$a_1\in \sum_{i\ne 1}(x_i)$.   Consider the following DG
$\Gamma$-algebras:
$$
\mathcal D=\mathcal K\langle \Lambda, \Upsilon \rangle\,,\quad \mathcal
E=R\langle X_1,\Lambda\rangle\quad\text{and}\quad \mathcal A=\mathcal
K^1\langle \Upsilon\rangle
$$
where $\Upsilon$ and $\Lambda$ are such that $\partial(\Upsilon)=x_3X_3-x_1A_1-bX_3$ and $\partial(\Lambda)=x_1X_1$.

We now show $\nu(\ov x_2\ov{\mathcal A})=0$, where overbars denote
the corresponding objects modulo $x_1$. In the ring $\ov R=R/(x_1)$
we have $\ov \m^2=\ov\m(\ov x_2)=\ov x_3(\ov x_2)$,  $\ov x_3^2=\ov
x_3\ov b$ and $\partial(\ov \Upsilon)=\ov x_3\ov X_3-\ov b\,\ov X_3$.  Note
that $\ov b\ov x_2\in (\ov x_2^2)=0$, hence $\ov b\in \ann(\ov
x_2)$.   We apply then Proposition \ref{1} with $\b=(\ov x_2)$,
$x_s=x_t=\ov x_3$ and $K=\ov b\,\ov X_3$ to get that the complex
$\overline{\mathcal A}=\overline {\mathcal K^1}\langle \overline
\Upsilon\rangle$ satisfies  $\nu(\overline x_2\overline{\mathcal A})=0$. We
also get $\ov x_2\,\ov\m^2=0$, hence $x_2\m^2\subseteq x_1\m^2$.

Note that $\mathcal D$ can be identified with the complex $\mathcal
E\otimes_R\mathcal A$. We have  $\m^2=x_1\m+(x_2)\m$ and
$x_1(x_2)=0$. Since $\nu(\ov x_2\ov{\mathcal A})=0$,  we apply
Proposition \ref{mix} with $x_s=x_t=x_1$ and $\a=(x_2)$ to conclude
$\nu(\m\mathcal D)=0$. Since $f_1=\wt x_1^2$ and $g=\wt x_3^2-\wt
x_1\wt a_1-\beta \wt x_2\wt x_3$ form a regular sequence, $\mathcal
D$ is a short Tate complex corresponding to the complete
intersection ring $P=Q/(f_1, g)$.

Finally, to show $\m^3=0$, we recall that $x_2\m^2\subseteq x_1\m^2$
from above and $x_1x_2=0$, hence
$$
\m^3=x_1\m^2+x_2\m^2\subseteq x_1\m^2=x_1^2\m+(x_1x_2x_3)=0\,.
$$
\end{Case(3)}

\begin{Case(4)} $\m^2=x_1\m+x_2(x_3)$, $x_1^2=0=x_1x_3$ and $x_2^2\in x_1\m$.

We use the notation in \ref{2-3} with $\a=(x_3)$. Consider the
following DG $\Gamma$-algebras:
$$
\mathcal D=\mathcal K\langle \Lambda_1, \Lambda_2\rangle\,,\quad \mathcal
E=R\langle X_1,\Lambda_1\rangle\quad\text{and}\quad \mathcal A=\mathcal
K^1\langle \Lambda_2\rangle\,.
$$

We first show $\nu(\ov x_3\ov{\mathcal A})=0$, where overbars denote
the corresponding objects modulo $x_1$. In the ring $\ov R$ we have
$\ov \m^2=\ov\m(\ov x_3)=\ov x_2(\ov x_3)$, $\ov x_2^2=0\in \ann(\ov
x_3)\ov\m$ and $\partial(\ov \Lambda_2)=\ov x_2\ov X_2$. We apply then
Proposition \ref{1} with $\b=(\ov x_3)$, $x_s=x_t=\ov x_2$ and $K=0$
to get that the complex $\overline{\mathcal A}=\overline {\mathcal
K^1}\langle \overline \Lambda_2\rangle$ sastisfies  $\nu(\overline
x_3\overline{\mathcal A})=0$. We also get $\ov x_3\,\ov\m^2=0$,
hence $x_3\m^2\subseteq x_1\m^2$.

Note that $\mathcal D$ can be identified with the complex $\mathcal
E\otimes_R\mathcal A$. We have  $\m^2=x_1\m+(x_3)\m$ and
$x_1(x_3)=0$. Since $\nu(\ov x_3\ov{\mathcal A})=0$,  we apply
Proposition \ref{mix} with $x_s=x_t=x_1$ and $\a=(x_3)$ to conclude
$\nu(\m\mathcal D)=0$. Since $f_1, f_2$ is a regular sequence,
$\mathcal D$ is a short Tate complex corresponding to the complete
intersection ring $P=Q/(f_1, f_2)$.

Finally, to show that $\m^3=0$ recall that $x_3\m^2\subseteq
x_1\m^2$ and $x_1x_3=0$, hence
$$
\m^3=(x_1\m^2+x_2x_3\m)\subseteq x_1\m^2=x_1^2\m+x_1x_2(x_3)=0\,.
$$
\end{Case(4)}

\begin{Case(5)} $\m^2=x_1\m+x_2(x_3)$, $x_1^2=0=x_1x_4$, $x_2^2\in x_1\m, x_2x_4\in x_1\m$ and $x_4^2-x_2x_3\in x_1\m$.

We use the notation in \ref{2-3} with $\a=(x_3)$. Consider the
following DG $\Gamma$-algebras:
$$
\mathcal D=\mathcal K\langle \Lambda_1, \Lambda_4\rangle\,,\quad \mathcal
E=R\langle X_1,\Lambda_1\rangle\quad\text{and}\quad \mathcal A=\mathcal
K^1\langle \Lambda_4\rangle\,.
$$

We first show $\nu(\ov x_4\ov{\mathcal A})=0$, where overbars denote
the corresponding objects modulo $x_1$. In the ring $\ov R$ we have
$\ov \m^2=\ov\m (\ov x_4)=\ov x_4(\ov x_4)$ and $\ov x_4^2=\ov
x_2\ov x_3$, with $\ov x_2\in \ann(\ov x_4)$, and $\partial(\Lambda_4)=\ov
x_4\ov X_4-\ov x_2\ov B_4$. We apply then Proposition \ref{1} with
$\b=(\ov x_4)$, $x_s=x_t=\overline x_4$ and $K=\ov x_2\ov B_4$ to get that the complex
$\overline{\mathcal A}=\overline {\mathcal K^1}\langle \overline
\Lambda_4\rangle$ sastisfies  $\nu(\overline x_4\overline{\mathcal A})=0$.
We also get $\ov x_4\,\ov\m^2=0$, hence $x_4\m^2\subseteq x_1\m^2$.

Note that $\mathcal D$ can be identified with the complex $\mathcal
E\otimes_R\mathcal A$. We have  $\m^2=x_1\m+x_2(x_3)=x_1\m+(x_4)\m$
and $x_1(x_4)=0$. Since $\nu(\ov x_4\ov{\mathcal A})=0$,  we  apply
Proposition \ref{mix} with $x_s=x_t=x_1$ and $\a=(x_4)$ to get
$\nu(\m\mathcal D)=0$. Since $f_1, f_4$ is a regular sequence,
$\mathcal D$ is a short Tate complex corresponding to the complete
intersection ring $P=Q/(f_1, f_4)$.

Finally, to show that $\m^3=0$ recall that $x_4\m^2\subseteq
x_1\m^2$ and $x_1x_4=0$.  We further have:
$$
\m^3=x_1\m^2+x_4\m^2\subseteq x_1\m^2=x_1^2\m+x_1(x_4^2)=0\,.
$$
\end{Case(5)}

\begin{Case(6)}  $\m^2=x_1\m+x_2\m$, $x_1^2=0=x_1x_3$, $x_2^2\in x_1\m$,  $x_3\m=(x_3^2)$,  $x_1\m=(x_1x_j, x_3^2)$ where $j=4$ if  $x_1x_2\in x_3\m$ and $j=2$ if $x_1x_2\notin x_3\m$.

We use the notation of \ref{2-3} with $\a=\m$. Consider the following DG $\Gamma$-algebras:
$$
\mathcal D=\mathcal K\langle \Lambda_n\mid n\in \{2,3,j\}\rangle\,,\quad
\mathcal E=\mathcal K\langle \Lambda_3\rangle\quad\text{and}\quad \mathcal
A=\mathcal K^3\langle \Lambda_n\mid  n\in\{2,j\}\rangle\,.
$$
Note that the special assumptions for case (6) recorded in \ref{2-3} give $\partial(\Lambda_3)=x_3X_3-x_1A_3$ and $\partial(\Lambda_n)\in
\mathcal K^3$ for any $n\ne 3$, hence $\mathcal A$ is well-defined.

We have $(x_3)\m=x_3(x_3)$ and  $x_3^2\in x_1\m$, $x_1\in
\ann(x_3)$. We apply then Proposition \ref{1} with $\b=(x_3)$, $x_s=x_t=x_3$  and
$K=x_1A_3$ to conclude that $\nu(x_3\mathcal E)=0$ and $x_3\m^2=0$.

Now consider the ring $R=R/(x_3)$ and let overbars denote the
corresponding classes modulo $(x_3)$. In this ring we have $\ov
\m^2=\ov x_1\ov\m+\ov x_2\ov\m$ and $\ov x_1\ov \m=(\ov x_1\ov
x_j)$. The  ring $\ov R$ falls into Case (2). Using the proof of Case
(2), we see that $\nu(\ov \m\ov{\mathcal A})=0$. The  ring $\ov R$
satisfies $\ov \m^3=0$, hence $\m^3\subseteq x_3\m^2$. Since $\nu(x_3\mathcal E)=0$ and $\nu(\ov \m \ov{\mathcal A})=0$, we
apply Proposition \ref{reduction-proposition} to conclude that
$\nu(\m\mathcal D)=0$. Since the elements  $f_n$ with $n\in
\{2,3,j\}$ form a regular sequence, the complex $\mathcal D$ is a
short Tate resolution corresponding to the ring $P=Q/J$ where $J$ is
the ideal generated by this sequence.

Finally, to show $\m^3=0$, recall that we know  $x_3\m^2=0$ and
$\m^3\subseteq x_3\m^2$.
\end{Case(6)}

\begin{Case(7)} $\m^2=x_1\m+x_2\m+x_3\m$,  $x_1^2=0$, $x_2^2\in x_1\m$, $x_3^2\in (x_1, x_2)\m$, $x_1\m=(x_1x_i)\ne 0$ and $(x_1,x_2)\m= (x_1x_i, x_2x_j)$ with $x_2x_j\notin (x_1x_i)$, for some $i\ne 1$, $j\notin\{1,2\}$.   We may further assume that $i=2$ if $x_1x_2\ne 0$ and $i=3$ if $x_1x_2=0$ and $x_1x_3\ne 0$.

Consider the following DG $\Gamma$-algebras:
$$
\mathcal D=\mathcal K\langle \Lambda_n\mid n\in \{2,i,j\}\rangle\,,\quad
\mathcal E=\mathcal K\langle \Lambda_i\rangle\quad\text{and}\quad \mathcal
A=\mathcal K^1\langle \Lambda_n\mid n \in \mathcal \{2,j\}\rangle\,.
$$

We first show that $\nu(\ov\m\ov{\mathcal A})=0$, where overbars
denote the corresponding objects modulo $x_1$. In the ring $\ov
R=R/(x_1)$ we have $\ov \m^2=\ov x_2\ov\m+\ov x_3\ov \m$, $\ov
x_2^2=0$, $\ov x_3^2=\ov x_2\ov b_3\in \ov x_2\ov\m$ and $\ov x_2\ov
\m=(\ov x_2\ov x_j)$. The  proof of Case (2) above then gives
$\nu(\ov\m\ov{\mathcal A})=0$. Also, we know $\ov \m^3=0$, hence
$\m^3\subseteq x_1\m^2$. Next we  show $\nu(x_1\mathcal E)=0$ and $x_1\m^2=0$. We identify
three cases that are treated separately.

If $x_1x_2\ne 0$, then $i=2$. Note that $\m(x_1)=x_2(x_1)$ and
$x_2^2=x_1a_2\in \ann(x_1)\m$. We then apply Proposition \ref{1}
with $\b=(x_1)$, $x_s=x_t=x_2$  and $K= x_1A_2$ to conclude that $\nu(x_1\mathcal
E)=0$ and $x_1\m^2=0$.

 If $x_1x_2= 0$ and $x_1x_3\ne 0$,  then $i=3$. Note that $\m(x_1)=x_3(x_1)$ and $x_3^2=x_1a_3+x_2b_3\in \ann(x_1)\m$. We then apply Proposition \ref{1} with $\b=(x_1)$, $x_s=x_t=x_3$ and $K=x_1A_3+x_2B_3$ to conclude that $\nu(x_1\mathcal E)=0$ and $x_1\m^2=0$.

Assume now $x_1x_2=x_1x_3=0$. Note that $\m(x_1)=x_i(x_1)$ and
$x_i^2=x_1a_i+x_2b_i+x_3c_i\in \ann(x_1)\m$. We then apply
Proposition \ref{1} with $\b=(x_1)$, $x_s=x_t=x_i$ and $K= x_1A_i+x_2B_i+x_3C_i$ to
conclude that $\nu(x_1\mathcal E)=0$ and $x_1\m^2=0$.

Since $\nu(x_1\mathcal E)=0$ and $\nu(\ov\m\ov{\mathcal A})=0$,
Proposition \ref{reduction-proposition} shows that $\nu(\m\mathcal
D)=0$. Note that the ideal $J$ generated by $f_n$ with $n\in
\{2,i,j\}$ is a complete intersection ideal. Thus  $\mathcal D$ is a
short Tate complex and the  corresponding complete intersection ring
is $P=Q/J$; it has codimension $1$ or $2$ or $3$.

Finally, note that $\m^3\subseteq x_1\m^2=0$.
\end{Case(7)}

\begin{Case(8)}$\m^2=x_1\m=x_2\m$ and $x_1x_2=0$;
In this case apply Corollary \ref{apply-st} to conclude that
$\nu(\m\mathcal D)=0$, where $\mathcal D=\mathcal K\langle V\mid
\partial(\Lambda)=x_2X_1\rangle$. The  complex $\mathcal D$ is a short
Tate comples corresponding to the ring $P=Q/(\wt x_1\wt x_2)$. Finally, one has $\m^3=x_1\m^2=x_1(x_2\m)=0$.
\end{Case(8)}

In view of Corollary \ref{tool}, the considerations above finish the
proof of Theorem \ref{artinian-case}. \qed

\section{Non-artinian quadratic $\kk$-algebras with $\dim_{\kk}R_2\le 2$}
\label{dim2}

In this section we prove:

\begin{theorem}
\label{dim2-thm} Let $R$ be a quadratic standard graded
$\kk$-algebra with $\dim_{\kk}R_2\le 2$. Assume that $\kk$ is algebraically closed.

 Then $R$ is Koszul and there  exists a quadratic complete intersection $\kk$-algebra
$P$ of codimension at most $1$ and a surjective Golod homomorphism
$P\to R$. In particular, $R$ is absolutely Koszul. Moreoever, $R$ is
Golod when $\dim_{\kk}R_2\le 1$.
\end{theorem}

We use the notation introduced in the previous sections. As in Section 3, we first establish structure through the
next proposition. When $\ch\kk\ne 2$, such structure is already
established in \cite{C1}. Our proof, while inspired by \cite{C1},
develops arguments that are independent of characteristic.

\begin{proposition}
\label{structure-nonA} Let $R$ be a non-artinian quadratic standard
graded $\kk$-algebra with no socle elements in degree $1$. Assume
$\kk$ is algebraically closed.

The  following statements then hold, up to a change of variables.
\begin{enumerate}[\quad\rm(1)]
\item If $\dim_{\kk}R_2=1$, then $R= \kk[\wt x_1]$.
\item If $\dim_{\kk}R_2=2$,  then one of the following holds:
\begin{enumerate}[\quad\rm(a)]
\item $\m^2=x_1\m+(x_2^2)$ and $x_1^2=0=x_1x_2$;
\item  $R$ is a  quadratic hypersurface of embedding dimension $2$;
\item $R= \kk[\wt x_1, \wt x_2, \wt x_3]/(\wt x_1^2,\wt x_1\wt x_2, \wt x_2^2-\wt x_1\wt x_3, \wt x_2\wt x_3)$.
\end{enumerate}
\end{enumerate}
\end{proposition}

\begin{bfchunk}{Notation and general considerations.}
\label{S} Assume $R$ is a  quadratic standard graded $\kk$-algebra
with no socle elements in degree $1$. Let $x\in R_1$ be  a non-zero
element with $x^2=0$.  We adopt the following notation:
$$
 W=\{r\in R_1\mid r\m\subseteq x\m\}\,,\quad S=R/(W)\,,\quad V=\ann(x)\cap R_1\,,\quad\text{and} \quad W'=V\cap W\,.
$$
Note that $\kk x\subseteq W$ and $W/\kk x$ is the degree 1 component
of $\socle(\ov R)$, where $\ov R=R/(x)$. The  ring $S$ is thus
isomorphic to $\ov R/(\socle(\ov R)_1)$. We record below some needed
properties.

\begin{subchunk}
\label{S1} $\m (W')=\m(W)=x\m$. This  can be easily seen from the
definitions, together with the assumption that $x^2=0$, which
guarantees $x\in W'$.
\end{subchunk}

\begin{subchunk}
\label{S2} $\dim_{\kk}W=e-\dim_{\kk}S_1$ and
$\dim_{\kk}V=e-\rank(x)$.
\end{subchunk}

\begin{subchunk}
\label{S3} $\dim_{\kk}S_i=\dim_{\kk}\ov R_i$ for all $i\ge 2$ and in
particular $\dim_{\kk}S_2=\dim_{\kk}R_2-\rank(x)$.  \end{subchunk}
\begin{subchunk}
\label{S4} $S$ is artinian if and only if $\ov R$ is artinian if and
only if $R$ is artinian; the last equivalence holds because $x^2=0$.
\end{subchunk}
\begin{subchunk}
\label{S5} $S$ has no socle elements in degree $1$. Indeed, let $\ov
y\in \socle(S)_1$, where $y\in R_1$ and overbars denote classes
modulo $(W)$. We have then inclusions $y\m\subseteq W\m \subseteq
x\m$, hence $y\in W$, and thus $\ov y=0$. \end{subchunk}
\begin{subchunk}
\label{S6} $\dim_{\kk}W'\ge e-\dim_{\kk}S_1-\rank(x)$ and equality
holds iff $R_1=W+V$; this can be explained using the equality
$\dim_{\kk}W'=\dim_{\kk}W+\dim_{\kk}V-\dim_{\kk}(W+V)$ and \ref{S2},
noting that $\dim_{\kk}(W+V)\le \dim_{\kk}R_1=e$. \end{subchunk}
\begin{subchunk}
\label{S7} Assume $\rank(x)=1$, hence $\dim _{\kk} V=e-1$ and so
$R/(V)$ is a quadratic algebra of embedding dimension $1$. We have
thus $R/(V)\cong \kk[t]/(t^2)$ or $R/(V)\cong \kk[t]$.
\begin{enumerate}[\quad(i)]
\item If $R/(V)\cong \kk[t]/(t^2)$ then $\m^2=(V)\m$. This  implies that $x\m^2=0$. We have $(V)\subseteq \ann(x)$, hence $(V)\m=\m^2=\ann(x)\m$. In particular, we see that the ideals $(V)$ and $\ann(x)$ coincide in all degrees $i$ with $i\ge 2$. Since they also coincide in degree $1$ , we have thus $(V)=\ann(x)$.
\item If  $R/(V)\cong \kk[t]$, then $(V)$ is a prime ideal.
\end{enumerate}
\end{subchunk}
\end{bfchunk}

\noindent{\it Proof of Proposition {\rm \ref{structure-nonA}}.} (1)
Assume $\dim_{\kk}R_2=1$. If $e>1$, then by \ref{conca-element},
there exists $x\in R_1$ non-zero such that $x^2=0$. Since $R$ has no
socle elements in degree $1$ we have $x\m\ne 0$ and hence
$x\m=\m^2$. It follows that $\m^3=0$, a contradiction. We must have
thus $e=1$, and hence $R$ is a polynomial ring in one variable.

(2)  Assume $\dim_{\kk}R_2=2$, hence $e\ge 2$.  If  $e=2$, then (b)
holds. Assume now $e>2$.  Then  by \ref{conca-element} there exists
$x\in R_1$ non-zero such that $x^2=0$. Since $R$ is not artinian we have
$x\m\neq \m^2$. Since $x\m\ne 0$, we see that $\rank(x)=1$. The  ring
$\ov R=R/(x)$  is not artinian by \ref{S4} and has $\dim_{\kk}(\ov
R_2)=1$. It follows by (1) that $\ov R$ is isomorphic to a trivial
fiber extension of a polynomial ring in one variable.

We adopt  the notation of \ref{S}. By \ref{S4}, the ring $S$ is not
artinian. Since $\rank(x)=1$, we have $\dim_{\kk}S_2=2-1=1$ by
\ref{S3}. Since $S$ has no socle elements in degree $1$ by \ref{S5},
we use (1) to see that $S$ is a polynomial ring in one variable,
hence $(W)$ is a prime ideal and  $\dim_{\kk}S_1=1$. We conclude
$\dim_{\kk} W=e-1$ by \ref{S2}. We have then $\dim_{\kk}W'\le e-1$,
since $W'\subseteq W$.  On the other hand, \ref{S6} gives  an
inequality $\dim_{\kk} W'\ge e-2$, with equality when $R_1=W+V$.

\begin{Case1} $\dim_{\kk}W'=e-2$.

As noted in \ref{S6}, it follows that $R_1=W+V$. We have $W=W'\oplus \kk
z$ for some $z\in R_1\smallsetminus V$ with $z\m\subseteq x\m$ and
$V=W'\oplus \kk y$ for some $y\in V\smallsetminus W$. We have then
$R_1=W'\oplus \kk y\oplus \kk z$ and $x\m=(xz)$, since $\rank(x)=1$
and $xz\ne 0$. Since $y\notin W$ and $(W)$ is prime, we see
$y^2\notin (W)\m=x \m$, hence  $\m^2=(xz,
y^2)=x\m+(y^2)$. Note that $x,y$ are linearly independent, hence we
can make a change of variables so that $x_1=x$ and $x_2=y$, and
thus (a) holds.
\end{Case1}

\begin{Case2} $\dim_{\kk}W'=e-1$.

Recall that $\dim_{\kk} W=e-1$ and note that $\dim_{\kk}V=e-1$ by
\ref{S2}. We have thus $W=W'=V$.  In particular, $(V)$ is a prime
ideal and $\m(V)=x\m$. Let $z\in R_1$ such that $R_1=V\oplus \kk z$.
 Then  $x\m=(xz)$ and $z^2\notin (xz)$; otherwise, we would get
$z^2\in (V)$, hence $z\in V$, a contradiction.
 Thus  $\m^2=(xz,z^2)$.

Set
$U=\ann(z)\cap R_1$. Note $U\subseteq V$ because $(V)$ is a prime ideal and $z\notin
(V)$.  We have $(xz)\subseteq (V)z\subseteq (V)\m=x\m=(xz)$, hence equalities hold throughout. The linear map $V\xrightarrow{\cdot z}\kk xz$ is a surjection, hence $V=U\oplus \kk x$. So $R_1=V\oplus \kk z=U\oplus \kk x\oplus \kk z$.

We have $UR_1\subseteq (V)\m=(xz)$.  In particular, $\dim_{\kk}U^2\le 1$. If $\dim_{\kk} U>1$,  there exists $u\in U$ non-zero with
$u^2=0$ by \ref{conca-element}. Since $u$ is not a socle element and
$UR_1\subseteq (xz)$, we have $u\m=(xz)$, hence $\m^2=u\m+(z^2)$.
Condition (a)  is  satisfied, after a change of variables with
$x_1=u$ and $x_2=z$.

Assume now  $\dim_{\kk}U=1$. Note that
$\dim_{\kk}U=e-\rank(z)=e-2$ by \ref{S2}, and hence $e=3$ in this
case.  Let $u\in U$ non-zero, so that we have $U=\kk u$. If $u^2=0$
proceed as above. If $u^2\ne 0$, note that $u^2\in (xz)$, hence we
may assume $u^2=xz$, after possibly  replacing $z$ with a scalar
multiple of itself. Since $\dim_{\kk}R_2=2$, the  relations $x^2=xu=uz=u^2-xz=0$ define $R$ and, since $x,u,z,$ are linearly independent,  (c) holds with  $x_1=x$,
$x_2=u$, $x_3=z$. \qed
\end{Case2}

\begin{remark} If condition  (a) of Proposition \ref{structure-nonA} holds, then the argument in the proof of Proposition \ref{nonA-use} shows that $\m^i$ is principal for all $i>2$, see \eqref{x2}. The  same conclusion  holds when (c) holds. Furthermore, if (b) holds, then $\m^i$ is $2$-generated for all $i\ge 2$.

 Let $R$ be a standard graded quadratic $\kk$-algebra with
$\dim_{\kk}R_2\le 2$. Set $$s=\dim_{\kk}\left(\socle(R)\cap
R_1\right)\,.$$
In view of the considerations above, and upon taking
into account the socle elements in degree $1$ as well,  Proposition
\ref{structure-nonA} gives  the following formula for the Hilbert
series of $R$:
\begin{equation}\label{Hilbert2}
\HH_R(T)=1+eT+(\dim_{\kk} R_2)T^2+(\dim_{\kk} R_3)T^3(1-T)^{-1}\,.
\end{equation}
In addition,  one has:
\begin{enumerate}
\item $\dim_{\kk}R_3=0$ when $R$ is artinian, in view of Theorem \ref{artinian-case};
\item $\dim_{\kk}R_3=1$ when $R$ is not artinian and $e-s>2$ or $\dim_{\kk}R_2=1$;
\item $\dim_{\kk} R_3=2$ if and only if $e-s=2$ and  $\dim_{\kk} R_2=2$.
\end{enumerate}
\end{remark}

\begin{proof}[Proof of Theorem {\rm\ref{dim2-thm}}]
As in the proof of Theorem \ref{artinian-case}, we may assume that
$R$ has no socle elements in degree $1$. We use now Proposition \ref{structure-nonA}. If
$\dim_{\kk}R_2=1$, then $R$ is a polynomial ring, and hence is
Golod. Assume now $\dim_{\kk}(R_2)=2$.

If (a) holds, then we  apply Proposition \ref{nonA-use} with
$x_t=x_1$ and $x_s=x_2$ and Corollary \ref{tool} to conclude that
the induced homomorphism $Q/(\wt x_1^2)\to R$ is Golod and $R$ is absolutely Koszul.

If (b) holds there is nothing to prove, since the ring is a
hypersurface.

If (c) holds, we use Corollary \ref{special-case} (noting that $R$ is both quadratic and non-artinian) and
Corollary \ref{tool} to conclude that the induced homomorphism
$Q/(\wt x_2\wt x_3)\to R$ is Golod and $R$ is Koszul.
\end{proof}

\section{Non-artinian quadratic $\kk$-algebras with $\dim_{\kk}R_2=3$}
\label{dim3-nonartinian}

The work in this section is done in the case when the
embedding dimension $e$ satisfies $e\ge 4$ and $R$ has no  socle
elements in degree $1$. The  case $e=3$ is handled in the next section, within the proof of the Main Theorem. Here we prove:

\begin{theorem}
\label{last} Let $(R,\m,\kk)$ be  a non-artinian quadratic standard graded $\kk$-algebra
such that  $\dim_{\kk}(R_2)=3$,  $\dim_{\kk}R_1\ge 4$ and $\socle(R)\subseteq\m^2$. Assume that $\kk$ is algebraically closed.

 Then $R$ is Koszul  and there exists a surjective Golod homomorphism $\varphi\colon
P\to R$ with $P$ a quadratic complete intersection of codimension
at most $2$. In particular, $R$ is absolutely Koszul.
\end{theorem}

\begin{remark}
In the proof of this theorem we depart from the strategy of the earlier sections, where  structural cases were established in the beginning. This is in part due to the fact that there are about 13 structural cases that need to be considered and we found it more efficient to construct the Golod homomorphism (by identifying a short Tate complex $\mathcal D$ with $\nu(\m\mathcal D)=0$) at the same time a particular structural case is identified. Our structural cases are comparable to the ones in \cite{Alessio}, but they do not use the hypothesis $\ch\kk\ne 2$.
\end{remark}

\noindent{\it Proof.}  We  identify the
complete intersection $P$ and the corresponding short Tate complex
such that $\nu(\m\mathcal D)=0$. The  conclusion follows then from
\ref{tool}.

Since $\dim_{\kk} R_1>\dim_{\kk} R_2$ we apply \ref{conca-element}
to find a  non-zero element $x\in R_1$ such that $x^2=0$. The  ring  $R/xR$ is  not artinian by \ref{S4}. If $\rank(x)=3$,
then $\m^2=x\m$ and this implies $R$ is artinian, a contradiction.
We consider next two main cases: $\rank(x)=2$ and $\rank(x)=1$. We use
the notation introduced in \ref{S} for the subspaces $V$, $W$, $W'$,
together with the properties recorded there.

\begin{Case1} $\rank(x)=2$.

We recall from \ref{S} that $S=R/(W)$ is not artinian, has no socle
elements in degree $1$ and $\dim_{\kk}(S_2)=\dim_{\kk}(R/xR)_2$,
hence $\dim_{\kk}(S_2)=1$.  By Proposition \ref{structure-nonA},
$R/(W)$ is isomorphic to a polynomial ring in one variable.  Then
$(W)$ is a prime ideal and  $\dim_{\kk}W=e-1$. We also have
$\dim_{\kk}V=e-2$ by \ref{S2}.

We further use \ref{S6} to see $e-3\leqslant \dim_{\kk} W'\leqslant
e-2$, with $\dim_{\kk} W'=e-3$ if and only if $R_1=W+V$.

If $\dim_{\kk} W'=e-3$, then $R_1=W+V$. We have  $V=W'\oplus \kk y$ with
$y\in V\smallsetminus W$, and $W=W'\oplus \kk z_1\oplus \kk z_2$
with $\m z_i\subseteq x\m$ for $i=1,2$. We have thus $R_1=W'\oplus
\kk y\oplus \kk z_1\oplus \kk z_2$ and $\m((z_1,z_2)+(W'))\subseteq
x\m$. We have thus $\m^2=x\m+ (y^2)$ and $xy=0$. We apply
Proposition \ref{nonA-use} with $x_t=x$ and $x_s=y$ to  conclude
that the short Tate complex $\mathcal D=\mathcal K\langle V\mid
\partial(\Lambda)=xX\rangle$ corresponding to $P=Q/(\wt x^2)$ satisfies
$\nu(\m\mathcal D)=0$.

If $\dim_{\kk} W'=e-2$, then we have $W'=V$ and $V\subseteq W$. In
particular, we have $V\m\subseteq x\m$.  Complete $x$ to a basis of
$V$, say $x, y_2,\ldots,y_{e-2}$. Let  $V'$ denote the subspace
generated by $y_2,\ldots,y_{e-2}$. We write  $R_1=W\oplus  \kk z$
and $W=V\oplus \kk y_1$ for some $z\in R_1\smallsetminus W$ and
$y_1\in W\smallsetminus V$. Since $x\m\subseteq(W)$ and $(W)$ is
prime we have $z^2\notin x\m$.  We have then
$\m^2=(xy_1,xz,z^2)$ and
\begin{align}
\label{y_im}
&y_i\m\subseteq (xy_1,xz)\quad\text{for all $i$}\,;\\
&xy_i=0\quad\text{for all $i>1$}\,.
\end{align}
We write $zy_i=\alpha_ixy_1+\beta_ixz$ with $\alpha_i,\beta_i\in
\kk$, and thus
$z(y_i-\beta_ix)=\alpha_ixy_1=\alpha_ix(y_1-\beta_1x)$. The
replacements $y_i\leftarrow y_i-\beta_ix$ for all $i$ do not
affect any of the assumptions on the elements $y_i$ that we made so
far, and, in addition,  yield
\begin{align}
&zy_i\in (xy_1)\quad\text{for all $i$}\,.
\end{align}
We also note that
\begin{equation}
\label{need-ann}
\ann(z)\cap \m^2\subseteq x\m\,.
\end{equation}
Indeed, let  $w\in \m^2\cap \ann(z)$. We write
$w=xa+\gamma z^2$ with $a\in\m$ and $\gamma\in R$.  We have then
$
0=wz=axz+\gamma z^3
$, so $\gamma z^3\in x\m\subseteq (W)$. Since  $(W)$ is a prime ideal not containing $z$, it follows $\gamma\in (W)$.  Then $\gamma z^2\in(W)\m=x\m$. Hence $w=xa+\gamma z^2\in x\m$, as desired.

\begin{Case1.1} $\rank(y_1)=2$.

Using the  relations above, we have  $y_1\m=(xy_1, xz)$. Since
$y_1z\in (xy_1)$ there exists $i$ such that $y_1y_i\notin\ (xy_1)$.

\begin{Case1.1.1} $y_1y_i\notin\ (xy_1)$ for some $i>1$.

The  hypothesis implies $x\m=(xy_1, y_1y_i)$. Set $\b=(x,y_i)$. We
have $x\b=(x^2, xy_i)=0$, hence $x\in \ann(\b)$,  and
$$\b\m=x\m=(xy_1,xz)=(xy_1, y_1y_i)=y_1\b$$
with $y_1^2\in x\m\subseteq \ann(\b)\m$. We write $y_1^2=xa$ with
$a\in \m$.  We apply Proposition \ref{useful} with $\a=\m$,  $x_s=z$, $x_t=y_1$ and $K=xA$
to conclude $\nu(\m\mathcal D)=0$, where $\mathcal
D=\mathcal K\langle \Lambda\rangle$ with $\partial(\Lambda)=y_1Y_1-xA$.
 The  complex $\mathcal D$ is a
short Tate complex and corresponds to the complete intersection
$Q/(\wt y_1^2-\wt x\wt a)$.
\end{Case1.1.1}

\begin{Case1.1.2} $y_1^2\notin (xy_1)$ and
$y_1y_i\in (xy_1)$ for all $i>1$.

We assume first that  $y_iy_j\notin(xy_1)$ for some $i>1$ and $j>1$.
 Then  $y_i(y_1-y_j)\notin(xy_1)$ and $\rank(y_1-y_j)=2$. The  replacement $y_1\leftarrow
y_1-y_j$ maintains the assumptions of Case 1.1, and we may thus
suppose that $y_1y_i\notin(xy_1)$, after this replacement.  This   is
Case 1.1.1.

We assume from now on that
\begin{equation}
\label{ij} y_iy_j\in(xy_1)\quad\text{for all}\quad i>1, j>1\,.
\end{equation}
 Since $zy_1\in (xy_1)$ and
$y_1y_i\in (xy_1)$ for all $i>1$, we write $zy_1=\alpha xy_1$ and
$y_1y_i=\beta_ixy_1$ with $\alpha, \beta_i\in \kk$. We have thus
$y_1(z-\alpha x)=0$ and $y_1(y_i-\beta_ix)=0$. The  replacements
$z\leftarrow z-\alpha x$ and $y_i\leftarrow y_i-\beta_ix$
for all $i>1$ preserve all  assumptions of Case 1.1.2 and
\eqref{ij}, and  yield
\begin{equation}
y_1z=0=y_1y_j\quad\text{for all} \quad j>1\,.
\end{equation}

Assume first that $y_2y_j=0$ for all $j>1$. We have then $y_2z\neq 0$, as
$\socle(R)\subseteq\m^2$. Since $\m^2=(xy_1, xz,z^2)$ and $y_2z\in (xy_1)$, it follows that $\m^2=z\m$.

Set $x_1=x$, $x_2=z$, $x_3=y_1$, $x_4=y_2, \dots $.
We claim that the hypotheses of Proposition \ref{newp} hold for $s=2$ and $t=3$ ( so $x_s=z$ and $x_t=y_1$). The hypotheses (1)-(2) have all been established above. It remains to check (3), namely $x_i\m\subseteq y_1\m$ for all $i\ne 2$. This is clear once we recall  $y_1\m=x\m=(xy_1, xz)$ and see that  $y_i\m\subseteq y_1\m$ by \eqref{y_im}. Proposition \ref{newp} applies thus, and yields $\nu(\m\mathcal D)=0$, where
$\mathcal D=\mathcal K\langle \Lambda \rangle$ with $\partial(\Lambda)=y_1Z$. The corresponding complete intersection is $P=Q/(\wt y_1\wt z)$.

Form now on, we may thus assume:
\begin{equation}
y_2y_j\neq 0\quad\text{for some $j>1$}\,.
\end{equation}

Since $y_1^2\notin (xy_1)$,  $y_1^2\in (xz, xy_1)$, we  replace $z$ with a scalar multiple
of itself and assume
\begin{equation}
\label{11} y_1^2-xz\in (xy_1)\,.
\end{equation}

Since $y_2y_j\in (xy_1)$, we replace $y_2$ with a scalar multiple
of itself and assume
\begin{equation}
y_2y_j=xy_1\,.
\end{equation}
 We write $y_j^2=ax$, with $a\in (y_1, z)$.
 Set $\ov R=R/(y_2)$ and let overbars denote the corresponding objects modulo $y_2$. We have:
$$ {\ov{x}}^2=\ov{x}\,\ov y_1=\ov y_1^2-\ov{x}\,\ov{z}=\ov y_1\,\ov{z}=\ov y_k\,\ov{\m}=0 \quad\text{for all $k>1$}\,.$$

Make a change a variable so that  $x_1=x$, $x_2= y_1$,  $x_3= z$, $x_i=y_{i-2}$ for all $i\ge 4$ and consider  the DG-algebras
$$
\mathcal D=\mathcal K\langle \Upsilon, \Lambda\rangle\,,\quad \mathcal E=\mathcal
K\langle \Upsilon\rangle\,,\quad\mathcal A=\mathcal K^{4}\langle \Lambda\rangle
$$
where $\partial(\Lambda)=y_1Z=x_2X_3$ and $\partial(\Upsilon)=y_jY_j-aX$. Since $\ov R$ is quadratic, Corollary \ref{special-case} yields
$\nu(\ov \m\ov{\mathcal A})=0$.

Since $\m(y_2)=(xy_1)=y_j(y_2)$ and $y_j^2\in x\m\subseteq
\ann(y_2)\m$ we apply Proposition \ref{1} with $\b=(y_2)$,
$x_s=x_t=y_j$  and $K=aX$ to get $\nu(y_2\mathcal E)=0$ and
$y_2\m^2=0$.  We apply next  Proposition \ref{reduction-proposition}
to obtain $\nu(\m \mathcal D)=0$. The  complex $\mathcal D$ is a
short Tate complex corresponding to the complete intersection ring
$P=Q/(\wt y_1\wt z, \wt y_j^2-\wt a\wt x)$.
\end{Case1.1.2}
\end{Case1.1}

\begin{Case1.2} $\rank(y_1)=1$.

In this case we have $y_1\m=(xy_1)$. After changes of variables of the type $y_i\leftarrow y_i-\alpha_ix$ for $i>0$ and $z\leftarrow z-\alpha x$, where $\alpha_i$, $\alpha\in \kk$, we
may assume, with no harm to previous assumptions, that
\begin{equation}
y_1y_i=y_1z=0\quad\text{for all $ i>1$.}
\end{equation}
If for some  $i>1$, $j>1$ we have $y_iy_j\notin(xy_1)$, then note
that $(y_1-y_j)y_i\notin (xy_1)=(x(y_1-y_j))$. This   is Case 1.1.1,
with $y_1-y_j$ playing the role of $y_1$.  Therefore we may further
assume
\begin{equation} y_iy_j\in(xy_1)\quad\text{for all $i>1$, $j>1$.}
\end{equation}
The  ring $R/(y_1)$   is isomorphic to a trivial fiber extension of
the hypersurface $\kk[\wt x, \wt z]/(\wt x^2)$,  and it is thus a
Golod Koszul ring by \ref{reduction-prelim}. Since
$(y_1)\m=x(y_1)$ and  $x^2=0\in \ann(y_1)\m$ we  apply Corollary
\ref{use-Golod} with $x_u=y_1$, $x_s=x_t=x$ and $K=0$ to get $\nu(\m
\mathcal{D})=0$, where $\mathcal D=\mathcal K\langle \Upsilon\mid
\partial(\Upsilon)=xX\rangle$. The  complex $\mathcal D$ is a short Tate
complex corresponding to the hypersurface $Q/(\wt x^2)$.
\end{Case1.2}
\end{Case1}

\begin{Case2} $\rank(x)=1$.

In this case we have  $\dim _{\kk} V=e-1$ and so  $R/(V)$ is a
quadratic algebra of embedding dimension $1$. We have thus
$R/(V)\cong \kk[t]/(t^2)$ or $R/(V)\cong \kk[t]$.

\begin{Case2.1}  $R/(V)\cong \kk[t]/(t^2)$.

As noted in \ref{S7}, we have  $(V)=\ann(x)$, $\m^2=\ann(x)\m$ and $x\m^2=0$.
Recall from \ref{S5} that $S=R/(W)$ has no socle elements in degree
$1$. By \ref{S3}, we have  $\dim_{\kk}S_2=2$ and
$\dim_{\kk}S_i=\dim_{\kk}(R/(x))_i$ for all $i\ge 2$ and thus
$\dim_{\kk}R_i=\dim_{\kk}S_i$ for all $i>2$.

We make a change of variables so that $x=x_1$. Since $x$ has rank
$1$, we write $x_1\m=(x_1v)$ for some $v\in \sum_{i\ne 1} \kk x_i$.

We consider now the structure of $S$, as given by Proposition
\ref{structure-nonA}.

\begin{Case2.1.1} $S$ satisfies condition (a) in  Proposition \ref{structure-nonA}.

Note that the ring $\ov R=R/(x_1)$ satisfies the same condition as
$S$, since $S=\ov R/((\socle \ov R)_1)$.

After a change of variables in which $x_1$ stays the same, we write
${\ov{\m}}^2=\ov x_2\ov{\m}+(\ov y^2)$, with  ${\ov x_2}^2=\ov
x_2\,\ov y=0$. Lifting back to $R$ we get
\begin{equation}\label{rank1}
\m^2=x_1\m+x_2\m+y^2R\qquad \text{with}\quad x_2^2\in x_1\m,\quad
x_2y\in x_1\m\,.
\end{equation}
Since $x_2^2\in x_1\m$ and $x_1\m=(x_1v)$, we write $x_2^2=\alpha x_1v$ with $\alpha \in \kk$.

If $x_1x_2\ne 0$, we can take $v=x_2$, hence $x_2^2=\alpha x_1x_2$.  Consider the DG-algebras
$$
\mathcal D=\mathcal E=\mathcal K\langle
\Lambda_2\rangle\quad\text{and}\quad\mathcal A=\mathcal K^1\langle
\Lambda_2\rangle
$$
where $\partial(\Lambda_2)=x_2X_2-\alpha x_1X_2$. Recall that
$(x_1)\m=(x_1x_2)=x_2(x_1)$. Since $x_2^2=\alpha x_1x_2$ and $x_1\in
\ann(x_1)$, we apply Proposition \ref{1} with $\b=(x_1)$,
$x_s=x_t=x_2$ and $K=\alpha x_1X_2$ to conclude that $x_1\m^2=0$ and
$\nu(x_1\mathcal E)=0$. In the ring $\ov R$ one has $\partial(\ov
\Lambda_2)=\ov x_2\ov X_2$. Proposition \ref{nonA-use} gives that
$\nu(\ov\m\ov{\mathcal A})=0$ and finally Proposition
\ref{reduction-proposition} shows that $\nu(\m\mathcal D)=0$. The
complex $\mathcal D$ is a short Tate complex corresponding to the
complete intersection $P=Q/(\wt x_2^2-\alpha \wt x_1 \wt x_2)$.

If $x_1x_2= 0$, then we can  make a change of variables such that
$v=x_e$ and $x_1, \dots, x_{e-1}$ is a basis of $V$. Since
$\m^2=(V)\m$, we write
\begin{equation}
\label{ee} x_e^2=\sum_{(i,j)\in \mathcal I}\alpha_{i,j}x_ix_j\quad
\text{with}\quad \mathcal I=\{(i,j)\mid 1\le i\le e-1, 1\le j\le
e\}\,.
\end{equation}
with $\alpha_{i,j}\in \kk$. Consider the DG-algebras
$$
\mathcal D=\mathcal K\langle \Upsilon, \Lambda\rangle\,,\quad\mathcal
E=\mathcal K\langle \Upsilon\rangle\quad\text{and}\quad\mathcal A=\mathcal
K^1\langle \Lambda\rangle
$$
where $\partial(\Upsilon)=x_eX_e-\sum_{(i,j)\in\mathcal
I}\alpha_{i,j}x_iX_j$ and $\partial(\Lambda)=x_2X_2-\alpha x_1X_e$. Recall $x_1\m=(x_1x_e)$ and note that
$x_i\in\ann(x_1)$ for all $i$ with $(i,j)\in \mathcal I$, hence
$x_e^2\in \ann(x_1)\m$. We  apply Proposition \ref{1} with $\b=(x_1)$, $x_s=x_t=x_e$ and
$K=\sum_{(i,j)\in\mathcal I}\alpha_{i,j}x_iX_j$ to conclude that
$\nu(x_1\mathcal E)=0$ and $x_1\m^2=0$. Consider the ring $\ov R=R/(x_1)$. Let overbars denote the
corresponding objects modulo $(x_1)$. The  ring $\ov R$ is not artinian and $\ov \m^2=\ov x_2\ov\m +(\ov y^2)$, with $\ov x_2^2=0$ and $\ov x_2\,\ov y=0$.  Proposition \ref{nonA-use}
applied to the ring $\ov R$ shows that $\nu(\ov\m\ov{\mathcal
A})=0$.  Then  Proposition \ref{reduction-proposition} gives
$\nu(\m\mathcal D)=0$. The  complex $\mathcal D$ is a short Tate
complex corresponding to the complete intersection $P=Q/(\wt
x_2^2-\alpha \wt x_1 \wt x_e, \wt x_e^2-\sum_{(i,j)\in\mathcal
I}\alpha_{i,j}\wt x_i\wt x_j)$.
\end{Case2.1.1}

\begin{Case2.1.2} $S$ satisfies condition (b) in Proposition \ref{structure-nonA}.

$S$ is thus a hypersurface.   Since $R/(x)$ is a trivial fiber
extension of $S$ and $S$ is a Golod Koszul ring, $R/(x)$ is
also Golod and Koszul by \ref{reduction-prelim}.  We make a change of variables such that
$x_e=v$ and $x_1=x,x_2 \dots, x_{e-1}$ is a basis of $V$. We write
$x_e^2$ as in \eqref{ee}. Recall that $x_1\m=(x_1x_e)$ and $x_e^2\in
\ann(x_1)\m$. We apply then Corollary \ref{use-Golod} with $x_u=x_1$
and $x_s=x_t=x_e$  to conclude that $\nu(\m\mathcal D)=0$, where $\mathcal D=\mathcal K\langle \Upsilon\rangle$, with $\partial(\Upsilon)=x_eX_e-\sum_{(i,j)\in\mathcal I}\alpha_{i,j}x_iX_j$. The  corresponding complete intersection is $P=Q/(\wt
x_e^2-\sum_{i,j}\alpha_{i,j}\wt x_i\wt x_j)$.
\end{Case2.1.2}

\begin{Case2.1.3} $S$ satisfies condition (c) in Proposition \ref{structure-nonA}.

Since $\ov R=R/(x)$ is  trivial fiber extension of $S$, we have $\ov
\m=( \ov y_1,\ov y_2,\ov y_3,\ldots,\ov y_{e-1})$, $\ov y_1^2=\ov
y_1\ov y_2=\ov y_2^2-\ov y_1\ov y_3= \ov y_2\ov y_3=0$ and $\ov
y_i\ov \m=0$ for all $i\geq 4$, where overbars denote the
corresponding objects modulo $(x)$.  Lifting back to $R$ we get
$\m=(x,y_1,\ldots,y_{e-1})$ and
$$
y_1^2,\, y_1 y_2,\,y_2^2- y_1 y_3,\, y_2 y_3\in x\m \quad \text{and}
\quad  y_i\m\subseteq x\m \quad \text{for all}\quad  i\geq 4.
$$
 We  have  $\m^2=(xv,y_2^2,y_3^2)=(xv,y_1y_3,y_3^2)$.

Assume first  $xy_2=0$. If $xy_i=0$  for all $i\neq 3$, then
$xy_3\neq0$ and a basis for $V$ consists of $x$ and all $y_i$ with
$i\ne 3$. It follows that $\m^2=(V)\m\subseteq (xy_3, y_1y_3)$. This
contradicts the hypothesis $\dim R_2=3$.  Therefore
 $x(y_2+y_i)=xy_i\neq 0$ for some $i\notin \{2,3\}$. If $i\ge 4$, then replace $y_2\leftarrow y_2+y_i$. If $i=1$, then replace $y_3\leftarrow y_3-y_2$ and $y_2\leftarrow y_2+y_1$. One can check that the hypotheses of Case 2.1.3 are retained in both cases. After the replacement(s), we have
\begin{equation}
xy_2\ne 0\,.
\end{equation}
 Since  $y_1y_2,y_2y_3\in x\m=(xy_2)$, we have $y_2(y_1-\beta x)=0$ and $y_2(y_3-\gamma x)=0$ for some $\beta,\gamma\in \kk$.
The  replacements $y_1\leftarrow y_1-\beta x$ and
$y_3\leftarrow y_3-\gamma x$ do not modify the assumptions
above, and we can thus make a change of variables as to assume
\begin{equation}
\label{and} y_1y_2=y_2y_3=0\,.
\end{equation}
We now argue below that we may further assume
\begin{equation}
\label{or} xy_3=0\quad\text{or}\quad  xy_1=0\,.
\end{equation}
Indeed, if $xy_3$ and $xy_1$ are non-zero, then $(xy_1)=(xy_3)$ and
hence $x(y_3-\alpha y_1)=0$ for some $\alpha \in \kk$. Note that we
also have $y_2(y_3- \alpha y_1)=0$. All assumptions above then hold
when replacing $y_3$ with $y_3-\alpha y_1$, and thus one may assume
$xy_3=0$.

We assume thus \eqref{or} holds. Also, we may take $x_1=x$ and $x_i=y_{i-1}$ for $i>1$. We write $y_2^2=y_1y_3+\varepsilon
xy_2$ with $\varepsilon\in\kk$. Consider the DG-algebras
$$
\mathcal D= \mathcal K\langle \Upsilon, \Lambda\rangle, \quad \mathcal E=
\mathcal K\langle \Upsilon\rangle\quad\text{and}\quad \mathcal A= \mathcal
K^1\langle \Lambda\rangle
$$
with  $\partial(\Upsilon)=y_2Y_2-\varepsilon xY_2-y_iY_j$ and
$\partial(\Lambda)=y_2Y_3$, where $(i,j)=(1,3)$ if $xy_1=0$ and
$(i,j)=(3,1)$ if $xy_3=0$. In both cases, we have $y_i\in \ann(x)$.

Using \eqref{or}, we see that   $y_2^2\in
\ann(x)\m$. We also have $\m(x)=y_2(x)$. Applying Proposition \ref{1}
with $\b=(x)$, $x_s=x_t=y_2$ and $K=\varepsilon xY_2+y_iY_j$ we
have then $\nu(x\mathcal E)=0$ and $x\m^2=0$. Let overbars denote
congruence classes modulo $x$. Since $\ov R$ is quadratic, Corollary \ref{special-case} gives
$\nu(\overline{\m}\overline{\mathcal A})=0$ and then it follows from
Proposition \ref{reduction-proposition} that $\nu(\m\mathcal{D})=0$.

The  complex $\mathcal D$ is a short Tate complex corresponding to
the complete intersection $P=k[\wt x,\wt y_1,\ldots,\wt
y_{e-1}]/(\wt y_2^2-\varepsilon\wt x\wt y_2-\wt y_1\wt y_3,\wt
y_2\wt y_3)$.
\end{Case2.1.3}
\end{Case2.1}

 \begin{Case2.2} $R/(V)\simeq \kk[t]$.

The  ideal  $(V)$ is  prime in this case. Let $z\in R_1$ such that
$R_1=V\oplus \kk z$. We have then  $x\m=(xz)$. Note that $z^i\not\in
(xz)$, since  otherwise we get $z\in (V)$ and this is a
contradiction.  Then   $\rank(z)>1$ holds.

 Set $U'=\ann(z)\cap R_1$.  Since $(V)$ is  prime  and $z\notin (V)$, we have thus
$U'\subseteq V$. Let $U$ denote a $\kk$-vector space such that
$U'\subseteq U\subseteq V$ and $V=U\oplus\kk x$.

\begin{Case2.2.1} $\rank(z)=2$.

 In this case we have $U=U'$, since both spaces have dimension $e-2$, hence $Uz=0=Ux$. We have $R_1=U\oplus \kk x\oplus \kk z$, hence $UR_1=U^2$ and $R_2=\kk xz+\kk z^2+U^2$. If $\dim_{\kk}U^2=3$, then $U^2=R_2$, hence $z^2\in U^2$ and in particular $z^2\in (V)$. This   is a contradiction, since $(V)$ is prime and $z\notin V$.
We have thus  $\dim_{\kk} U^2\leqslant 2$.

Assume there exists $u\in U$ non-zero such that $u^2=0$. Set
$x'=x+u$, and note that we have ${x'}^2=0$. Since $xu=uz=0$  and $\socle(R)\subseteq \m^2$, we have
$uv\neq 0$ for some $v\in U$. If $\rank(x')=2$, then the hypothesis
of Case 1 is verified, with $x'$ playing the role of $x$. Assume now
$\rank(x')=1$ and set
 $V'=\ann(x')\cap R_1$. If  $R/(V')$ is not artinian, then we have $R/(V')\cong \kk[t]$, which implies that $(V')$ is a
prime ideal. Since  $zv=0$ and $z\notin V'$  we obtain $v\in V'$.
 Then  $0=x'v=(x+u)v=uv$ which is a contradiction.  Therefore $R/(V')$
is artinian. This   is Case 2.1, with
 with $x'$ playing the role of $x$.

From now on we assume
\begin{equation}
\label{u^2ne0}
u^2\ne 0\quad\text{for all}\quad u\in U\quad\text{with}\quad u\ne
0\,.
\end{equation}

Note that if  $\dim_{\kk}U^2=1$, or if
 $e>4$, then
$\dim_{\kk}U^2<\dim_{\kk}U$.  In this case, there exists $u\in U$
non-zero such that $u^2=0$ by \ref{conca-element}, a contradiction.
We have thus
\begin{equation}
e=4\quad\text{and}\quad \dim_{\kk}U^2=2\,.
\end{equation}
Since $z^2\notin U^2$, we have $R_2=U^2\oplus \kk z^2$. Let $v, w$
be a basis of $U$.
Since $ v^2,w^2,vw$ are linearly dependent, $\kk$ is algebraically
closed, and $u^2\ne 0$ for all non-zero $u\in U$, we see that there
exist linearly independent elements $v'$, $w'$ in $U$ such that
$v'w'=0$. Renaming the  elements, we may assume thus
\begin{equation}
\label{or-char}
vw=0\,.
\end{equation}
In particular, it follows that $U^2=(v^2, w^2)$ and
\begin{equation}
\label{uvz}
\m^2=(v^2,w^2,z^2)=\m(v,w)+(z^2)\,.
\end{equation}
Since $z^2\notin (V)$, we also see that $xz\in (v^2,w^2)$. Without loss of generality, we may assume
\begin{equation}
\label{w^2}
w^2\in (xz, v^2)\,.
\end{equation}

Assume  first $w^2\in (xz)$. After replacing $x$ with a scalar multiple of itself, we may assume
$w^2=xz$. We make a change of variables so that $x_1=x, x_2=w, x_3=v$ and $x_4=z$. Since $(U)z=0$, $e=4$ and $\dim_{\kk}R_2=3$, we see that the ring $R$ is defined by the following relations
$$x_1^2=x_1x_2=x_1x_3=x_4x_2=x_4x_3=x_2x_3=x_2^2-x_1x_4=0\,.$$
We see that $x_2\m=(x_2^2)$ and $x_2^2=x_1x_4\in \ann(x_2)\m$. Furthermore, note that $R/(x_2)$ is
a trivial fiber extension of $\kk[\wt x_3, \wt x_4]/(\wt x_3\wt
x_4)$, which is a Golod Koszul ring, hence $R/(x_2)$ is a Golod Koszul ring by
\ref{reduction-prelim}.
We use Corollary \ref{use-Golod} with $u=s=t=2$ and $K=x_1X_4$ to conclude that $\nu(\m\mathcal D)=0$, where $\mathcal D=\mathcal K\langle \Lambda\rangle$ with $\partial(\Lambda)=x_2X_2-x_1X_4$. This   is a short Tate complex corresponding to
the hypersurface $Q/(\wt x_2^2-\wt x_1\wt x_4)$.

Finally, assume $w^2\notin (xz)$. In view of \eqref{w^2}, we can replace $v$ and $w$  by appropriate scalar multiples
of themselves as to assume $v^2+w^2=xz$. (Recall that $\kk$ is algebraically closed.)  Using \eqref{or-char}, note that $(v+w)^2=v^2+w^2=xz$. Set $u=v+w$. We have  $u(v,w)=(v^2,w^2)=\m(v,w)$. With $\b=(v,w)$, we have thus $u\b=\m\b$. Also, note that $u^2=xz\in \ann(\b)\m$ and $\ann(z)\cap
\m^2\subseteq (U)\m=\b\m$.  We also have $\m^2=\b\m+z(z)$ by \eqref{uvz}.  Note that $x, z, u$ are linearly independent, so we can make a change of variables so that $x_1=u$, $x_2=z$ and $x_3=x$.  Apply then Proposition \ref{useful} with $t=1$, $s=2$, $\a=(x_2)$ and $K=x_3X_2$ to see that $\nu(\m\mathcal D)=0$, where $\mathcal D=\mathcal K\langle \Upsilon\rangle$ with $\partial(\Upsilon)=x_1X_1-x_3X_2$. This   is a short Tate complex corresponding to the complete intersection $P=Q/(\wt x_1^2- \wt x_2\wt x_3)$.
\end{Case2.2.1}

\begin{Case2.2.2} $\rank(z)=3$.

In this case we have $\dim_{\kk} U=\dim_{\kk} U'+1$.  Let $u\in
U\smallsetminus U'$. We have then  $U=U'\oplus \kk u$ and $R_1=U'\oplus \kk x\oplus \kk u\oplus \kk z$.

Assume there exists $v\in U'$ non-zero such that $v^2=0$. Since Case 1 is already settled, we may assume $\rank(v)=1$. Set
$V'=\ann(v)\cap R_1$. If $R/(V')$ is artinian we are in Case 2.1,
with $v$ playing the role of $x$. Otherwise, since $vz=0=xv$ and
$v\m\ne 0$, there exists $z'\in U$ such that $vz'\ne 0$. We have
then $R_1=V'\oplus \kk z'$, hence $\rank(z')>1$ as argued in the
beginning of Case 2.2, with $z'$ playing the role of $z$. On the
other hand, we have $z'\m\subseteq  (V)$ hence $\rank(z')\le 2$
because $z^2\notin (V)$. We have thus $\rank(z')=2$. We are then in
Case 2.2.1, with $v$ playing the role of $x$ and $z'$ playing the
role of $z$.

We may assume thus from now on
\begin{equation}
\label{v^2} v^2\ne 0\quad\text{for all}\quad v\in U'
\quad\text{with}\quad v\ne 0\,.
\end{equation}

Assume now there  exists $u'\in U'$ such that $\rank(u')=2$.  We have inclusions
$$u'R_1\subseteq U'R_1\subseteq UR_1\subseteq VR_1\,.$$ Since $z^2\notin (V)$ and  $\rank(u')=2$, we see that these inclusions must be equalities and in particular $u'\m=(V)\m$ and  $\m^2=u'\m+(z^2)$.

We show now that  $\ann(z)\cap \m^2=u'\m$. Indeed, let $w=u'v+\alpha z^2\in \m^2$ with $v\in \m$ and $\alpha \in R$. If $wz=0$, then $\alpha z^3=0$. Since $(V)$ is a prime ideal not containing $z$, it follows that $\alpha \in (V)$, hence $\alpha z^2\in (V)\m=u'\m$. We conclude $w\in u'\m$.

Since $\rank(z)=3$, we have $\m^2=z\m$.  We see that the hypotheses of Proposition \ref{newp} are satisfied with
$x_s=z$ and $x_t=u'$. By Proposition \ref{newp}, we have $\nu(\m\mathcal D)=0$, where $\mathcal D=\mathcal K\langle \Lambda\rangle$ with $\partial(\Lambda)=x_tX_s$. This is a short Tate complex corresponding to the complete intersection $Q/(\wt u'\wt z)$.

Assume from now on
\begin{equation}
\rank(u')=1\quad\text{for all $u'\in U'$ with $u'\ne 0$.}
\end{equation}

Assume first  that $e>4$, hence $\dim_{\kk} U'\geq 2$. Let $u_1,
u_2$ be linearly independent in $U'$. If $u_1u_2\ne 0$, then we have
$u_1\m=(u_1u_2)=u_2\m$. If $u_1u_2=0$, then we have $u_1(u_1+u_2)\ne
0$ and $u_2(u_1+u_2)\ne 0$ and then we have
$u_1\m=(u_1+u_2)\m=u_2\m$. We conclude  that $u_1\m=u_2\m$ for all
$u_1, u_2\in U'$ non-zero, and hence $\dim_{\kk}(U'^2)=1$. By
\ref{conca-element}, there exists $u\in U'$ such that $u^2=0$ and
$u\ne 0$, a contradiction with \eqref{v^2}.

Assume now $e=4$. Let $v\in U'$ with $v\ne 0$, so that  $U'=\kk v$ and $U=\kk
u\oplus  \kk v$. We have thus $\m=(x,u,v,z)$.

If $uv\ne 0$, then \eqref{v^2} implies $(uv)=(v^2)$ since
$\rank(v)=1$. We have thus $(u-\alpha v)v=0$ for some $\alpha \in
\kk$. After replacing $u$ with $u-\alpha v$,  we may assume
\begin{equation}
\label{uv=0}uv=0\,.
\end{equation}
Note that since $\rank(z)=3$ and $zv=0$  we have $\m^2=(xz,uz,z^2)$. Since
$z^2\notin V$, we must have $u^2, v^2\in(xz,uz)$.

If $u^2\notin (xz)$, then we  write $u^2=\alpha xz+\beta uz$ for
some $\alpha,\beta \in \kk$ with $\beta\neq 0$. Set $z'=u-\beta z$.
We have then $uz'\in (xz)$, $xz'=-\beta xz$, $zz'=uz-\beta z^2$ and
$vz'=0$, hence $\rank(z')=2$. This   is Case 2.2.1, with $z' $ playing
the role of $z$.

We assume from now  on $ u^2\in (xz)$. We review now the relevant
assumptions so far:
\begin{equation}
\label{relevant} \m^2=(xz, uz, z^2), \,\,\, u^2\in(xz), \,\,\, 0\ne
v^2\in (xz,uz), \,\,\, x^2=zv=xv=xu=uv=0\,.
\end{equation}
Write
\begin{equation}
\label{ab}
v^2=\alpha xz+\beta  uz\quad\text{with}\quad \alpha, \beta\in \kk\,.
\end{equation}

We first assume $\beta=0$, so $v^2=\alpha xz$. Since $v^2\ne 0$, we have $\alpha\ne 0$, and we conclude $xz\in (v^2)$, and hence $(xz)=(v^2)$ and $u^2\in (v^2)$.  Since $\rank(v)=1$, we have $v\m=(v^2)=(xz)$.
We have then $R/(v)\cong k[\tilde x,\tilde u,\tilde z]/(\tilde
 x^2,\tilde u^2, \tilde x\tilde u,\tilde x\tilde z)$. Note that the image of $x$ is a socle element in $R/(v)$ and $R/(v,x)\cong  k[\tilde u,\tilde z]/(\tilde u^2)$. Since $R/(v,x)$ is a hypersurface, it is Golod, and hence $R/(v)$ is also Golod by \ref{reduction-prelim}.
  
 Since  $vz=0$, note that  $v^2\in \ann(v)\m$. Set $x_1=x,x_2=v,x_3=z,x_4=u$. The hypothesis of Corollary \ref{use-Golod}
 are satisfied with $u=s=t=2$ and so $\nu(\m \mathcal D)=0$, where $\mathcal D=\mathcal K\langle \Upsilon\mid \partial(\Upsilon)=x_2X_2-\alpha x_3X_1\rangle$.
  This   is a short Tate complex corresponding to the complete
  intersection $P=Q/(\tilde x_2^2-\alpha\tilde x_1\tilde x_3)$.

  Assume now $\beta\neq 0$ in \eqref{ab}. The replacement $
  u\leftarrow \alpha x+\beta u$  does not modify the assumptions \eqref{relevant}. Hence  we can make a change of variables as to assume that
  $v^2=uz$. We can furthermore  assume that $u^2=xz$ by replacing
$x$ with a scalar multiple of itself. With  $x_1=x$, $x_2=v$,
$x_3=z$, $x_4=u$, we see that the following relations hold:
\begin{equation}
x_1^2=x_1x_4=x_1x_2=x_2x_4=x_2x_3=x_2^2-x_3x_4=x_4^2-x_1x_3=0\,.
\end{equation}
 We apply Corollary \ref{last-one} (noting that $R$ is both non-artinian and quadratic) to see $\nu(\m\mathcal D)=0$, where $\mathcal D=\mathcal K\langle \Upsilon\rangle$ with $\partial(\Upsilon)=x_2X_2-x_4X_3$.
 This   is a short Tate complex  corresponding to the complete intersection $P=Q/(\tilde x_2^2-\tilde x_3\tilde x_4)$.
\end{Case2.2.2}
\end{Case2.2}
\end{Case2}
The proof in all cases is now complete, in view of  \ref{tool}. \qed

\section{Non-Koszul quadratic $\kk$-algebras with $\dim_{\kk}R_2\le 3$}
\label{exceptional-sec}

Let $\kk$ be an algebraically closed field. A complete classification, up to isomorphism, of non-Koszul quadratic $\kk$-algebras  with $\dim_{\kk}R_2\le 3$ is provided by D'Al\`i \cite[Theorem 3.1]{Alessio} in the case when $\ch\kk\ne 2$ and is recalled  in Remark \ref{Alessio-exceptional} below. D'Al\`i  also points out in \cite[Remark 3.2(b)]{Alessio} that the classification needs to be enlarged when $\ch\kk=2$. In this section we establish structure when $\ch \kk=2$ as well, but we stop short of providing a classification up to isomorphism. We provide information on the Hilbert series of non-Koszul quadratic $\kk$-algebras  with $\dim_{\kk}R_2\le 3$. This information is then used at the end of the section to finalize the proof of the Main Theorem from the introduction.

\begin{remark}
\label{Alessio-exceptional}
Assume $R$ is a standard graded quadratic  non-Koszul $\kk$-algebra, with $\kk$ an  algebraically closed field,  $\dim_{\kk}R_2\le 3$ and $\ch\kk\ne 2$. As shown by D'Al\`i \cite[Theorem 3.1]{Alessio}, $R$ is isomomorphic, up to trivial fiber extension,  with one of the following non-Koszul quadratic algebras:
\begin{enumerate}[\quad\rm(i)]
\item $\kk[x,y,z]/(y^2+xy,xy+z^2,xz)$;
\item $\kk[x,y,z]/(y^2,xy+z^2,xz)$;
\item $\kk[x,y,z]/(y^2,xy+yz+z^2,xz)$.
\end{enumerate}
Furthermore, if the characteristic is not $3$, then the algebras (ii) and (iii) are isomorphic.

If $R$ is any of the rings in (i)-(iii) above, one can check that $\m^i=(x^i)$ for all $i\ge 3$, hence $\dim_{\kk}R_i=1$ for all such $i$.  The result below shows that the same conclusion on the Hilbert series of $R$ holds when  $\ch\kk=2$. In its proof we will use repeatedly basic results that show an algebra is Koszul, as recalled in \ref{Grobner}.
\end{remark}

\begin{theorem}
\label{exceptional}
Let $\kk$ be an algebraically closed  field and let $(R,\m,\kk)$ be a standard graded quadratic $\kk$-algebra with  $\dim_{\kk}R_2\le 3$  and $\socle(R)\subseteq \m^2$.

If $R$ is not Koszul, then $\dim_{\kk}R_1=3=\dim_{\kk}R_2$ and $\dim_{\kk}R_i=1$ for all $i\ge 3$.
\end{theorem}

\noindent{\it Proof.}
The statement of the theorem remains invariant under faithfully flat extensions, so we may assume $\kk$ is algebraically closed. If $R$ is not Koszul, then Theorems \ref{artinian-case}, \ref{dim2-thm} and \ref{last} show that we must have $\dim_{\kk}R_2=3=\dim_{\kk}R_1$ and $R$ is not artinian. We write $R=\kk[\wt x,\wt y,\wt z]/I$, where $I$ is a homogeneous ideal minimally generated by $3$ quadrics. The elements $x,y,z$ denote the corresponding images of the variables in $R$.

 In view of Remark \ref{Alessio-exceptional}, we only need to consider the case $\ch\kk=2$. Although our work here could be adapted to recover the case $\ch \kk\ne 2$, we will assume from now on $\ch\kk=2$,  in order to simplify computations.  If $R$ is not Koszul, we will show that $R$ satisfies  $\dim_{\kk}R_i=1$ for all $i\ge 3$.

First, assume there is a linear form $\ell$ such that $\ell^2=0$. We may assume $\ell=x$. If $\rank(x)=3$, then $R$ is artinian. This is a contradiction.
Hence we have  $\rank(x)\leq 2$.

 Assume $\rank(x)=1$. We use the notation introduced in \ref{S}, namely
$$V=\ann(x)\cap R_1\quad\text{and}\quad W=\{r\in R_1\mid r\m\subseteq x\m\}\,.$$ We proceed as in Case 2  in the proof of Theorem \ref{last}. Assume first that $R/(V)\cong \kk[t]/(t^2)$. In this case, we consider the ring $S=R/(W)$ and note that $\dim_{\kk}S_2=2$. Since $e=3$ and $x\in W$, we must have $\dim_{\kk}S_1=2$ as well, hence $S$ is a hypersurface.  We see that the  proof of Case 2.1.2 in Theorem \ref{last} carries through. We may assume thus  $R/(V)\cong \kk[t]$ and $(V)$ is a prime ideal. After a change of variables, we may also assume $V=\kk x\oplus \kk  y$. We have thus $x^2=xy=0$ and $xz\ne 0$.

Since $(V)$ is prime and $z\not\in V$, we see that  $z^2\notin (x,y)\m$. If  $yz\in (xz)$, then, after a change of variables, the ring $R$ is defined by the relations $x^2=xy=yz=0$ and is Koszul. If $yz\notin (xz)$, then $yz\notin (z^2, xz)$, hence $\m^2=(z^2, xz, yz)$. We write
$$y^2=\alpha xz+\beta yz\quad\text{for some}\quad  \alpha, \beta\in \kk\,.$$
 If $\alpha=\beta=0$, then the ring $R$ is defined by $x^2=xy=y^2=0$ and is Koszul.  If $\beta=0$ and $\alpha\ne 0$, we can make a change of variables as to assume that the ring is defined by the relations $x^2=xy=y^2+xz=0$. In this case, it can also  be seen that $R$ is Koszul; for example, one can check that $\wt x^2, \wt x\wt y, \wt y^2+\wt x\wt z$ is a Gr\"obner basis of $I$ with respect to the lexicographic order with $\wt y>\wt x>\wt z$. If $\beta\ne 0$ we have $(y+\alpha\beta^{-1} x)(\beta z+y+ x)=0$, and we can make a change of variables as to assume that the ring is defined by the relations $x^2=xy=yz=0$, and is thus Koszul.

Assume now $\rank(x)=2$,
hence $\dim_{\kk} V=1$ and thus $V=\kk x$.  Set
\begin{equation*}
\label{H}H=\{u\in R_1\mid \rank(u)<3\}\,.
\end{equation*}
The proof of  \cite[Lemma 2.8(1)]{CRV}  shows that $H$ is a variety in the projective space $\mathbf{P}(R_1)$ and $\dim H\geq 1$. In particular it follows that there exists $\ell\in R_1$ with $\rank(\ell)<3$ and $\ell\notin \kk x$. We may assume $\rank(y)<3$ and thus $\dim_{\kk} \ann(y)\cap R_1\ge 1$. Let $u\in R_1$ such that $yu=0$. Since $y\notin \kk x=V$, we have $yx\ne 0$, and hence $u\notin \kk x$. We may assume $u=y$ or $u=z$.

Assume $u=y$, hence $y^2=0$. If $z^2\in (x,y)\m$, it follows that $z^4=0$. Recalling that $x^2=0$, we see that $R$ is artinian, a contradiction. We have thus $z^2\notin (x,y)\m$. Since $\rank(x)=2$ and $x^2=0$, we see that $x\m=\kk xy\oplus \kk xz$. We have thus $\m^2=(z^2, xy, xz)$. Since $z^2\notin (x,y)\m$, we have $yz\in (xy, xz)$. Let $\alpha, \beta\in \kk$ such that
$
yz=\alpha xy+\beta xz
$.
After a change of variables, we may assume the ring $R$ is defined by the relations $x^2=y^2=yz=0$, hence is Koszul. (The assumption $\ch\kk=2$ is used here.)

Assume now $u=z$, hence $yz=0$. As in the beginning of Case 1 of the proof of Theorem \ref{last}, we see that $\dim_{\kk}W=2$.
Let $W=\kk x\oplus \kk w$, where $w=\alpha y+\beta z$, with $\alpha, \beta\in \kk$.
Then $w(y,z)\subseteq x\m=(xz,xy)$.
 By symmetry we can assume that $\alpha \neq 0$ and since $yz=0$ we have then  $y^2\in (xz,xy)$. We write $y^2=\beta' xz+\gamma xy$, with $\beta', \gamma\in \kk$. The ring $R$ is then defined by the relations $x^2=yz=y^2+\beta' xz+\gamma xy=0$. Clearly, this ring is Koszul when $\beta'=\gamma=0$. Otherwise, $R$ is isomorphic to either (ii) or (iii) from Remark \ref{Alessio-exceptional}.

Now suppose that there is no null-square non-zero linear form in $R$.  Using  \cite[Lemma 2.8(1)]{CRV}, we see  that there exists  a non-zero linear form $\ell$ with $\rank(\ell)<3$. We may assume $\ell=x$, and then, since $x^2\ne 0$,  we may further assume $xy=0$.

Assume $\rank(x)=1$. In this case we may also assume $xz=0$. If $y^2\in (yz)$ we can make a change of variables to assume $R$ is defined by $xy=xz=yz=0$, and hence is Koszul. Assume now that $y^2, yz$ are linearly independent. If $z^2\in (y^2, yz)$, then we may asssume $z^2=y^2+\alpha yz$ or $z^2=\alpha yz$ with $\alpha\in \kk$. If $z^2=\alpha yz$,
we have $z(z-\alpha y)=0$ and the replacement $y\leftarrow z-\alpha y$ further yields that $R$ is defined by the relations $xz=xy=yz=0$ and $R$ is Koszul. Assume now $z^2=y^2+\alpha yz$. If $\alpha=0$, we get $(z-y)^2=0$, a contradiction. We have thus $\alpha\ne 0$. Since $\kk$ is algebraically closed, we see that there exists $a\in \kk$ such that $(z+ay)^2=\alpha y(z+ay)$. (We need $a^2+\alpha a=1$.) With this value of $a$, the  replacement $z\leftarrow z+ay$ yields $z^2-\alpha yz=0$, and this case has been treated already.  We may assume thus $z^2\notin (y^2, yz)$ and thus $\m^2=(z^2, yz, y^2)$, and in particular  $x^2\in (y,z)\m\subseteq \ann(x)\m$. Note that the ring $R/(x)$ is a hypersurface, hence is Golod.  Apply Corollary \ref{use-Golod} with $x_s=x_t=x_u=x$ and Corollary \ref{tool} to see that $R$ is Koszul.

Proceed similarly when $\rank(y)=1$.

Assume now that $\rank(x)=2=\rank(y)$. We have thus $xy=0$ and $xz\ne 0$, $yz\ne 0$.
Set $\ov R=R/yR$. We  have then
$\dim \ov R_1=2$ and $\dim \ov  R_2=1$. Applying \ref{conca-element}, there exists
$u\in R_1\setminus \kk y$ such that $u^2\in yR_1$. We
consider two cases.

\begin{Case1} $x^2\in yR_1$.

We write $x^2=\alpha y^2+\beta yz$ with $\alpha, \beta\in \kk$. If $\beta =0$ then we see that $(x+\alpha^{1/2}y)^2=0$ which is a
contradiction. Hence $\beta\neq 0$. After the replacement
$z\leftarrow \alpha y+\beta z$, we get $x^2=yz$. In particular, it follows that $xz$, $yz$ are linearly independent, since $\rank(x)=2$ and $xy=0$.

If $y^2\in (yz, xz)$, then $y^2=\alpha yz+\beta xz$ for some $\alpha, \beta\in \kk$. Note that $\beta\ne 0$, since $\rank(y)=2$ implies that $y^2, yz$ are linearly independent. The ring $R$ is thus defined by the relations
\begin{equation}
\label{NK1}
xy=x^2-yz=y^2-\alpha yz-\beta xz=0\quad \text{with}\quad \beta \ne 0\,.
\end{equation}
We see that $\m^2=(yz,xz,z^2)$ and $x^2z=xz^2=y^2z=yz^2=0$. Since $\m^2=z\m$, we get $\m^3=z\m^2=z^2\m=(z^3)$. Thus for all $i\ge 4$, $\m^i=z\m^{i-1}=(z^i)$  by induction, and hence $\dim_{\kk}R_i=1$ for all $i\ge 3$.

If $y^2\notin (yz, xz)$, then $\m^2=(xz,yz, y^2)$.  We have then  $z^2=\alpha x^2+\beta y^2+\gamma xz$ for $\alpha, \beta, \gamma\in \kk$. If $\beta\ne 0$, then we see that $R$ is artinian, a contradiction.  We have then $\beta=0$, so $z^2=\alpha x^2+\gamma xz$. Since $\kk$ is algebraically closed, we see that there exists $a\in \kk$ such that $(z+ax)^2=\gamma x(z+ax)$. (We need $a^2+\gamma a=\alpha$.) With this choice of $a$, the replacement $z\leftarrow z+ax$  gives $xy=x^2-yz=z^2-\gamma xz=0$. Since $R$ has no null-square linear forms, we see $\gamma\ne 0$. By further replacing $x\leftarrow \gamma x$ and $y\leftarrow \gamma^2y$, we may assume $R$ is defined by
\begin{equation}
\label{NK2}
xy=x^2-yz=z^2- xz=0\,.
\end{equation}
We can check that $y^2z=yz^2=x^2z=xz^2=0$. Since $\m^2=(y^2, xz,yz)$, this implies $\m^3=(y^2,xz,yz)(x,y,z)=(y^3)$. Using the relations $xy=zy^2=0$ and induction see that $\m^i=(y^i)$ for all $i\ge 3$, hence $R$ satisfies the desired conclusion.
\end{Case1}

\begin{Case2} $x^2\notin yR_1$.

We have thus $u\in R_1\setminus (\kk x+\kk y)$ and we may assume
$u=z$, hence $z^2\in yR_1$.  We have then $\m^2=(y^2, yz, x^2)$ and
there exist $\alpha'$, $\beta'$, $\gamma\in \kk$ such that
$xz=\alpha' y^2+\beta' yz+\gamma x^2$.

If $\beta'=0=\alpha'$, then $\rank(x)=1$ which is a contradiction.
If $\beta'=0\neq \alpha'$, then $y^2\in xR$, and this is Case 1 with $y$
and $x$ interchanged. Therefore we may assume
$\beta'\neq 0$. Replacing $y$ with $\beta' y$ we can assume that
$\beta'=1$. The  replacement $z\leftarrow \alpha' y+z$ yields $xz+ yz+\gamma x^2=0$ and it does not change the fact that $z^2\in yR_1$.
We have thus $z^2=\alpha y^2+\beta yz$, with $\alpha, \beta\in \kk$.
If $\gamma=\alpha=0$ then $\rank(z)=1$ and we are done by one
of the previously considered cases. We may assume thus $(\alpha, \gamma)\ne (0,0)$. The ring $R$ is thus defined by the relations:
\begin{equation}
\label{NK3}
xy=z^2+\alpha y^2+\beta yz=xz+ yz+\gamma x^2=0\quad\text{with}\quad (\alpha, \gamma)\ne (0,0)\,.
\end{equation}

If $\gamma=0\neq\alpha$ we use the defining relations to see that $xz^2=x^2z=yz^2=y^2z=z^3=y^3=0$, and we see from here that  $\m^i=(x^i)$ for all $i\ge 3$.

Assume now  $\gamma\neq 0$. We obtain  $xz^2=y^2z=z^3+\beta yz^2=\alpha y^3+yz^2=yz^2+\gamma x^2z=x^2z+\gamma x^3=0$  and we derive from here that  $\m^i=(y^i)$ for all $i\ge 3$. \qed
\end{Case2}

\begin{remark}
\label{classify}
The proof of Theorem \ref{exceptional} describes the structure of exceptional rings in any characteristic.  More precisely, we have: If $\kk$ is algebraically closed, $R$ is a non-Koszul standard graded quadratic $\kk$-algebra with $\socle(R)\subseteq \m^2$ and $\dim_{\kk}R_2\le 3$, then $R$ is isomorphic  to one of the rings in Remark \ref{Alessio-exceptional} or to one of the rings defined by the relations \eqref{NK1}, \eqref{NK2} and \eqref{NK3}.
\end{remark}

\begin{lemma}\label{non-K}
If $R$ is a standard graded $\kk$-algebra with Hilbert series  $(1+2t-2t^3)(1-t)^{-1}$, then $R$ is not Koszul.
\end{lemma}

\begin{proof}
Assume $H_R(t)=(1+2t-2t^3)(1-t)^{-1}$.  Let $\Po_{\kk}^R(z)=\sum_{i\ge 0}b_iz^i$ be the Poincar\'e series of $\kk$ over $R$, where $b_i=\dim_{\kk}\Tor_i^R(\kk,\kk)$.

If $R$ is Koszul, then the relation $\Po_R^{\kk}(t)H_R(-t)=1$ holds by \ref{PH}. We have thus
$$
(b_0+b_1t+b_2t^2+...)(1-2t+2t^3)=1+t\,.
$$
The sequence $\{b_i\}$ satisfies $b_0=1$, $b_1-2b_0=1$, $b_2-2b_1=0$ and $b_{i+3}-2b_{i+2}+2b_i=0$ for all $i\ge 0$.
We have thus $b_1=3$, $b_2=6$, $b_3=10$, $b_4=14$, $b_5=16$, $b_6=12$ and  $b_7=-4$. The last equality is a contradiction.
\end{proof}

We are now  ready to prove the Main Theorem, as stated in the
introduction.

\begin{proof}[Proof of Main Theorem.]
Assume $R$ is a standard graded quadratic $\kk$-algebra with $\dim_{\kk}R_2\le 3$.
We first prove the following statement: If $\kk$ is algebraically closed, there  exists a surjective
Golod homomorphism $P\to R$ with $P$ a complete intersection of
codimension $3$.

To  prove this statement, we may assume there are no socle elements
in degree $1$, in view of \ref{reduction-prelim}. If $\dim_{\kk}R_1=3$, then the existence of a Golod homomorphism $P\to R$ from a complete
intersection follows from \cite[Proposition 6.1]{AKM};  one can also see
that $\codim(P)\le 3$. Assume now $\dim_{\kk}R_1\ne 3$.  The statement is  proved then by Theorem
\ref{artinian-case} when $R$ is artinian and  by Theorem
\ref{dim2-thm} when $\dim_{\kk}R_2\le 2$. Assume thus $R$ is not artinian and  $\dim_{\kk}R_2=3$. If $\dim_{\kk}R_1<3$, then $R$ is a polynomial ring, hence it  is a Golod ring.  If $\dim_{\kk}R_1>3$, then the statement is proved by Theorem \ref{last}.

We now prove the equivalence of the three statements in the theorem.  We set $\ov R=R\otimes_{\kk}\ov{\kk}$.

Since the ring
is graded, we know (1)$\implies$(2) by \cite[Proposition 1.13]{HI}. We use \ref{reduction-prelim-Koszul} and Lemma \ref{non-K} to see  (2)$\implies$(3). The implication (2)$\implies$(1) follows from \ref{Golod-ci}.

We prove now the implication (3)$\implies$(2). Assume $R$ is not Koszul, hence $\ov R$ is not Koszul either. Theorem \ref{exceptional} shows that $\ov R$ is a trivial fiber extension of a standard graded $\kk$-algebra with Hilbert series $(1+2t-2t^3)(1-t)^{-1}$.  In particular, we have $e\ge 3$ and \begin{equation}
\label{HR}H_{R}(t)=H_{\ov R}(t)=(1+2t-2t^3)(1-t)^{-1}+(e-3)t\,.
\end{equation}
We prove by induction on $e$ that there exist elements $v_1, \dots, v_{e-3}$ such that  $v_i\in  R_1\cap\socle( R)$ and the ring $T=R/(v_1, \dots, v_{e-3})$ has Hilbert series $(1+2t-2t^3)(1-t)^{-1}$. If $e=3$, then there is nothing to prove. Assume now $e>3$. Since the extension $R\to \ov R$ is flat, we know $\socle(R)\ov R=\socle(\ov R)$. In particular, since $\socle(\ov R)\cap \ov R_1\ne 0$, we see that $\socle(R)\cap R_1\ne 0$. Let $v_1$ be an element in $\socle(R)\cap R_1$ and set $R'=R/(v_1)$.  By \ref{reduction-prelim-Koszul}, we know that $R'$ is not Koszul, since $R$ is not Koszul.  The induction hypothesis applied to $R'$ gives elements  $v_2', \dots, v'_{e-3}$ in $\socle(R')\cap R'_1$ such that the ring $T'=R'/(v_2', \dots, v'_{e-3})$ has Hilbert series $(1+2t-2t^3)(1-t)^{-1}$. Now let $v_2, \dots, v_{e-3}$ denote preimages of $v_2', \dots, v'_{e-3}$ in $R_1$ and set $T=R/(v_1, \dots, v_{e-3})$. Note that $T\cong T'$. In particular, $T$ has Hilbert series $(1+2t-2t^3)(1-t)^{-1}$.  Using the expression for $H_{R}(t)$ in \eqref{HR}, we see that the Hilbert series of the ideal $(v_1, v_2, \dots, v_{e-3})$ is equal to $(e-3)t$, and, in particular, this ideal is annihilated by $R_{\ge 1}$.
\end{proof}

\section{Local rings}
\label{local rings}

We now prove the local version of the Main Theorem. We discuss first some preliminaries.

 In this section, the notation $(R,\m, \kk)$ identifies a local (meaning also commutative and Noetherian) ring $R$ with maximal ideal $\m$ and residue field $\kk=R/\m$.

\begin{bfchunk}{Initial forms.} Let $(R,\m, \kk)$ be a local ring. We denote $\agr R$ the associated graded
ring $\oplus_{i\geq 0}  \m^i/\m^{i+1}$.  Each $a \in R$ has a
natural image in  $\agr R$ which is denoted by $a^*$ and we
call it the {\it initial form} of $a$. More precisely, if $a\neq 0$, then
$a^*=a+\m^{i+1}$, where $i$ is the unique integer such that
$a\in\m^i\setminus \m^{i+1}$; otherwise $a^*=0$.  If $J$ is an ideal of $R$ we denote by
$J^*$ the ideal consisting of the  initial forms of the
elements of $J$.
\end{bfchunk}

\begin{lemma}
\label{Golod when Koszul}
Let $(R,\m,\kk)$ be a local ring and set $A=\agr R$. Assume $A$ is Koszul and there exists a complete intersection of quadrics $P'$ and a surjective Golod  homomorphism of graded $\kk$-algebras $\varphi'\colon P'\to A$.

 There  exists then a local complete intersection ring $P$ with $\codim P=\codim P'$ and a surjective Golod homomorphism of local rings $\varphi\colon P\to \widehat R$.
\end{lemma}
\begin{proof}
We write $\widehat R=Q/I$ with $(Q,\q,\kk)$ a regular local ring and
$I\subseteq \q^2$.  We identify  $A$ with $\agr Q/I^*$. Note that $\agr Q$ is a polynomial ring over $\kk$. Since the homomorphism $\varphi'$ in the hypothesis is surjective, we may assume $P'=\agr Q/J'$, where $J'$ is generated by a regular sequence, $J'\subseteq I^*$ and $\varphi'$ is the canonical map.  Let $f_1^*, \dots, f_d^*$ be a regular sequence generating $J'$, where $f_i\in I$ for all $i$,  and consider the ideal $J=(f_1, \dots, f_d)$ of $Q$. Note the inclusion $J\subseteq I$. By Valabrega and Valla \cite[Proposition 2.1]{VV}, we see that $P=Q/J$ is a complete intersection and $J^*=J$. We have thus $\agr P=P'$ hence the canonical map $\varphi\colon P\to Q/I$ satisfies $\agr \varphi=\varphi'$. We use then Proposition \ref{graded-help} to conclude that $\varphi$ is Golod.
\end{proof}

\begin{lemma}
\label{socle} Let $(R,\m,\kk)$  be a local ring and set $A=\agr R$. Assume the following hold:
\begin{enumerate}[\quad\rm(a)]
\item  There  exists an element $u\in A_1$ such that $A_i=uA_{i-1}$ for all $i\ge 3$.
\item $\dim_{\kk}A_i=\dim_{\kk}A_{i+1}\ne 0$ for all $i\ge 3$.
\end{enumerate}
 Then  the graded vector spaces $\socle(A)$ and $(\socle(R))^*$ coincide in degree 1. In particular, if $\socle(R)\subseteq \m^2$, then $A$ has no socle elements in degree $1$.
\end{lemma}
\begin{proof}
The  inclusion $(\socle(R))^*\subseteq \socle(A)$ is clear.

To  prove the reverse inclusion, let $x_1\in
(\socle(A))_1$ . Complete $x_1$ to a basis
$x_1,x_2,\ldots,x_e$ of $A_1$ such that $u=x_e$. Let $y_1,\ldots,y_e\in\m\smallsetminus \m^2$
such that $y_i^*=x_i$ for all $i$. We have thus $y_1^*y_i^*=0$ for all $i$.

The  hypothesis (a) implies $\m^i=y_e\m^{i-1}+\m^{i+1}$ for all $i\ge 3$. Nakayama's Lemma gives:
\begin{equation}
\label{Nak}
\m^i=y_e\m^{i-1}\quad\text{ for all $i\ge 3$.}
\end{equation}

Assume $y_1y_e\ne 0$. Since $y_1^*y_e^*=0$, there exists $p\ge 3$ such that $y_1y_e\in \m^p$. By \eqref{Nak} there exists then $b\in \m^{p-1}$ such that $y_1y_e=by_e$, hence $y_e(y_1-b)=0$. Note that $(y_1-b)^*=y_1^*=x_1$, since $b\in \m^2$. We may replace thus $y_1$ with $y_1-b$ as to assume $y_1y_e=0$.

Assume now $y_1y_e=0$.  Let $2\leq i\leq e$ and assume $y_1y_i\ne 0$.  There  exists then $p\ge 3$ such that $y_1y_i\in \m^p\smallsetminus \m^{p+1}$. The  multiplication map $\m^p/\m^{p+1}\xrightarrow {y_e}\m^{p+1}/\m^{p+2}$ is surjective by \eqref{Nak}. Hypothesis (b) yields that this map is injective as well. Since $y_e(y_1y_i)=(y_1y_e)y_i=0 $, we conclude that $y_1y_i\in \m^{p+1}$, a contradiction.

We proved thus $y_1y_i=0$ for all $i$, hence $y_1\in \socle(R)$ and thus $x_1\in (\socle(R))^*$.
\end{proof}

\begin{theorem}
\label{localv}Let $(R, \m,\kk)$ be a local ring.    Assume $\dim_{\kk}\m^2/\m^3\le 3$ and $\agr R$ is
quadratic.

There exists then a flat homomorphism of local rings $R\to R'$, where $R'$ is a local ring with maximal ideal $R'\m$ and has the property that there exists a surjective Golod homomorphism $P\to R'$ of local rings, with  $P$ a  complete intersection of codimension at most $3$.  Furthermore, if $R$ is Koszul then it is absolutely Koszul.
\end{theorem}

\begin{proof}
There exists a flat ring homomorhism $R\to R'$, where $(R', \m',\kk')$ is a local ring such that $\m'=R'\m$, $R'$ is complete and  $\kk'$ is algebraically closed. Indeed, one can use \cite[Chapter 0, 10.3.1] {Gro} to construct a flat homomorphism to a local ring with an algebraically closed  field, and then follow with the completion homomorphism.

Note that $\agr{R'}\cong \agr R\otimes_{\kk}\kk'$, $R$ is Koszul if and only if $R'$ is Koszul, and, if $R'$ is absolutely Koszul then $R$ is absolutely Koszul, cf.\,\ref{flat-absKoszul}. Thus, for the part of the proof regarding absolute Koszulness, it suffices to show that if $R'$ is Koszul, then $R'$ is absolutely Koszul.   For the purposes of the proof, we may assume thus $R=R'$, hence $R$ is complete and $\kk$ is algebraically closed.  We write
$R=Q/I$ with $(Q,\q,\kk)$ a regular local ring and $I\subseteq \q^2$.

We prove below the existence of the Golod homomorphism. The  fact that $R$ is absolutely Koszul when $R$ is Koszul follows then from \ref{Golod-ci}. In view of \ref{reduction-prelim}, it suffices to assume that $\socle(R)\subseteq \m^2$. We set $A=\agr R=\agr Q/I^*$. Note that $\agr Q$ is a polynomial ring over $\kk$ in $e$ variables, where $e=\dim_{\kk}\m/\m^2=\dim_{\kk}A_1$.

Assume first that $R$, hence $A$, is not Koszul.  Using Theorem \ref{exceptional} we see that the hypotheses of Lemma \ref{socle} are satisfied and  $\dim_{\kk}A_2=3$ . We apply then Lemma \ref{socle} to conclude that $A$ is not a trivial fiber extension. Since $A$ is not Koszul, Theorem \ref{last} gives that $e<4$, and hence $e=3$. We use then \cite[Proposition 6.1]{AKM} as in the proof of the Main Theorem to see that there exists a Golod homomorphism onto $R$ from a complete intersection of codimension at most $3$.

If $R$, hence $A$,  is Koszul, we use the Main Theorem to see there exists a complete intersection of quadrics $P'$  of codimension at most $3$ and a surjective Golod  homomorphism of graded $\kk$-algebras $\varphi'\colon P'\to A$.  We use then  Lemma \ref{Golod when Koszul}.
\end{proof}

\end{document}